\documentclass[11pt,a4paper,reqno]{amsart}
\usepackage{mathrsfs}
\usepackage{bigints}
\usepackage{amsmath,amsfonts,verbatim}
\usepackage{mathtools}
\usepackage{latexsym}
\usepackage{amssymb,leftidx}
\usepackage{extarrows}
\usepackage{overpic}
\usepackage{color}
\usepackage{epsfig}
\usepackage{subfigure}
\usepackage{tikz}
\usepackage{bbm}
\usepackage{tikz, caption}
\usepackage{physics}
\usepackage[font=small,labelfont=bf]{caption}
 \mathtoolsset{showonlyrefs=true}
\usepackage[colorlinks=true,backref=page]{hyperref}
\hypersetup{urlcolor=blue, citecolor=blue}
\usepackage{geometry}
\usepackage{appendix}
\usepackage{amssymb}
\usepackage{mathrsfs}
\usepackage{amscd}
\usepackage{cite}
\newcommand{\diff}{\mathrm{d}}

\newcommand\R{\mathbb{R}}
\newcommand\C{\mathbb{C}}

\usepackage{graphicx}
\usepackage{float}

\numberwithin{equation}{section}
\newtheorem{proposition}{Proposition}[section]
\newtheorem{definition}{Definition}[section]
\newtheorem{lemma}{Lemma}[section]
\newtheorem{theorem}{Theorem}[section]

\newtheorem{assumption}{Assumption}[section]
\newtheorem{Problem}[theorem]{Problem}
\newtheorem{example}{Example}[section]
\newtheorem{remark}{Remark}[section]

\begin{document}

\title[UNCERTAINTY PRINCIPLE AND UNIQUENESS]{Geometric uncertainty principles for Schr\"odinger evolutions on negatively curved manifolds}
	
	\begin{abstract}
		In this paper, we study the uncertainty principle for  Schr\"odinger equations with a bounded time-independent potentials on certain Cartan-Hadamard manifolds endowed with an asymptotic hyperbolic metric in dimensions $n\geq2$. The classical Hardy uncertainty principle in Euclidean space, as developed in the works of Escauriaza–Kenig–Ponce–Vega (JEMS, 2008; Duke Math.~J., 2010), reveals a rigidity phenomenon for solution $u$ to Schr\"odinger equations: sufficiently strong Gaussian decay at two distinct times yields $u\equiv0$. In this work, we show that a similar rigidity persists in the setting of hyperbolic geometry, despite the absence of translation invariance and Fourier representation. Our approach follows a general strategy of Escauriaza-Kenig-Ponce-Vega, where the underlying geometry brings an essential change. This enables us to establish new Carleman estimates and logarithmic convexity. Unlike the Euclidean setting, the hyperbolic geometry exhibits exponential volume growth and nontrivial geodesic escape at infinity, which fundamentally alters the propagation mechanism of Schr\"odinger evolutions.  
        Based on the newly-built virial identities and an approximation argument, we derive the logarithmic convexity. The main difficulty in proving the logarithmic convexity is the lack of convolution structure on general manifolds. By making use of the exponential map and  Jacobi field, we define a new mollifier  on curved geometry. 
        Meanwhile, to establish the Carleman estimate adapted to hyperbolic space, we introduce a new weight function adapted to the curved manifold. 
       Our results highlight the role of curvature in shaping quantitative uniqueness properties for dispersive equations.
	\end{abstract}
\author{Changxing Miao}
\address{Changxing Miao
		\newline \indent Institute of Applied Physics and Computational Mathematics,
		Beijing, 100088, China.
		\newline\indent
		National Key Laboratory of Computational Physics, Beijing 100088, China}
\email{miao\_changxing@iapcm.ac.cn}
    
	\author{Yilin Song}
\address{Yilin Song
	\newline \indent The Graduate School of China Academy of Engineering Physics,
		Beijing 100088,\ P. R. China}
\email{songyilin21@gscaep.ac.cn}

\author{Ruihan Zhou}
\address{Ruihan Zhou
	\newline \indent Institute of Applied Physics and Computational Mathematics,
	Beijing, 100088, China.}
\email{19210180085@fudan.edu.cn}
	\maketitle
	
	\begin{center}
		\begin{minipage}{130mm}
		\end{minipage}
	\end{center}

\tableofcontents



\section{Introduction}
In this article, we study the unique continuation property for solution to the Cauchy problem
\begin{align}\label{eq1}
    \begin{cases}
        i\partial_tu+\Delta_gu-V(x)u=0,&(x,t)\in M\times\R,\\
        u(x,0)=u_0(x),
    \end{cases}
\end{align}
where $u: M\times I\to\mathbb C$ is the solution and $(M,g)$ is the $n$ dimensional Cartan-Hadamard manifold  endowed with the standard metric $g$. Here, $\Delta_g$ denotes the standard Laplace-Beltrami operator on $M$.

In 1933, Hardy \cite{Hardy} first proved the following classical uncertainty principle. 
\begin{theorem}[Hardy's uncertainty principle]If $f(x)=\mathcal{O}(e^{-\alpha|x|^2})$ and its Fourier transform has Gaussian decay $\widehat{f}(\xi)=\mathcal{O}(e^{-\beta|\xi|^2})$ with $\alpha\beta>\frac{1}{16}$, then $f=0$.
\end{theorem}
Hardy's method relies on techniques from complex analysis, e.g., the Phragm\'en-Lindel\"of principle and entire function theory.
By Fourier transform, the free Schr\"odinger equation can be represented as
\begin{align*}
    e^{it\Delta }f=\frac{1}{(2\pi t)^\frac d2}\int_{\R^d}e^{i\frac{|x-y|^2}{4t}}f(y)\,\dd y=(2\pi t)^{-\frac d2}e^{i\frac{|x|^2}{4t}}\mathcal{F}\big(e^{i\frac{|\cdot|^2}{4t}}f\big)(\frac{x}{2t}).
\end{align*}
Then, by invoking  the uncertainty principle, the dynamical version of Hardy's uncertainty principle can be stated as follows: If $u(x,0)=\mathcal{O}(e^{-\alpha|x|^2})$ and $u(x,T)=\mathcal{O}(e^{-\beta|x|^2})$
 with $\alpha\beta>\frac{1}{16T^2}$, then  $u\equiv0$. For the detailed proof, we refer to a survey \cite{survey}. Bourgain \cite{Bourgain-IMRN} extends the linear uniqueness result to the nonlinear case. He showed that if the solution to a nonlinear Schr\"odinger equation is compactly supported in a nontrivial time interval, then $u\equiv0$.

Even though Hardy gave a proof based on complex analysis, many mathematicians have hoped to provide a proof relying solely on real-analytic arguments. The first attempt was made by Escauriaza-Kenig-Ponce-Vega in \cite{EKPV-CPDE}, who proved the unique continuation property for a linear Schr\"odinger equation with a bounded potential. By introducing a delicate iteration scheme, they proved the sharp Hardy type uncertainty principle for solution to Schr\"odinger equation  in \cite{EKPV-Duke}. Based on this pioneering work, Cowling-Escauriaza-kenig-Ponce-Vega \cite{CEKPV-IUMJ} gave a real-analytic proof of Hardy's result. 


The unique continuation property for the Schr\"odinger equation with certain physically relevant potentials has been of great research interest to many mathematicians, see \cite{EKPV-CPDE,EKPV-JEMS,EKPV-Duke}. 
We state the  unique continuation property  of Schr\"odinger equation in the following theorem.
\begin{theorem}[Unique continuation property, \cite{EKPV-CPDE,EKPV-JEMS,EKPV-Duke}]\label{Th-EKPV}
Let $d\geq3$ and $T>0$ bounded. Assume that $u\in L^\infty([0,1],L^2(\R^d))\cap L^2([0,1],H^1(\R^d))$ is the solution to 
\begin{align*}
    i\partial_tu+\Delta u+V(x,t)u=0.
\end{align*}
Assume that $V\in L_{t,x}^\infty([0,1]\times\R^d)$ and either $V(x,t)=V_1(x)+V_2(x,t)$ with \begin{equation}\sup\limits_{t\in[0,1]}\|e^{\frac{|x|^2}{(\alpha t+\beta(1-t))^2}}V_2\|_{L^\infty(\R^d)}<\infty\label{v2}
\end{equation}or 
$$
\lim_{R\to\infty}\|V\|_{L_t^1([0,1],L^\infty(\R^d\setminus B(0,R)))}<\infty. 
$$
Then 
$$
\big\|e^{\alpha^{-2}|x|^2}u(x,0)\big\|_{L^2(\R^d)}+\big\|e^{\beta^{-2}|x|^2}u(x,1)\big\|_{L^2(\R^d)}<\infty,
$$
then we have $u\equiv0$ if $\alpha\beta<4$.
\end{theorem}
The proof of the above theorem relies on the logarithmic convexity of some quadratic forms and a Carleman estimate. Compared to the original proof of the Hardy uncertainty principle, which relied on the Phragm\'en–Lindel\"of principle, the present proof in Theorem \ref{Th-EKPV} replaces it by the logarithmic convexity. In the following, we summarize the Phragm\'en–Lindel\"of principle.
Let $ f(z) $ be a continuous function  on the closed vertical strip $ a \le \operatorname{Re}(z) \le b $. Moreover, we assume that $f $ is  holomorphic in the interior $ a < \operatorname{Re}(z) < b $, and bounded on the boundary. For each $ x \in [a,b] $, define
\[
M(x) = \sup_{y \in \mathbb{R}} |f(x+iy)|.
\]
Then $ \log M(x)$  is a strictly convex function on the closed strip.  In the dynamical situation, we consider a quadratic form $H(t)=\|e^{\gamma|x|^2}u\|_{L_x^2}$ where $u$ is a solution to the linear Schr\"odinger equation with a bounded  potential
\begin{align}\label{NLS}i\partial_tu+\Delta u+V(x)u=0.\end{align}
The boundary values of the strip are $H(0)$ and $H(1)$ which are bounded by our assumption.  Using the equation \eqref{NLS}, one can show that 
$$\partial_t^2(H(t)-G(t))\geq -C_1,$$
where $G(t)$ is a function with a bounded gradient and $C_1>0$ is a constant. Therefore, the logarithmic convexity for the solution $u$ to \eqref{NLS}  can be viewed as a generalized version of Phragm\'en-Lindel\"of's principle. 
 The significance of this logarithmic convexity result stems from several reasons. The most important one is that this method can be applied to deal with the case where $V\in L^\infty$, which is very difficult to treat using the techniques from complex analysis. 
For the study of Hardy's uncertainty principle for  Schr\"odinger equations with various potentials, we refer to  \cite{BFGRV-JFA,Cassano-Fanelli,Indiana,Jensen}. 


Compared with the classical result in \cite{EKPV-JEMS}, one should believe that the  Hardy-type uncertainty principle holds for  the variable-coefficient Schr\"odinger operator $\operatorname{div}(A\nabla\cdot)$ provided suitable structural conditions are imposed on $A$. We refer to the work of  Federico--Li--Yu \cite{Yu-CCM}, in which the Hardy uncertainty principle is established under the non-trapping condition on the Hamiltonian generated by $A$. 
A complementary but fundamentally different question arises when the geometry is encoded in the underlying manifold rather than in the Hamiltonian. In our situation, the propagation is governed by the Laplace–Beltrami operator, and wave packets evolve along geodesic trajectories influenced by curvature. A natural problem arises:
\begin{Problem}
How does the underlying geometry influence the  uncertainty principle?
\end{Problem} 
From this perspective,  Anderson \cite{Anderson1,Anderson2} and Krist\'aly \cite{Kristaly} establish static uncertainty principles on negatively curved manifolds, revealing that curvature influences spatial localization inequalities in a fundamental way. However, their results do not address the dynamical setting of dispersive equations, where geometry interacts with time-dependent Schr\"odinger propagation.

In the present work, we provide an affirmative answer to the problem of determining which classes of manifolds can support a Hardy-type uncertainty principle. For general manifolds with non-positive curvature, we establish a sufficient condition that guarantees strong uniqueness: if the solution exhibits almost quadratic exponential decay at times $t=0$ and $t=1$, then $u\equiv 0$. In the case of hyperbolic space, we are able to further relax the required decay rate from almost quadratic down to quadratic. Our results reveal how geometric curvature fundamentally shapes the uncertainty principle phenomenon. In addition, our results show that the Gaussian localization is still  rigid in hyperbolic geometry. Before presenting our theorems, we first review the well-posedness of the Schr\"odinger equation on curved spaces from the PDE perspective.  
Banica \cite{Banica-CPDE} proved Strichartz estimates by deriving an explicit kernel representation of the Schr\"odinger propagator. As an application, the local well-posedness for subcritical NLS on hyperbolic space follows from the Strichartz estimates and fixed point argument as in \cite{Anker-AIHPC,Banica-CPDE,MR-book,MZZ-book,Shen,Wilson-Yu}.

\subsection{Main results}
Our main results stated below,  provide the first uniqueness properties for dispersive equations on non-flat manifold.
\begin{theorem}\label{thm1}
    Suppose that $u\in C([0,1],L^{2}(\mathbb{H}^{n}))\cap L^{2}([0,1],H^{1}(\mathbb{H}^{n}))$ is a solution to \eqref{eq1}.
    Assume $V(x)\in L^{\infty}(\Bbb H^n)$. If $e^{\gamma\rho^{2}}u_{0}\in L^{2}(\Bbb H^n)$ and $e^{\gamma\rho^{2}}u(x,1)\in L^{2}(\Bbb H^n)$ for $\gamma>\frac12$, then $u\equiv0$ where $\rho=d(x,\mathbf{0})$ and $\mathbf{0}$ is the origin of the hyperbolic space. 
\end{theorem}

Our second result concerns the unique continuation property for the heat equation
\begin{equation}\label{parabolic}
        \begin{cases}
        \partial_{t}u=\Delta_gu+V(x)u,&(x,t)\in M\times[0,1],\\ u(x,0)=u_{0}.
        \end{cases}
    \end{equation}
\begin{theorem}\label{thm2}
       Suppose that $u\in C([0,1],L^{2}(\mathbb{H}^{n}))\cap L^{2}([0,1],H^{1}(\mathbb{H}^{n}))$ is a solution to \eqref{parabolic}.
    Assume that the potential $V\in L^{\infty}(\Bbb H^n)$. If $e^{\gamma\rho^{2}}u_{0}\in L^{2}(\Bbb H^n)$ and $e^{\gamma\rho^{2}}u(x,1)\in L^{2}(\Bbb H^n)$ for $\gamma>\frac12$, then $u\equiv0$ where $\rho=d(x,\mathbf{0})$ and $\mathbf{0}$ is the origin.
\end{theorem}
\begin{remark}
    Our result can be generalized to the Schr\"odinger equation with time-dependent potential $V\in L^\infty([0,1]\times\Bbb H^n)$ and an additional condition \eqref{v2}.
\end{remark}
Next,  we aim to extend the Hardy uncertainty principle to the more general setting in which the metric can be understood as a perturbation of the hyperbolic metric. Such manifolds have a close connection with the classical asymptotic hyperbolic manifolds.
\begin{assumption}\label{assum}
	Let $(M,g)$ be an $n$-dimensional Cartan-Hadamard manifold. Assume that the metric $g$ can be expressed in the following way:
	\begin{align}
	g=d\rho^2+\sum_{j,k=1}^{n-1}(\sinh \rho)^2\Upsilon_{jk}(\rho,\theta)d\theta^j d\theta^k.
	\end{align}
Here, $\Upsilon=h+\Lambda$ is the perturbation of standard spherical metric $h$. In particular, we suppose that the perturbation satisfies the polynomial decay rate
\begin{align}
\big\|\partial_{\rho}^j \Lambda\big\|_{L^\infty}\leq C\rho^{-m-j},\,\,\mbox{as }\rho\to\infty.
\end{align}

\end{assumption}
\begin{remark}
If $\Lambda=0$, then it degenerates to the classical hyperbolic space. If $\Lambda\neq0$, the sectional curvature of $M$ converges to $-1$ as $\rho\to\infty$. For the proof, see Appendix A.
\end{remark}
\begin{example}If we take $\Lambda=\langle\rho\rangle^{-m}\Phi(\theta)h$, where $\Phi(\theta)$ is a smooth function defined on $\mathbb{S}^{n-1}$, then it can be viewed as a perturbed hyperbolic metric satisfying Assumption \ref{assum}.
\end{example}
Now, we can state the unique continuation property for the Schr\"odinger equation on a Riemannian manifold under Assumption \ref{assum}.
\begin{theorem}\label{thm3}
	Let $u\in C([0,1],L^2(M))\cap L_t^2([0,1],H^1(M)) $ be a solution to \eqref{eq1} with a bounded potential $V\in L_{x}^\infty( M)$. For arbitrary fixed $\sigma>0$ such that 
	\begin{align*}
		e^{\sigma\rho^{2}\log\rho}u(x,0)\in L^2(M),\,\,		e^{\sigma\rho^{2}\log\rho}u(x,1)\in L^2(M),
	\end{align*}
	then $u\equiv0$.  
\end{theorem}

\begin{remark}
	Our result is an improvement of \cite{EKPV-CPDE} even in the flat geometry. Compared to  the pioneering work of \cite{LGMN}, our result can be understood as a weaker version of Landis's conjecture of the dispersive equation. 
\end{remark}

\subsection{Motivations and outline of the proofs}
Our proof will follow the general strategy developed in Escauriaza-Kenig-Ponce-Vega \cite{EKPV-CPDE, EKPV-JEMS}. Now, we summarize the main novelties of our results. 

 The unique continuation with almost quadratic exponential rate obtained in Theorem \ref{thm3} optimizes the well-known result in \cite{EKPV-CPDE} where the decay rate is $e^{-\sigma\rho^{2+\varepsilon}}$ with some $\sigma>0$. Notice that this result is valid for a broad class of negatively curved manifolds that may not be rotationally symmetric.  Compared to the rotationally symmetric manifold, the lack of rotational symmetry for such type manifolds exhibits the extra difficulty. For the Hardy uncertainty principle, our results in Theorem \ref{thm1} and Theorem \ref{thm2} match the classical result of \cite{EKPV-JEMS}. It is worth  emphasizing that even though hyperbolic space is also a rotationally symmetric manifold, the influence of curvature brings new challenges in establishing the Carleman estimates. Our analysis shows that the uncertainty principle is not tied to Fourier duality, but is instead a manifestation of dispersive propagation rigidity governed by the underlying geometry. 

Our arguments proceed roughly as follows.
\begin{enumerate}
\item \textit{A weighted $L^2$ estimate.} 
In this step, we establish an  estimate within the weighted norm $\|e^{\gamma\rho^2} u(t,x)\|_{L^2_x}$. For dispersive equations, even if we impose exponential decay at times $t=0$ and $t=1$, we do not know whether the exponential decay persists at the intermediate times $t\in(0,1)$. To overcome this, we introduce the artificial dissipation structure 
\begin{equation}\label{dis}
    \partial_{t}u=(a+bi)\Big(\Delta_{g}u+V(x)u+F(x,t)\Big),
\end{equation}
where $a>0$, $b\in\Bbb R$ and $F$ is an inhomogeneous term. By a conjugate transformation $v=e^{\varphi}u$, $v$ satisfies
\begin{equation}
  \partial_t v=(\mathcal{S}+\mathcal{A})v+(a+ib)\big(V(x)v+e^{\varphi}F\big),
\end{equation}
where $\mathcal{S}$ and $\mathcal{A}$ are the symmetric part and antisymmetric part, see  \eqref{formula sym} and \eqref{formula anti} below.
To establish the weighted estimate formally, we first give the calculation of the commutator
\begin{align*}
    \int_M(\mathcal{S}_tf+[\mathcal S,\mathcal A]f)\bar f\,\dd{\rm Vol}.
\end{align*}
By virtue of the parabolic structure, we obtain the $L^2$ weighted estimate formally. Next, it suffices to find the proper approximating arguments to make the proof rigorously. 
For the flat geometry, the authors of \cite{EKPV-CPDE} achieved the approximation by introducing a convolution mollifier. However, the lack of the Fourier transform makes it impossible to use such a mollifier directly. To overcome this difficulty, we define a new mollifier using the exponential map. The main goal is to recover an estimate similar to that in the flat geometry.  To achieve this, we introduce the Jacobi equation and the expansion of Jacobi fields.  
 Then taking $a\to0$, we recover the estimate for the Schr\"odinger equation.
\item {\textit {Logarithmic convexity for the Ginzburg-Landau flow.}}
In our setting, the logarithmic convexity can be viewed as a replacement of the Phragm\'en-Lindel\"of principle in Hardy's method. To show this convexity, we adapt the abstract framework of Escauriaza-Kenig-Ponce-Vega, which is a dispersive version of Almgren's frequency function \cite{Almgren}. The remaining parts consist in proving a coercive estimate for the commutator term, which can be rewritten as
\begin{equation}\label{lowerbound comu}
    \int_{M}(\mathcal{S}_{t}f+[\mathcal{S},\mathcal{A}]f)\bar{f}\,\dd {\rm Vol}\geq-\mathfrak{C}_{n},
\end{equation}More precisely, taking $\varphi=\gamma\rho^2$ and applying Lemma \ref{technical-lemma}, we need to control two geometric quantities in order to obtain \eqref{lowerbound comu}. By the comparison theorem, $\operatorname{Hess}(\varphi)(\nabla_gf,\nabla_g\bar f)$ can be absorbed by the gradient term. It remains to bound $\Delta_g^2(\rho^2)$. In hyperbolic space, using the geodesic polar coordinates, we can give the proper bound on this quantity. However,  the lack of symmetry of the asymptotic hyperbolic metric makes the calculation extremely difficult which brings us the first challenge. To overcome this difficulty, we make full use of the fundamental equation  in the Riemannian submanifolds (e.g. Gauss equation, Codazzi-Mainardi equation) and perform a rigorous calculation of several curvature tensors.   Under Assumption \ref{assum}, the metric can be viewed as a perturbation of that in standard hyperbolic space. Motivated by this observation,  we decompose the geometric quantities into a part corresponding to the standard hyperbolic metric and the perturbation part. Using the decay of the perturbation metric, we finally bound it by $|\Delta_g^2(\rho^2)|\lesssim C(n)+\mathcal{O}(\rho^{-m-1})$ where $m$ is associated with the decay rate of the perturbed metric. In addition, this estimate is also fundamental in the blow up theory of nonlinear Schr\"odinger equation on curved space which has its own interests. The details will be given in Appendix  $B$. 
    In Appendix $A$, we provide a complete and explicit calculation of all components of the Riemannian curvature tensor, the Ricci tensor, and the scalar curvature.
    

\item \textit{Carleman estimates for the evolution equation on $\Bbb H^n$.}
    As mentioned earlier,  Carleman estimates are widely applied to unique continuation property for different types of equations, where the difficulty falls on choosing suitable weight functions to accommodate different geometries and different structures of the equations. 
    
    First, we focus on the hyperbolic case. The primary difficulty is brought by the non-flat geometry. In Euclidean space, Escauriaza-Kenig-Ponce-Vega \cite{EKPV-JEMS} developed the Carleman estimates for the Schr\"odinger operator. Namely, they consider the following quantity
\begin{align*}
    \big\|e^{\mu|x+a(t)e_1|^2}u\|_{L^2}.
\end{align*}
In the non-Euclidean setting, $x+a(t)e_1$ is not available since a general manifold is not an affine space. To develop the Carleman estimates adapted to the curved manifold, we introduce  the exponential mapping to characterize the relative position of a fixed point and a moving curve. Motivated by this idea, we consider the following new weight function adapted to the hyperbolic space
\begin{align*}
    \varphi(x,t)=\mu d(x,P(t))^2-\frac{(1+\varepsilon)R^2t(1-t)}{16\mu},
\end{align*}
where $\rho(t,x):=d(x,P(t))$ is the geodesic distance between $x$ and a parametrized curve $P(t)=\exp_\mathbf{0}(-Rt(1-t)\mathbf e_1)$. Notice that $P(t)$ is well-defined by the Hopf-Rinow theorem. 
The choice of $\varphi$ corresponds to the center of mass moving along the curve $P(t)$. In particular,  it reduces to $\mu d(x,\mathbf{0})^2$ at $t=0$ and $t=1$. For $t\in(0,1)$ the center of mass leaves the base point $x=\mathbf{0}$. On the one hand, the weight function $\varphi$ is selected to be consistent with the bound $\|e^{\gamma\rho^2}u(x,0)\|^2_{L^2}+\|e^{\gamma\rho^2}u(x,1)\|^2_{L^2}\leq C$, where $\rho=d(x,\mathbf{0})$. 
On the other hand, such a choice of $\varphi$ ensures that the center of mass does not move far away from the base point $x=\mathbf{0}$. In other word, the center of mass goes around and returns. This is closely connected with the almost finite speed of propagation for solutions of the Schr\"odinger equation.

The main ingredient in obtaining the desired Carleman estimate is to obtain the positive lower bound in the virial identity. To achieve this goal, we need to deduce some positive terms via completing the square. To compute the time derivatives of the weight function, it requires the derivation of $\partial_{t}^2\rho$.   While the first derivative 
$\partial_t\rho$ enjoys similar structure as its Euclidean counterpart, the second derivative  $\coth\rho\,(|\dot{P}|_g^2-\rho_t^2)+\langle\ddot{P}(t),\nabla_{g}\rho\rangle_{g}$ involves a curvature correction term $\coth\rho$. To handle this additional term, we will apply radial and tangential decomposition to the parallel vector $\mathbf{e}_1$. As a result, the desired weighted \(L^2\) estimate follows provided that $R>4\mu\epsilon^{-1/2}\mathfrak{C}_n$. 
\item \textit{Proof of Hardy's uncertainty principle.}
Using the Carleman estimate and logarithmic convexity inequality established above, we prove the uniqueness of solution provided $\gamma>\frac{1}{2}$.
The strategy for proving the uniqueness relies on the lower  and upper bounds for the weight function $\varphi(x,t)$ in different regions.  Since the Carleman estimate  is only valid for functions with compact support, we need to introduce the space-time cutoff solution $h=\theta_M(x)\eta_{R}(t)u$, where $u$ is the solution of the Schr\"odinger equation, $\theta_M(x)$ and  $\eta_{R}(t)$ are cutoff functions. A direct calculation  gives 
$$
  \quad\quad \qquad \partial_t h-i\Delta_g h=iVh+\theta_{M}(x)\eta_{R}^{\prime}(t)u(x,t)-i\Big(2\langle\nabla_{g}\theta_M,\nabla_g u\rangle_g+u\Delta_{g}(\theta_M(x))\Big)\eta_{R}(t).
    $$
    Now, we estimate $\|(\partial_t-i\Delta_g)h\|_{L^2_{t,x}}$ term by term.
    Using boundedness of $V$, we can absorb $\|e^{\varphi}Vh\|_{L^{2}_{t,x}}$ on the right hand side of the inequality into the left hand side term $\frac{R}{4}\sqrt{\frac{\varepsilon}{\mu}}\|e^{\varphi}h\|_{L^{2}_{t,x}}$.
Furthermore, utilizing the support of cutoff functions and $\gamma>\frac{1}{2}$, we then get 
\begin{equation}
    \frac{R}{8}\sqrt{\frac{\varepsilon}{\mu}}e^{\mathcal{C}_{\mu,\varepsilon}R^{2}}\|u\|_{L^{2}(B_{\varepsilon(1-\varepsilon)^{2}\frac{R}{4}}\times[\frac{1-\varepsilon}{2},\frac{1+\varepsilon}{2}])}\leq CRe^{\frac{\gamma}{\varepsilon}}.
\end{equation}
Let $R\to\infty$, then we get $u=0$ in $[\frac{1-\varepsilon}{2},\frac{1+\varepsilon}{2}]$. A standard argument then gives $u\equiv0$.
\item \textit{Unique continuation under Assumption \ref{assum}.}
In order to establish the unique continuation property under almost quadratic exponential decay, we adapt the strategy of \cite{EKPV-CPDE} by deriving the Carleman estimate with the form 
\begin{align*}
		&\frac{\mu}{R^{2}}\int_{0}^{1}\int_{M}|\nabla_{g} f|_{g}^{2}\,\dd{\rm Vol}\dd t+\frac{\mu^{3}}{R^{6}}\int_{0}^{1}\int_{M}|\rho f|^{2}\,\dd{\rm Vol}\dd t\\\leq&\int_{0}^{1}\int_{M}\Big|e^{\mu\frac{\rho^2}{R^2}+\mu^{\mathcal{Q}(\ell,R)}\varphi(t)}(\partial_{t}-i\Delta_{g})(e^{-\mu\frac{\rho^2}{R^2}-\mu^{\mathcal{Q}(\ell,R)}\varphi(t)}f)\Big|^2\,\dd{\rm Vol}\dd t,
	\end{align*}
    where $f\in C_0^\infty(M\backslash B_{\rho_0})$ and $\mathcal{Q}(\ell,R)=3-\frac{6\log R}{2\log R+\log\frac{\log R}{\ell}}$ for any fixed $\ell\in\mathbb{N}$. Compared with the weight function chosen in \cite{EKPV-CPDE}, we adopt a weight function that separates space and time variables. 
    
We proceed by a contradiction argument. Assuming that $u\neq0$ and applying the Carleman estimate we establish a lower bound of the truncated energy
\begin{align}\label{goal estimate1}
\delta(R)=\int_{\frac{1}{8}}^{{\frac{7}{8}}}\int_{B_{R}\backslash B_{R-1}}\big(|u|^{2}+|\nabla_g u|_{g}^{2}\big)\,\dd{\rm Vol}\dd t\geq C_1 e^{-\frac{C_0}{\ell}R^2\log R}.
\end{align}
On the other hand, by introducing a new integral transform, we convert the logarithmic convexity of $\|e^{\rho^2}u\|_{L^2_x}$ into $\|e^{\rho^2\log\rho}u\|_{L^{2}_{x}}$. As a consequence, we have the upper bound $$\delta(R)\leq e^{-\sigma R^2\log R},$$ which is a contradiction.

 
\end{enumerate}

\section{Preliminaries}
In this section, we briefly recall the definition and some  properties  associated to operators and functions on hyperbolic space. 
We define 
\begin{gather}
    \big\|f\big\|_{L_t^q(I,L^r(M))}=\Big[\int_I\Big(\int_{M}\big|f(t,x)\big|^r\,\dd{\rm Vol}\Big)^\frac{q}{r}\,\dd t\Big]^\frac{1}{q},
\end{gather}
where $I\subset\mathbb R$ is a time interval.
\subsection{Hyperbolic geometry}
For integers $n \ge 2$, we consider the Minkowski space $\mathbb{R}^{n+1}$ with the standard Minkowski metric
\begin{align*}
- \left( \mathrm{d} x_0 \right)^2 +  \left( \mathrm{d} x_1 \right)^2 + \cdots +  \left( \mathrm{d} x_n \right)^2,
\end{align*}
and define the bilinear form on $\mathbb{R}^{n+1} \times \mathbb{R}^{n+1}$,
\begin{align*}
\langle x,y\rangle_{\Bbb H^n} = x_0 y_0 - x_1 y_1 - \cdots - x_n y_n
\end{align*}
The hyperbolic space $\mathbb{H}^n$ is defined as
\begin{align*}
\mathbb{H}^n= \left\{ x \in \mathbb{R}^{n+1}:\langle x, x\rangle_{\Bbb H^n} = 1 \text{ and } x_0 > 0 \right\}.
\end{align*}

Let $\mathbf{0}= (1, 0, \cdots, 0 )$ denote the origin of $\mathbb{H}^n$. The Minkowski metric on $\mathbb{R}^{n+1}$ induces a Riemannian metric $g$ on $\mathbb{H}^n$, with covariant derivative $D$ and induced measure $\mathrm{d} \mu$. Under the hyperbolic metric $g$, the geodesic distant function can be written as 
\begin{equation*}
    d_{\mathbb{H}}(x, y) = \operatorname{arccosh}( -\langle x,y\rangle_{\Bbb H^n} ) = \ln \left( -\langle x,y\rangle_{\Bbb H^n} + \sqrt{\langle x,y\rangle_{\Bbb H^n}^2 - 1} \right).
\end{equation*}

Hyperbolic space is a special case of symmetric space where the group action is given by $ \operatorname{SO}(n, 1)/\operatorname{SO}(n)$. Here, $\operatorname{SO}(n)$ is the connected Lie group.


The Riemannian metric $g$ of the hyperbolic space in geodesic polar coordinate system is expressed by the line element $\diff s^2$:
\begin{equation}
    \diff s^2 = \diff r^2 + \sinh^2(r) \, \diff \Omega_{n-1}^2
\end{equation}
where $\diff \Omega_{n-1}^2$ is the metric on the standard unit sphere $\mathbb{S}^{n-1}$.
The Laplace-Beltrami operator $\Delta_{g}$ is defined as $\Delta_{g}f = \text{div}(\text{grad} f)$.  On a Riemannian manifold, we have
\begin{align*}
  \Delta_gf=g^{\alpha\beta}\nabla_\alpha\nabla_\beta f=\sum_{j,k=1}^{n}\frac{1}{\sqrt{det(g)}}\partial_j(\sqrt{det(g)}g^{jk}\partial_kf).
\end{align*}
For a special case, i.e. $\Bbb H^n$, we have the following expansion
\begin{equation}
    \Delta_{\mathbb{H}^n} = \frac{\partial^2}{\partial r^2} + (n-1)\coth(r) \frac{\partial}{\partial r} + \frac{1}{\sinh^2(r)} \Delta_{\mathbb{S}^{n-1}}.
    \end{equation}

	
\subsection{Some basic properties of distant functions on manifold}
Let $x\in M$ be a fixed point in the $n$-dimensional Cartan-Hadamard manifold  and $P(t)$ be a smooth curve in $M$ parametrized by $t\in(-1,1)$. Next, we define the geodesic distant function by $\rho:=d(x,P(t))$ where $g$ is the canonical hyperbolic metric. Without loss of generality, we assume that $x\neq P(t)$ to ensure the smoothness of distant function. 
\begin{lemma}\label{distant-1}
    Let $\rho$ be the geodesic distant function defined as above, then it holds
    \begin{align*}
        \frac{\partial\rho}{\partial t}=\Big\langle\dot P (t),\nabla_yd(x,y)\big|_{y=P(t)}\Big\rangle_{g},
    \end{align*}
    for $x\neq P(t)$.
\end{lemma}
\begin{proof}

For each $t$, there exists an unique minimizing geodesic between $x$ and $P(t)$ and we denote it by $\gamma_t(s):[0,1]\to M$. Further, we define a one-parameter group by
\begin{align*}
    \Gamma(s,t):=\gamma_t(s),
    \,\,\,\,s\in[0,1],\,\,\,t\in(-1,1).
\end{align*}
Moreover, it satisfies the following boundary conditions
\begin{align}\label{boundary}
\Gamma(0,t)=x,\,\,\Gamma(1,t)=P(t).
\end{align}
By the definition of arc length $L(t)$ and the first variation formula, we have
\begin{align*}
    \frac{\dd L}{\dd t} = \left[ \big\langle \tfrac{\partial \Gamma}{\partial t}, \,\tfrac{\partial \Gamma}{\partial s} \right\rangle_{g} |\tfrac{\partial \Gamma}{\partial s}|^{-1}_{g} \big]_{s=0}^{s=1} - \int_{0}^{1}\big\langle\tfrac{\partial\Gamma}{\partial t},|\tfrac{\partial\Gamma}{\partial s}|^{-1}_{g}\nabla_{\frac{\partial}{\partial s}}\tfrac{\partial\Gamma}{\partial s}\big\rangle_{g}\,\dd s.
    \end{align*}
    Since $\gamma_t$ is along the geodesic direction, we have $\nabla_{\frac{\partial}{\partial s}}\frac{\partial\Gamma}{\partial s}=0$. This implies that the second term vanishes. Now, it remains to calculate the first term, which contains two boundary terms. For convenience, we denote by $V=\frac{\partial\Gamma}{\partial t}$ and $T=\big|\frac{\partial\Gamma}{\partial s}\big|_{g}^{-1}\frac{\partial\Gamma}{\partial s}$. From \eqref{boundary}, we know that $V(0,t)=0$  and $V(1,t)=\dot P(t)$. Consequently, we have
    \begin{align*}
        \frac{\partial\rho}{\partial t}=\big\langle\dot P(t),\,T(1)\big\rangle_{g}=\big\langle\dot P(t),\,\nabla_yd_g(x,y)\big|_{y=P(t)}\big\rangle_{g},
    \end{align*}
    where we use the fact that the unit tangent vector at $y=P(t)$ of the minimizing geodesic from $x$ to 
$y$ is equal to the gradient of the distance function evaluated at 
$y=P(t)$. 
\end{proof}
Next, we will provide the calculation of $\partial_{tt}^2\rho$ on hyperbolic space $\Bbb H^n$.
\begin{lemma}\label{distant-2}
    Let $\rho$ be defined as above, then it holds
    \begin{align*}
        \frac{\partial^2\rho}{\partial t^2}=\coth\rho\big(\big|\dot P\big|_{g}^2-\rho_t^2\big)+\big\langle\ddot P(t),\nabla\rho\big\rangle_{g}
    \end{align*}
    for $x\neq P(t)$.
\end{lemma}
\begin{proof}
   Let $\gamma_t:[0,1]\to\Bbb H^n$ be an unique minimizing geodesic between $x$ and $P(t)$. For such geodesic, we can define a smooth map by
   \begin{gather*}
\Gamma(s,t)=\gamma_t(s),\quad\,s\in[0,1],\,\,t\in[-1,1],\\
\Gamma(0,t)=x,\quad\,\Gamma(1,t)=P(t).
   \end{gather*}
   For convenience, we denote the tangential component of geodesic  by $T=\frac{\partial\Gamma}{\partial s}$ and variation vector field by $V=\frac{\partial \Gamma}{\partial t}$. Also, we denote $U=|T|_{g}^{-1}T$ by the unit tangent vector. From Lemma \ref{distant-1}, we have
   \begin{align*}
       \frac{\partial\rho}{\partial t}=\langle V,U\rangle_{g}\big|_{s=1}.
   \end{align*}
   By differentiating again and using Leibniz's rule, we have
   \begin{align*}
       \frac{\partial^2\rho}{\partial t^2}=\frac{\partial}{\partial t}\big(\langle V,U\rangle_{g}\big|_{s=1}\big)=\big\langle\nabla_VV,U\rangle_{g}\big|_{s=1}+\langle V,\nabla_VU\rangle_{g}\big|_{s=1}.
   \end{align*}
   By making use of the boundary condition at $s=1$, the first term turns to $\langle \ddot  P,\nabla_yd_g(x,y)|_{y=P(t)}\rangle_{g}$. The second term is more complicated and we need to use Jacobi field to give the further calculation. Since $V$ can commute with $T$, we have
   \begin{align*}
       \nabla_VU&=\nabla_V(T|T|_{g}^{-1})=\frac{\nabla_VT}{|T|_{g}}-\frac{T}{|T|_{g}^{2}}\partial_t|T|_{g}\\
       &=\frac{1}{|T|_g}\Big(\nabla_TV-\frac{\langle\nabla_TV,T\rangle_{g} T}{|T|_g^2}\Big)=\frac{1}{|T|_g}\big(\nabla_TV\big)^\perp.
   \end{align*}
   Then, it holds
   \begin{align*}
       \langle V,\nabla_VU\rangle_{g}=\frac{1}{|T|_g}\big\langle V^\perp,(\nabla_TV)^\perp\big\rangle_g,
   \end{align*}
   where the parallel component vanishes. Recall that  $\gamma_t$ is along the geodesic direction for every $t$, then the vector field $V$ satisfies the Jacobi equation
   \begin{align*}
       \nabla_T\nabla_T V+R(V,T)T=0,
   \end{align*}
   where $R(X,Y)$ is a curvature operator with the form
   \begin{align*}
       R(X,Y)Z=\nabla_X\nabla_YZ-\nabla_Y\nabla_XZ-\nabla_{[X,Y]}Z.
   \end{align*}
   In this situation, the commutator $[V,T]=0$ which will simplify the calculation. Notice that the perpendicular component $V^\perp$ also satisfies the Jacobi equation,
   \begin{align*}
       \nabla_T\nabla_TV^\perp+R(V^\perp,T)T=0.
   \end{align*}
   Since the tangent vector $T$ is not of unit length, we introduce a rescaling of the arc length parameter. Let $\sigma=\rho s$ be an arc-length parameter with $s\in[0,1]$, then $|T|_g=\rho$. Denote $\tilde{V}^\perp(\sigma)=V^\perp(s)$, then it holds $\nabla_T=\rho\nabla_U$ where $U$ is the unit tangent vector. Then the Jacobi equation becomes
   \begin{align*}
\nabla_U\nabla_U\tilde V^\perp+R(\tilde V^\perp,U)U=0.
   \end{align*}
   By the direct calculation, we have
   \begin{align*}
       R(\tilde V^\perp,U)U=K\big(\langle U,U\rangle_{g}\tilde V^\perp-\langle\tilde V^\perp,U\rangle_{g} U\big)=-\tilde V^\perp.
   \end{align*}
   Then we obtain the following ODE:
   \begin{align*}
       \big(\tilde V^\perp\big)^{\prime\prime}-\tilde V^\perp=0.
   \end{align*}
   By the standard argument and the boundary conditions, we have
   \begin{align*}
       \tilde V^\perp(\sigma)=\frac{\sinh\sigma}{\sinh\rho}\tilde V^\perp(\rho).
   \end{align*}
   Taking the derivatives along the tangential direction, we have
   \begin{align*}
       \nabla_U\tilde V^\perp(\rho)=\frac{d}{d\sigma}\Big(\frac{\sinh\sigma}{\sinh\rho}\tilde V^\perp(\rho)\Big)\Big|_{\sigma=\rho}=\coth\rho \tilde V^\perp(\rho).
   \end{align*}
   Hence, we obtain
   \begin{align*}
       \big\langle\nabla_U\tilde V^\perp,\tilde V^\perp\big\rangle=\coth\rho|\dot P^\perp|_{g}^2=\coth\rho\big(\big|\dot P\big|_g^2-\rho_t^2\big),
   \end{align*}
   which complete the proof of  Lemma \ref{distant-2}.
\end{proof}
For the convenience, we list some  properties associated with the distant function on hyperbolic space $\Bbb H^n$.\begin{lemma}\label{Hessian-hyperbolic}
Let $\Bbb H^n$ be a $n$-dimensional hyperbolic space and $\rho=d(x,\mathbf{0})$ denotes the distant between $x\in\Bbb H^n$ and $\mathbf{0}$. Then we have
\begin{gather}
    \operatorname{Hess}\rho=\coth\rho(g-\dd\rho^2),\\
    \Delta_g\rho=(n-1)\coth\rho, \\
     \Delta_g^2\rho=(n-1)(3-n)\operatorname{csch}^{2}\rho\coth\rho,\\
    \operatorname{Ric}(\nabla_g\rho,\nabla_g\rho)=-(n-1),\\
|\nabla^{2}\rho|^{2}=(n-1)\coth^{2}\rho.
\end{gather}
\end{lemma}
\begin{remark}\label{bi-laplacian}
By the Leibniz rule, we have 
\begin{equation}\label{biharmonic of geodesic}
    \Delta_g^{2}(\rho^{2})=2(n-1)\Big[(n-1)+(n-3)(1-\rho\coth\rho)\operatorname{csch}^{2}\rho\Big].
\end{equation}
Using the above lemma, we have
\begin{align}\label{the value of biharmonic}
    \Delta_g^2(\rho^2)\in\begin{cases}
    (2,\frac{8}{3}],&n=2,\\\quad\quad8,&n=3,\\(\frac{4n(n-1)}{3},2(n-1)^{2}),&n\geq4,
    \end{cases}
\end{align}
which implies that there exists a positive constant  $\mathfrak{C}_{n}>0$ such that $\Big|\Delta_{g}^{2}(\rho^{2})\Big|\leq\mathfrak{C}_{n}$.    
\end{remark}

Next, for the general complete Riemannian manifold satisfying Assumption \ref{assum}, we claim the following asymptotic formula holds. We provide the detailed proof in Appendix $B$.
\begin{proposition}\label{biharmonic-2}
	Let $M$ be a $n$-dimensional Riemannian manifold satisfying Assumption \ref{assum}, it holds
	\begin{align*}
		|\Delta_g^2(\rho^2)|\leq \mathfrak{C}_n+\mathcal{O}(\rho^{-m-1}), \,{\text{ as }}\rho\to\infty,
	\end{align*}
    where $\mathfrak{C}_{n}$ is the same constant as above.
    In particular, there exists a positive constant $\mathfrak{F}_n$ such that $|\Delta_{g}^{2}(\rho^2)|\leq\mathfrak{F}_n$.
\end{proposition}
\section{Virial analysis and log-convexity for Ginzburg-Landau equation}
In this section, we consider the following  parabolic regularized equation of \eqref{eq1}, which is called Ginzburg-Landau equation with $a>0$ and $b\neq0$,
\begin{equation}\label{parabolic regularized equation}
    \partial_{t}u=(a+bi)\Big(\Delta_{g}u+V(x)u+F(t,x)\Big).
\end{equation}
We will establish the virial identity associated to the Ginzburg-Landau equation. As applications, we can  deduce logarithmic convexity and weighted regularity estimates for the Ginzburg-Landau equation. Due to the dispersive structure, we do not know whether the exponential decay holds for the  intermediate time. For this purpose we introduce  the additional dissipation.

\begin{lemma}\label{technical-lemma}
Let $u$ be the solution to \eqref{parabolic regularized equation} with $a\in\R_+,\,b\in\R$,  $V,F:\Bbb R\times M\to\C$.   
Denote $v=e^{\varphi}u$, then $v$ solves
\begin{equation}\label{equation of v}
    \partial_{t}v=(\mathcal{S}+\mathcal{A})v+(a+ib)(V(x)v+e^{\varphi}F),
\end{equation}
where
\begin{equation}\label{formula sym} \mathcal{S}=a\Big(\big\langle\nabla_g\varphi,\nabla_g\varphi\big\rangle_g+\Delta_{g}\Big)+ib\Big(-\Delta_{g}\varphi-2\big\langle\nabla_g\varphi,\nabla_g\cdot\rangle_g\Big)+\partial_{t}\varphi
\end{equation}
and
\begin{equation}\label{formula anti}
\mathcal{A}=ib\Big(\big\langle\nabla_{g}\varphi,\nabla_{g}\varphi\big\rangle_{g}+\Delta_{g}\Big)+a\Big(-\Delta_{g}\varphi-2\big\langle\nabla_{g}\varphi,\nabla_{g}\cdot\rangle_{g}\Big).
\end{equation}

Moreover, we obtain\begin{align}
    \int_{M}\mathcal{S}_{t}f\bar{f}\,\dd{\rm Vol}=&2a\int_{M}\langle\nabla_g\varphi,\nabla_{g}\varphi_{t}\rangle_g|f|^2\,\dd{\rm Vol}+2b\int_{M}\operatorname{Im}\langle\nabla_g\varphi_{t},\nabla_g f\rangle_g\bar{f}\,\dd{\rm Vol}\\&+\int_{M}\varphi_{tt}|f|^{2}\,\dd{\rm Vol}.\label{S t}
\end{align}
and
\begin{align}
    \int_M[\mathcal{S},\mathcal{A}]f\bar{f}\,\dd{\rm Vol}=&(a^{2}+b^{2})\Bigg(\int_M 4\operatorname{Hess}(\varphi)(\nabla_g\varphi,\nabla_g\varphi)|f|^{2}\,\dd{\rm Vol}-\int_{M}\Delta_{g}^{2}\varphi|f|^{2}\,\dd{\rm Vol}\\&+\int_M4\operatorname{Hess}(\varphi)(\nabla_g f,\nabla_g\bar{f})\,\dd{\rm Vol}\Bigg)+\int_M 2b\operatorname{Im}(\langle\nabla_g\partial_{t}\varphi,\nabla_g f\rangle_g)\bar{f}\,\dd{\rm Vol}\\&+\int2a\big\langle\nabla_g\varphi,\nabla_g\partial_{t}\varphi\big\rangle_g|f|^{2}\,\dd{\rm Vol}.\label{integral of commutator}
\end{align}
\end{lemma}

\begin{proof}
First, we show \eqref{S t}. A direct computation shows
\begin{align*}
\mathcal{S}_{t}=&2a\big\langle\nabla_g\varphi,\nabla_g\partial_{t}\varphi\big\rangle_g-ib\Delta_{g}\partial_{t}\varphi-2ib\big\langle\nabla_g\partial_{t}\varphi,\nabla_g\cdot\big\rangle_{g}+\partial_{tt}\varphi.
\end{align*}

Denote by $D=-\langle\nabla_g\partial_{t}\varphi,\nabla_g\cdot\rangle_g$, then its formally adjoint operator  is given by $$D^{*}=\Delta_{g}\partial_{t}\varphi+\big\langle\nabla_g\partial_{t}\varphi,\nabla_g\cdot\big\rangle_g.$$
Consequently, it follows $$\Delta_{g}\partial_{t}\varphi=D^{*}-\big\langle\nabla_g\partial_{t}\varphi,\nabla_g\cdot\big\rangle_g=D^{*}+D.$$ Moreover, we obtain \begin{align}\label{formal-adjoint}
-ib\Delta_{g}\partial_{t}\varphi-2ib\big\langle\nabla_g\partial_{t}\varphi,\nabla_g\cdot\big\rangle_g&=ib D-ib D^{*}=2b\operatorname{Im}\big(\big\langle\nabla_g\partial_{t}\varphi,\nabla_g\cdot\big\rangle_g).
\end{align}
Thus, $\mathcal{S}_t$ can be rewritten as
\begin{gather*}
\mathcal{S}_t=2a\big\langle\nabla_g\varphi,\nabla_g\varphi_{t}\big\rangle_g+\partial_{tt}\varphi+2b\operatorname{Im}\big(\big\langle\nabla_g\partial_{t}\varphi,\nabla_g\cdot\big\rangle_g\big).
\end{gather*}
This identity implies \eqref{S t}.

We now turn to the computation of the commutator $\int_{M}[\mathcal{S},\mathcal{A}]f\bar{f}\,\dd{\rm Vol}$.
\begin{align}
\int_{M}[\mathcal{S},\mathcal{A}]f\bar{f}\,\dd {\rm Vol}_{g}=&(a^{2}+b^{2})\int_{M}\Big[\big\langle\nabla_g\varphi,\nabla_g\varphi\big\rangle_g+\Delta_{g},-\Delta_{g}\varphi-2\big\langle\nabla_g\varphi,\nabla_g \cdot\rangle_g\Big]f\cdot\bar{f}\,\dd{\rm Vol}\\&+\int_{M}[\partial_{t}\varphi,\mathcal{A}]f\cdot\bar{f}\,\dd{\rm Vol}   \\
=&(a^2+b^{2})\int_{M}\Big[\big\langle\nabla_g\varphi,\nabla_g\varphi\big\rangle_g,-\Delta_{g}\varphi-2\big\langle\nabla_g\varphi,\nabla_g\cdot\rangle_g\Big]f\cdot\bar{f}\,\dd{\rm Vol}\\
&+(a^{2}+b^{2})\int_{M}\Big[\Delta_{g},-\Delta_{g}\varphi-2\big\langle\nabla_g\varphi,\nabla_g\cdot\rangle_g\Big]f\cdot\bar{f}\,\dd{\rm Vol}\\
&+\int_{M}[\partial_{t}\varphi,\mathcal{A}]f\cdot\bar{f}\,\dd{\rm Vol}\\
&\stackrel{\triangle}{=}I_1+I_2+I_3.
\end{align}

We calculate  $I_1$ and $I_3$ first. By the definition of Hessian, we have 
\begin{align}
    I_{1}=&(a^{2}+b^{2})\int_{M}\Big[\big\langle\nabla_g\varphi,\nabla_g\varphi\big\rangle_g,-2\big\langle\nabla_g\varphi,\nabla_g\cdot\big\rangle_g\Big]f\cdot\bar{f}\,\dd{\rm Vol}\\=&(a^2+b^2)\int_{M}2\Big\langle\nabla_g\varphi,\nabla_g\big\langle\nabla_g\varphi,\nabla_g\varphi\big\rangle_g\Big\rangle_g|f|^{2}\,\dd{\rm Vol}\\ =&4(a^{2}+b^{2})\int_M \operatorname{Hess}(\varphi)(\nabla_g\varphi,\nabla_g\varphi)|f|^2\,\dd{\rm Vol}.
\end{align}
Employing \eqref{formal-adjoint}, we obtain
\begin{align}
    I_3=&\int_{M}[\partial_{t}\varphi,\mathcal{A}]f\cdot\bar{f}\,\dd{\rm Vol}\\=&\int_{M}\Big[\partial_{t}\varphi,ib\Delta_g-2a\langle\nabla_g\varphi,\nabla_g\cdot\rangle_g\Big]f\cdot\bar{f}\,\dd{\rm Vol}\\=&\int_M -ib\Big((\Delta_{g}\partial_{t}\varphi)|f|^{2}+2\big\langle\nabla_g\partial_{t}\varphi,\nabla_g f\big\rangle_g\bar{f}\Big)+2a\big\langle\nabla_g\varphi,\nabla_g\partial_{t}\varphi\big\rangle_g|f|^2\dd {\rm Vol}\\
    =&\int_M 2b\operatorname{Im}(\big\langle\nabla_g\partial_{t}\varphi,\nabla_g f\big\rangle_g)\bar{f}+2a\big\langle\nabla_g\varphi,\nabla_g\partial_{t}\varphi\big\rangle_g|f|^{2}\,\dd{\rm Vol}
\end{align}
Next, we treat $I_2$. Since $$[\Delta_g,-\Delta_g\varphi-2\langle\nabla_g\varphi,\nabla_g\cdot\rangle_g]f=[\Delta_{g},-\Delta_g\varphi]f+[\Delta_{g},-2\langle\nabla_g\varphi,\nabla_g\cdot\rangle_g]f,$$  $I_2$ can be decomposed into following three terms 
\begin{align*}
    I_2=&-(a^{2}+b^{2})\int_{M}\Delta_{g}^2\varphi |f|^{2}\dd{\rm Vol}{-2(a^2+b^2)\int_{M}\langle\nabla_{g}\Delta_{g}\varphi,\nabla_{g} f\rangle_{g}\bar{f}\,\dd{\rm Vol}}\\
    &+2(a^{2}+b^{2})\int_{M}\big(-\Delta_{g}(\langle\nabla_g\varphi,\nabla_g f\rangle_g)\bar{f}+\big\langle\nabla_g\varphi,\nabla_g\Delta_{g}f\big\rangle_g\bar{f}\big)\,\dd{\rm Vol}\\
   \stackrel{\triangle}{=} &I_{21}
+I_{22}+I_{23}.\end{align*} 
$I_{21}$ will remain in the final identity, whereas $I_{22}$ will be cancelled later. It therefore suffices to compute $I_{23}$ in detail.

For the first contribution to $I_{23}$, the integrating by part yields
\begin{align}
    I_{23}^{1}\stackrel{\triangle}{=}&(a^{2}+b^{2})\int_M 2\Big\langle\nabla_g\big\langle\nabla_g\varphi,\nabla_g f\big\rangle_g,\nabla_g\bar{f}\Big\rangle_g\,\dd{\rm Vol}\\=&2(a^{2}+b^{2})\int_M\operatorname{Hess}(\varphi)(\nabla_g f,\nabla_g\bar{f})\dd{\rm Vol}+{2(a^{2}+b^{2})\int_{M}\operatorname{Hess}(f)(\nabla_g\varphi,\nabla_g\bar{f})\,\dd{\rm Vol}}\\
    \stackrel{\triangle}{=}&K_1+K_2,
\end{align}
where $K_1$ is the required term and $K_2$ will be cancelled out later.
For the second term contributed to $I_{23}$, we have 
\begin{align}
    I_{23}^{2}\stackrel{\triangle}{=}&2(a^{2}+b^{2})\int_M \big\langle\nabla_g\varphi,\nabla_g\Delta_g f\big\rangle_g\bar{f}\,\dd{\rm Vol}\\=&-2(a^{2}+b^{2})\int_M \big(\Delta_{g}\varphi\Delta_{g}f\bar{f}+\langle\nabla_g\varphi,\nabla_g\bar{f}\rangle_{g}\Delta_{g}f\big)\dd{\rm Vol}\\
    \stackrel{\triangle}{=}&K_3+K_4.
\end{align}
For $K_3$, by integrating by parts, it holds
\begin{align}\label{red part}
  K_3=&2(a^{2}+b^{2})\int_{M}\Big\langle\nabla_g(\Delta_{g}\varphi\bar{f}),\nabla_g f\Big\rangle_{g}\dd{\rm Vol}\\=&2(a^{2}+b^{2})\int_M\Big\langle\nabla_g\Delta_{g}\varphi,\nabla_g f\Big\rangle_g\bar{f}\dd {\rm Vol}+2(a^{2}+b^{2})\int_M\Delta_g\varphi\big\langle\nabla_g f,\nabla_g\bar{f}\big\rangle_g\dd {\rm Vol}\\
  \stackrel{\triangle}{=}&K_3^1+K_3^2.
\end{align}
Observe that the following  cancellation holds: 
\begin{align*}
    K_3^1+I_{22}=0.
\end{align*}
For the term $K_4$, another integrating by part yields
\begin{align}
K_4=&2(a^{2}+b^{2})\int_M \Big\langle\nabla_{g}\big\langle\nabla_{g}\varphi,\nabla_g\bar{f}\big\rangle_g,\nabla_g f\Big\rangle_g\,\dd {\rm Vol}\\=&2(a^{2}+b^{2})\int_M \big(\operatorname{Hess}(\varphi)(\nabla_g f,\nabla_g\bar{f})+\operatorname{Hess}(\bar{f})(\nabla_g\varphi,\nabla_g f)\big)\,\dd{\rm Vol}\\
\stackrel{\triangle}{=}&K_4^1+K_4^2.
\end{align}
Using the symmetry between $K_2$ and $K_4^2$, we deduce
\begin{align}
 K_3^2+K_2+K_4^2=& 2(a^{2}+b^2)\int_M|\nabla_g f|^{2}\Delta_g\varphi\,\dd{\rm Vol}+2(a^{2}+b^{2})\int_M \big\langle\nabla_g|\nabla_g f|_g^{2},\nabla_g\varphi\big\rangle_g\,\dd{\rm Vol}\\
   =&0.
\end{align}

Thus, we get 
\begin{align}
    \int_M[\mathcal{S},\mathcal{A}]f\bar{f}\,\dd {\rm Vol}=&(a^{2}+b^{2})\Bigg(\int_M 4\operatorname{Hess}(\varphi)(\nabla_g\varphi,\nabla_g\varphi)|f|^{2}\,\dd{\rm Vol}-\int_{M}\Delta_{g}^{2}\varphi|f|^{2}\,\dd{\rm Vol}\\&+\int_M4\operatorname{Hess}(\varphi)(\nabla_g f,\nabla_g\bar{f})\,\dd{\rm Vol}\Bigg)+\int_M 2b\operatorname{Im}(\langle\nabla_g\partial_{t}\varphi,\nabla_g f\rangle_g)\bar{f}\,\dd{\rm Vol}\\&+\int_{M}2a\langle\nabla_g\varphi,\nabla_g\partial_{t}\varphi\rangle_g|f|^{2}\,\dd{\rm Vol}.
\end{align}
Hence, we complete the proof of Lemma \ref{technical-lemma}.
\end{proof}
\subsection{Regularity estimate}
In this part, we will prove the quadratic exponential decay for the solution to \eqref{parabolic regularized equation}. 
\begin{lemma}\label{sub-Gaussian}
    Let $u\in C([0,1],L^{2}(M))\cap L^{2}([0,1],H^{1}(M))$ be the solution to \eqref{parabolic regularized equation} and $V\in L^\infty(M)$ be a complex-valued function. We assume that for some $\gamma>0$, it holds that
    \begin{equation}
    e^{\gamma\rho^2}u(x,0)\in L^{2}(M),\,\,e^{\gamma\rho^2}F(x,t)\in L^{\infty}([0,1],L^{2}(M)),
    \end{equation}
    where $\rho$ is the geodesic distant between $x$ and $\mathbf{0}$.
    Then, for any $t\in[0,1]$, the Gaussian decay holds for u
    \begin{equation}\label{parabolic gaussian decay}
       \|e^{\frac{\gamma a\rho^2}{a+4\gamma(a^{2}+b^{2})t}
       }u\|_{L^2}\leq e^{\|(a\operatorname{Re}V)^{+}-b\operatorname{Im}V\|_{L^{1}_{t}L^{\infty}_{x}}}\Big(\|e^{\gamma\rho^{2}}u(0,x)\|_{L^{2}}+\sqrt{a^{2}+b^{2}}\|e^{\frac{\gamma a\rho^2}{a+4\gamma(a^{2}+b^{2})t}}F\|_{L^{1}_{t}L_x^2}\Big)
    \end{equation}
\end{lemma}
\begin{proof}
    Let $v=e^{\varphi}u$, then $v$ satisfies 
    \begin{equation}\label{v aha}
        \partial_{t}v=\mathcal{S}v+\mathcal{A}v+(a+bi)(e^{\varphi}F+Vv),
    \end{equation}
    where $\mathcal{S}$ and $\mathcal{A}$ are defined in \eqref{formula sym} and \eqref{formula anti}. We use the energy method to get the decay rate of $v$ in $L^2$ space. Without loss of generality, we only consider the case that $v$ has non-trivial $L^2$ mass. Formally, multiplying $\bar{v}$ to both side of \eqref{v aha} and take real part, we get
    \begin{equation}\label{partial t v}
        \partial_{t}\big(\|v\|_{L^2}^{2}\big)=2\operatorname{Re}(\mathcal{S}v,v)_{L^2}+2\operatorname{Re}\Big((a+bi)(e^\varphi F+Vv),v\Big)_{L^2}.
    \end{equation}
    A formal integrating by parts gives
    \begin{align}
        \operatorname{Re}(\mathcal{S}v,v)_{L^2}=&\int_{M}a\langle\nabla_{g}\varphi,\nabla_{g}\varphi\rangle_{g}|v|^2\,\dd{\rm Vol}-\int_{M}a|\nabla_{g} v|_{g}^2\,\dd{\rm Vol}+2b\operatorname{Im}\int_{M}\bar{v}\langle\nabla_{g}\varphi,\nabla_{g} v\rangle_{g}\dd{\rm Vol}\\&+\int_{M}\partial_{t}\varphi|v|^2\,\dd{\rm Vol}.
    \end{align}
    Invoking it to \eqref{partial t v}, we obtain
    \begin{align}
        \partial_t(\|v\|^2_{L^2})=&-2a\int_{M}|\nabla_{g} v|_{g}^{2}\,\dd{\rm Vol}+2\int_{M}(a|\nabla_{g}\varphi|_{g}^{2}+\partial_{t}\varphi)|v|^{2}\,\dd{\rm Vol}+2\int_{M}(a\operatorname{Re}V-b\operatorname{Im}V)|v|^{2}\,\dd{\rm Vol}\\&+2\operatorname{Re}\Big((a+bi)e^\varphi F,v\Big)_{L^{2}}+4b\operatorname{Im}\int_{M}\bar{v}\langle\nabla_{g}\varphi,\nabla_{g} v\rangle_{g}\,\dd{\rm Vol}.
    \end{align}
    By Young's inequality, we have 
    \begin{align}
        4b\operatorname{Im}\int_{M}\bar{v}\langle\nabla_{g}\varphi,\nabla_{g} v\rangle_{g}\,\dd{\rm Vol}\leq 2a\int_{M}|\nabla_{g} v|_{g}^{2}\,\dd{\rm Vol}+\frac{2b^2}{a}\int_{M}|\nabla_{g}\varphi|_{g}^2|v|^{2}\,\dd{\rm Vol},
    \end{align}
    \begin{align}
        2\Big|\int_{M}(a\operatorname{Re}V-b\operatorname{Im}V)|v|^{2}\,\dd{\rm Vol}\Big|\leq2\Big\|(a\operatorname{Re}V)^+-b\operatorname{Im}V\Big\|_{L^\infty}\|v\|_{L^2}^2,
    \end{align}
    and
    \begin{align}
        2\Big|\operatorname{Re}\Big((a+bi)e^{\varphi}F,v\Big)_{L^2}\Big|\leq2\sqrt{a^2+b^2}\|e^{\varphi}F\|_{L^2}\|v\|_{L^2}.
    \end{align}
    Let $\varphi=\alpha(t)\phi(x)$, in which $\phi(x)=\rho^{2}$, and $\alpha(t)$ satisfy
    \begin{equation}
        \begin{cases}
            \alpha^{\prime}(t)=-4(a+\frac{b^2}{a})\alpha(t)^{2},\\ \alpha(0)=\gamma.
            \end{cases}
    \end{equation}
   By the classical ODE method, one has
    \begin{align*}
         \alpha(t)=\frac{\gamma a}{a+4\gamma(a^{2}+b^{2})t}.
    \end{align*}
    In addition, we obtain
    $$\int_{M}\Big(\Big(2a+\frac{2b^{2}}{a}\Big)|\nabla_{g}\varphi|_{g}^{2}+\partial_{t}\varphi\Big)|v|^{2}\,\dd{\rm Vol}\leq0.$$
    Then we infer that  
    \begin{align}
        \partial_{t}\|v\|_{L^2}\leq &\|(a\operatorname{Re}V)^{+}-b\operatorname{Im}V\|_{L^\infty}\|v\|_{L^2}\\ &+\sqrt{a^2+b^2}\|e^{\varphi}F\|_{L^2}.
    \end{align}
    Furthmore, we get
    \begin{equation}
        \frac{\dd}{\dd t}\Big(e^{-\int_{0}^{t}\|(a\operatorname{Re}V)^{+}-b\operatorname{Im}V\|_{L^\infty}ds}\|v\|_{L^2}\Big)\leq \sqrt{a^2+b^2}e^{-\int_{0}^{t}\|(a\operatorname{Re}V)^{+}-b\operatorname{Im}V\|_{L^\infty}ds}\|e^{\varphi}F\|_{L^{2}}.
    \end{equation}
    Then, taking the integral over time,
    \begin{equation}
        e^{-\int_{0}^{t}\|(a\operatorname{Re}V)^{+}-b\operatorname{Im}V\|_{L^\infty}ds}\|v\|_{L^2}-\|v_{0}\|_{L^{2}}\leq \sqrt{a^2+b^2}\int_{0}^{t}e^{-\int_{0}^{\tau}\|(a\operatorname{Re}V)^{+}-b\operatorname{Im}V\|_{L^\infty}ds}\|e^{\varphi}F\|_{L^{2}}\,\dd\tau,
    \end{equation}
    which implies that
    \begin{align}
        \|v\|_{L^2}\leq& e^{\int_{0}^{t}\|(a\operatorname{Re}V)^{+}-b\operatorname{Im}V\|_{L^\infty}ds}\|v_{0}\|_{L^{2}}+\sqrt{a^{2}+b^{2}}\int_{0}^{t}e^{\int_{\tau}^{s}\|(a\operatorname{Re}V)^{+}-b\operatorname{Im}V\|_{L^\infty}ds}\|e^{\varphi}F\|_{L^{2}}\,\dd\tau\\
        \leq& e^{\|(a\operatorname{Re}V)^{+}-b\operatorname{Im}V\|_{L^{1}_{t}L^{\infty}_{x}}}\Big(\|v_{0}\|_{L^{2}}+\sqrt{a^{2}+b^{2}}\|e^{\varphi}F\|_{L^{1}_{t}L_x^2}\Big).
    \end{align}
    Since $e^{\alpha(t)\phi(x)}u$ may not belong to $L^2$, all the calculations above are formal. To make the above process rigorous, we proceed by  approximation argument.

     First, we set the truncation $\varphi_{R}=\alpha(t)\Phi_{R}$ where
    \begin{equation}
     \alpha(t)=\frac{\gamma a}{a+4\gamma(a^{2}+b^{2})t},\,\,\,   
    \end{equation}
    and 
    \begin{equation}
    	\Phi_R(x)=\begin{cases}
    		\rho^2,\rho\leq R\\R^2,\rho\geq R.
    	\end{cases}
    \end{equation}
    Then we mollify $\Phi_{R}$.
    Let $\theta\in C_c^\infty([0,\infty))$ be a nonnegative cutoff function supported in $[0,1]$ with $\int_{\mathbb{R}}\theta(z)\,\dd z=1$. 
  Using the exponential map,  we define the mollified function as
    \begin{equation}\label{convolution}
        \Phi_{R,\varepsilon}(x)=\frac{1}{\int_{|\vec{v}|\leq\varepsilon}\theta_\varepsilon(|\vec{v}|) \mathcal{J}(\vec{v})\,\dd\vec{v}}\int_{|\vec{v}|\leq\varepsilon}\Phi_{R}(\exp_{x}{\vec{v}})\theta_{\varepsilon}(|\vec{v}|)\mathcal{J}(\vec{v})\,\dd\vec{v}
    \end{equation}
where $\mathcal{J}(\vec{v})=\sqrt{\det g}$.
Let
\begin{equation}\label{measure}
d\mu_\varepsilon(\vec{v})
=
\frac{\theta_\varepsilon(|\vec{v}|)\mathcal{J}(\vec{v})}{\int_{|\vec{v}|\leq\varepsilon} \theta_\varepsilon(\vec{v}) \mathcal{J}(\vec{v})\,\dd\vec{v}}
\,\dd\vec{v},
\end{equation}
which defines a  measure supported in $B_\varepsilon(0)\subset \Bbb R^n$. By the property of mollifier, $\Phi_{R,\varepsilon}$ is positive. Then, we denote a mollifier $\varphi_{R,\varepsilon}=a(t)\Phi_{R,\varepsilon}$.

\begin{lemma}\label{approximation}
	Let
	\[
	\Phi_R(x)=
	\begin{cases}
		\rho(x)^2, & \rho(x)\le R,\\
		R^2, & \rho(x)>R,
	\end{cases}
	\]
	and define its mollification
	\begin{equation}\label{mollify}
	\Phi_{R,\varepsilon}(x)
	=
	\frac{1}{\int_{|\vec{v}|\leq\varepsilon}\theta_\varepsilon(\vec{v}) \mathcal{J}(\vec{v})\,\dd\vec{v}}\int_{|\vec v|\le \varepsilon}
	\Phi_R(\exp_x \vec v)\,
	\theta_\varepsilon(|\vec v|)\,\mathcal{J}(\vec v)\,d\vec v,
	\end{equation}
	where $\theta_\varepsilon$ is a standard radial mollifier and $J(\vec{v})$ is the Jacobian. Then the following statements hold:
	
	\begin{enumerate}
		\item (Upper bound)
		\begin{equation}\label{upper bound}
		\Phi_{R,\varepsilon}(x)\le \Phi_R(x)+C\varepsilon.
		\end{equation}
		
		\item (Gradient structure estimate)
		\begin{equation}\label{structure estimate}
		|\nabla_g\Phi_{R,\varepsilon}(x)|_g^2 - 4\Phi_{R,\varepsilon}(x)
		\le C \varepsilon^2.
		\end{equation}
	\end{enumerate}
\end{lemma}
As a result, we derive that
\begin{equation}
    \|e^{\varphi_{\varepsilon,R}}u\|_{L^{2}}\leq e^{\|a(Re V)^{+}-bIm V\|_{L^1_t L^\infty_x}+C_R \varepsilon^2}\Big(\|e^{\varphi_{\varepsilon,R}(0,\cdot)}u_{0}\|_{L^{2}}+\sqrt{a^{2}+b^{2}}\|e^{\varphi_{R,\varepsilon}}F\|_{L^{1}_{t}L^{2}_{x}}\Big).
\end{equation}
We let $\varepsilon\to0$ for fixed $R$, and then let $R\to\infty$, which implies the desired estimate \eqref{parabolic gaussian decay}.
\end{proof}
Now, we give the proof of Lemma \ref{approximation}.
\begin{proof}[Proof of Lemma \ref{approximation}]
	First, we prove the bound \eqref{upper bound}.
Note that 
\begin{equation}\label{lipchitz}
	|\nabla_g \Phi_R|_g =
	\begin{cases}
		2\rho(x)\le 2R, & \rho(x)<R,\\
		0, & \rho(x)>R,
	\end{cases}
\end{equation}
	Let $\gamma(t)$ be a minimizing geodesic connecting $x$ and $y$.
	Then
	\[
	\Phi_R(y)-\Phi_R(x)
	=
	\int_0^1 \langle \nabla_g \Phi_R(\gamma(t)),\dot\gamma(t)\rangle_g dt.
	\]
	Hence,
	\[
	|\Phi_R(y)-\Phi_R(x)|
	\le
	\|\nabla_g \Phi_R\|_{L^\infty}\, d(x,y).
	\]
	Thus $\Phi_R$ is globally Lipschitz,
	\begin{align}\label{Lipschitz}
	|\Phi_R(x)-\Phi_R(y)|\le C_R\,d(x,y),
	\end{align}
where $C_R:=\|\nabla\Phi_R\|_{L^\infty}$.
	By definition  \eqref{mollify}, we can write
	\[
	\Phi_{R,\varepsilon}(x)
	=
	\int_{d(y,x)\leq\varepsilon} \Phi_R(y)\,d\mu_\varepsilon(x,y),
	\]
	where $\dd\mu_\varepsilon$ is a  measure supported in $B_\varepsilon(x)$.
	By the Lipschitz property \eqref{Lipschitz}, we deduce
	\[
	\Phi_{R,\varepsilon}(x)
	\le
	\Phi_R(x) + C_R \int_{d(x,y)\leq\varepsilon} d(x,y)\,d\mu_\varepsilon.
	\]
	Since $d(x,y)\le \varepsilon$ on the support, we obtain
	\[
	\Phi_{R,\varepsilon}(x)\le \Phi_R(x)+C_R\varepsilon.
	\]
This inequality finishes the proof of \eqref{upper bound}. 
	
	Next, we move on to proving \eqref{structure estimate}.	
	Applying the chain rule to $\Phi_{R,\varepsilon}$ infers
	\begin{equation}\label{chain rule}
	\nabla_g \Phi_{R,\varepsilon}(x)=\int_{|\vec{v}|\leq\varepsilon} (d\exp_x(v))^* \nabla_g \Phi_R(\exp_x v)\,\dd\mu_\varepsilon(\vec{v}).
	\end{equation}
	To estimate \eqref{chain rule}, we recall the following vector-valued Minkowski inequality. Suppose that $H$ is a inner product space. Let $\mu$ be a measure and $a\in L^2(\mu;H)$, it follows \begin{equation}\label{Minkowski inequality}
	\Big|\int a\,d\mu\Big|^2
	\le
	\int |a|^2\,d\mu.
	\end{equation}
	Applying the inequality \eqref{Minkowski inequality}   to $H=(T_x M,g_x)$, it holds
	\begin{equation}\label{L2 inequality}
	|\nabla_g \Phi_{R,\varepsilon}(x)|_g^2
	\le
	\int_{|\vec{v}|\leq\varepsilon} |(d\exp_x(v))^* \nabla_g \Phi_R(y)|_g^2\,\dd\mu_\varepsilon(\vec{v}),
	\end{equation}
	where $y=\exp_x v$.

	To give a explicit estimate for $d_x\exp_x(\vec{v})$, we apply the Taylor expansion to the Jacobi field. 
	Let $\gamma(s)=\exp_x(s\vec{v})$ and  $J(s)$ be the Jacobi field defined by
	\[
	J(s)=d_x\exp_x(s\vec{v})(\xi),
	\]
	with initial data
	\[
	J(0)=\xi,\quad \dot{J}(0)=0.
	\]
	Consequently, $J$ satisfies the Jacobi equation
	\begin{equation}\label{Jacobi equation}
	\ddot{J} + R(J,\dot\gamma)\dot\gamma=0.
	\end{equation}
	A Taylor expansion at $s=0$ yields
	\[
	J(s)
	=
	\xi + \frac{s^2}{2}\ddot{J}(0) + O(s^3).
	\]
	Using the Jacobi equation \eqref{Jacobi equation} at $s=0$,
	we obtain
	\[
	d_x\exp_x(\vec{v})(\xi)
	=
	\xi - \frac{1}{2}R(\xi,\vec{v})\vec{v} + O(|\vec{v}|^3).
	\]
	Let $L = d_x\exp_x(\vec{v}):T_xM \to T_yM$, then
	\begin{equation}\label{perturbation estimate for jacobian}
	L = \mathrm{Id} + A(\vec{v}), \quad |A(\vec{v})| \le C|\vec{v}|^2.
	\end{equation}
Suppose that $L^*$ be the formal adjoint operator of $L$, which reads as
	\[g_y(L\xi,\eta)=g_x(\xi,L^*\eta).
	\]
	According to \eqref{perturbation estimate for jacobian},
    one obtains
	\[
	L^* = \mathrm{Id} + E(\vec{v}), \quad |E(\vec{v})| \le C|\vec{v}|^2.
	\]
	That is,
	\[
	(d\exp_x(\vec{v}))^* = \mathrm{Id} + E(\vec{v}),
	\quad |E(\vec{v})| \le C|\vec{v}|^2.
	\]
	Thus, we get
	\begin{align}\label{decompose-ab}
	\nabla_g \Phi_{R,\varepsilon}(x)
	=
	\int \nabla_g \Phi_R(y)\,\dd\mu_\varepsilon
	+
	\int E(v)\nabla_g \Phi_R(y)\,\dd\mu_\varepsilon
	=: A + B.
	\end{align}
	By \eqref{L2 inequality} and \eqref{lipchitz}, it holds that
	\begin{align}\label{AA}	|A|^2
	\le
	\int |\nabla \Phi_R(y)|^2\,\dd\mu_\varepsilon
	\le
	4\Phi_{R,\varepsilon}(x).
	\end{align}
	Moreover,
	\begin{align}\label{BB}
	|B|
	\le
	C \int |v|^2 |\nabla \Phi_R(y)|\,\dd\mu_\varepsilon
	\le
	C \varepsilon^2.
	\end{align}
	Putting \eqref{decompose-ab}, \eqref{AA} and \eqref{BB} together, we obtain
	\[
	|\nabla \Phi_{R,\varepsilon}|^2
	=
	|A+B|^2
	\le
	4\Phi_{R,\varepsilon}
	+ C_R\varepsilon^2.
	\]
Therefore, we complete the proof of Lemma \ref{approximation}.
\end{proof}

\subsection{Logarithmic convexity inequality in hyperbolic geometry}
In this part, we establish the logarithmic convexity that links the states between the endpoints $t=0,1$ and those at intermediate times $t\in(0,1)$. This property will be employed in the final step of the proof of unique continuation. Specifically, it asserts that if exponential decay holds at the endpoints $t=0,1$, then the same  decay persists for all intermediate times $t\in(0,1)$. To this end, we should estimate the weighted $L^2$-norm 
$ \|e^{\gamma\rho^{2}}u(t)\|_{L^2(M)} $ for any $t\in(0,1)$, where the weight is given by $\varphi = \gamma\rho^2$.

For this purpose, we first introduce the abstract lemma
established in Escauriaza-Kenig-Ponce-Vega \cite{EKPV-JEMS}.  
\begin{lemma}\label{abstract lemma1}
Let $\mathcal{S}$ be a symmetric operator and $\mathcal{A}$ be a skew-symmetric operator, with coefficients depending on $x$ and $t$.   Assume that $f(x,t)$ is a  smooth function and $W$ is a positive function. Let 
\begin{align*}
    H(t)=(f,f)_{L^2}.
\end{align*}
Assume that 
\begin{align}
    \Big|\partial_{t}f-\big(\mathcal{S}+\mathcal{A}\big)f\Big|\leq M_{1}|f|+W, ~~~~~(x,t)\in M\times[0,1],\,\,\, M_{1}\geq0,
\end{align}
and 
\begin{align}\label{lower bound}
    \mathcal{S}_{t}+[\mathcal{S},\mathcal{A}]\geq-M_{0},~~~~~~M_0\geq0.
\end{align}
Let $M_2$ be defined as 
\begin{align}
    M_{2}:=\sup_{t\in[0,1]}\frac{\big\|W(t)\big\|_{L^{2}}}{\|f(t)\|_{L^{2}}},
\end{align}
then the function $\psi(t):=\log H(t)$ is convex in $t\in[0,1]$. In particular, if 
\begin{align}\label{boundedness of H}
    H(0)<\infty\Rightarrow H(t)<\infty,~~\forall t\in[0,1],
\end{align}
then there exists a constant $N\geq0$ such that
\begin{align}
    H(t)\leq e^{N(M_{0}+M_{1}+M_{2}+M_{1}^{2}+M_{2}^{2})}H(0)^{1-t}H(1)^{t},~~~\forall t\in[0,1].
\end{align}
\end{lemma}

With this lemma in hand, we have the following logarithmic convexity result.

\begin{lemma}\label{log convex lemma}
    Suppose that $u\in C([0,1],L^{2}(\mathbb H^n))\cap L^{2}([0,1],H^{1}(\mathbb H^n))$ is a solution to 
    \eqref{parabolic regularized equation}
    in $\mathbb H^n\times[0,1]$ with $a>0$ and $b\in\mathbb{R}$. Let $V:\Bbb H^n\to\Bbb C$, $F:\mathbb H^n\times[0,1]\to\mathbb{C}$ and  $\gamma>0$. Also, we assume that  
    \begin{equation*}
\big\|V\big\|_{L^{\infty}}:=M_{1}<\infty, ~~~~\sup_{t\in[0,1]}\frac{\|e^{\gamma|\cdot|^{2}}F(\cdot,t)\|_{L^{2}}}{\|u(\cdot,t)\|_{L^{2}}}:=M_{2}<\infty.
    \end{equation*}
If the following estimate holds 
    \begin{align}
\|e^{\gamma\rho^{2}}u(\cdot,0)\|_{L^{2}}+\|e^{\gamma\rho^{2}}u(\cdot,1)\|_{L^{2}}<\infty,
    \end{align}
then  $H(t)=\|e^{\gamma\rho^{2}}u(t,\cdot)\|^{2}_{L^{2}}$ is  finite and logarithmically convex in $t\in[0,1]$. Moreover,  there exists a constant $N=N(\gamma,a,b)$ such that 
\begin{equation}
    H(t)\leq e^{N[\gamma\mathfrak{C}_{n}(a^{2}+b^{2})+\sqrt{a^{2}+b^{2}}(M_{1}+M_{2})+(a^{2}+b^{2})(M_{1}^{2}+M_{2}^{2})]}H(0)^{1-t}H(1)^{t}.
\end{equation}

\end{lemma}
\begin{proof}
First, we recall the following properties borrowed from Lemma \ref{Hessian-hyperbolic}, 
\begin{equation}\label{Hess-rho}
\operatorname{Hess}\rho=\coth\rho(g-\dd\rho^{2}).
\end{equation}
From this identity and the chain rule, we obtain
\begin{equation}
    \operatorname{Hess}(\rho^{2})=2\rho\coth\rho\, g+2(1-\rho\coth\rho)\dd\rho^{2}.
\end{equation}
For any vector $\vec{v}=a\partial_\rho+\vec{u}$ with unit vector $\vec{u}$, we have the decomposition
 \begin{align}
     \operatorname{Hess}(\rho^{2})(\vec{v},\vec{v})=&2(a^{2}+1)\rho\coth\rho+2(1-\rho\coth\rho)a^{2}\\=&2a^{2}+2\rho\coth\rho\geq2.
\end{align}
Thus, we obtain the lower bound
\begin{align}
    &\int_{\Bbb H^n}\mathcal{S}_{t}f\bar{f}+[\mathcal{S},\mathcal{A}]f\bar{f}\,\dd{\rm Vol}\\ \geq&(a^{2}+b^{2})\int_{\Bbb H^n}(32\gamma^{2}\rho^{2}+32\gamma^{3}\rho^{3}\coth\rho)|f|^{2}\,\dd{\rm Vol}\\&-(a^{2}+b^{2})\gamma\mathfrak{C}_{n}\int_{\Bbb H^n}|f|^{2}\,\dd{\rm Vol}+8\gamma(a^{2}+b^{2})\int_{\Bbb H^n}|\nabla_g f|_{g}^{2}\,\dd{\rm Vol}\\=&(a^{2}+b^{2})\int_{\Bbb H^n}(32\gamma^{2}\rho^{2}+32\gamma^{3}\rho^{3}\coth\rho)|f|^{2}\,\dd{\rm Vol}-(a^{2}+b^{2})\gamma\mathfrak{C}_{n}\int_{\Bbb H^n}|f|^{2}\,\dd{\rm Vol}\\&+8\gamma(a^{2}+b^{2})\int_{\Bbb H^n}|\nabla_{g} f|_{g}^{2}\,\dd{\rm Vol}.
\end{align}
Then we deduce that 
\begin{equation}
    \mathcal{S}_{t}+[\mathcal{S},\mathcal{A}]\geq-(a^{2}+b^{2})\gamma\mathfrak{C}_{n},
\end{equation}
where $\mathfrak{C}_n$ is the constant appeared in Remark \ref{bi-laplacian}.

By Lemma \ref{abstract lemma1}, we obtain the desired logarithmic convexity inequality
\begin{equation}\label{log convex inequa}
\Big\|e^{\gamma\rho^{2}}u(\cdot,t)\Big\|_{L_x^{2}}\leq e^{N[\sqrt{a^{2}+b^{2}}(M_1+M_2)+(a^{2}+b^{2})(\gamma\mathfrak{C}_{n}+M_{1}^{2}+M_{2}^{2})]}\Big\|e^{\gamma\rho^{2}}u(\cdot,0)\Big\|_{L_x^{2}}^{1-t}\Big\|e^{\gamma\rho^{2}}u(\cdot,1)\Big\|_{L_x^{2}}^{t},
\end{equation}
where $$M_{1}=\|V\|_{L^\infty_x},\,\,M_{2}=\sup_{t\in[0,1]}\frac{\|e^{\gamma\rho^{2}}F\|_{L^2_x}}{\|u\|_{L^{2}_{x}}}.$$
However, the above calculation is formally. To make the proof rigorously,  we need 
\begin{equation}\label{claim-uniform}
    \|e^{\gamma\rho^{2}}u(t)\|_{L^{2}(\mathbb{H}^{n})}\leq C,\,\forall t\in[0,1].
\end{equation} 
To verify this, we will use the approximation argument.  First, as a consequence of Lemma \ref{sub-Gaussian}, we obtain a sub-Gaussian decay $$\|e^{\gamma\rho^{2-2\delta}}u(t,x)\|_{L^2(\Bbb H^n)}\leq C_\delta.$$ However, to take the limit as $\delta\to0$, we need to prove the estimate uniformly in $\delta$. To this purpose, we first give some calculations associated to the terms involving Hessian and Laplacian.

\begin{lemma}\label{pre-biLaplacian}
    For any smooth function $h(\rho)$ with $\rho:=d(x,\mathbf{0)}$ being the geodesic distant function, we have
     \begin{equation}\label{Hess com}
        \operatorname{Hess}(h(\rho))=h^{\prime\prime}(\rho)\nabla\rho\otimes\nabla\rho+h^{\prime}(\rho)\operatorname{Hess}(\rho),
    \end{equation}
    and \begin{equation}\label{Laplace com}
\Delta_{g}h(\rho)=h^{\prime\prime}(\rho)+\mathcal H(\rho)h^{\prime}(\rho),
    \end{equation}
    where, $\mathcal H(\rho)=\Delta_g\rho$ is the mean curvature on geodesic spheres $S_\rho$ with radius $\rho$.
\end{lemma}
Now we give the calculation of $\Delta_{g}^{2}(\rho^{2-2\delta})$ with $0<\delta\ll1$. Let
	$h(\rho)=\rho^{2-2\delta}$. Using  Lemma \ref{pre-biLaplacian},  we have
	\begin{equation}
		\begin{aligned}
			\Delta_g h(\rho)=&h^{\prime\prime}(\rho)+(d-1)\coth\rho\, h^{\prime}(\rho).
		\end{aligned}
	\end{equation}
From \eqref{Hess-rho}, we obtain
	\begin{equation}\label{178537}
    h^{\prime}(\rho)=2(1-\delta)\rho^{1-2\delta}, \,\, h^{\prime\prime}(\rho)=2(1-\delta)(1-2\delta)e^{-2\delta},
	\end{equation} 
and
\begin{equation}
\operatorname{Hess}(\rho^{2-2\delta})=2(1-\delta)\rho^{1-2\delta}\coth\rho\, g+2(1-\delta)\rho^{-2\delta}[(1-2\delta)-\rho\coth\rho]\dd\rho\otimes\dd\rho.
	\end{equation}
Thus, taking the trace of $\operatorname{Hess}(h(\rho))$, we obtain
\begin{equation}\Delta_g(\rho^{2-2\delta})=2(1-\delta)\rho^{-2\delta}\Big((1-2\delta)+(n-1)\rho\coth\rho\Big).
	\end{equation}
Furthermore, we have
	\begin{equation}\Delta_{g}^{2}(\rho^{2-2\delta})=\Delta_g(2(1-\delta)\rho^{-2\delta}((1-2\delta)+(n-1)\rho\coth\rho)).
	\end{equation}  
To calculate the biharmonic operator, we need to further calculate the derivatives of $$h_1(\rho):=
	2(1-\delta)\rho^{-2\delta}\Big((1-2\delta)+(n-1)\rho\coth\rho\Big).$$ A direct calculation yields that  
	\begin{equation}\label{h1-1}
		h_{1}^{\prime}(\rho) = 2(1-\delta)\rho^{-2\delta-1} [ -2\delta ( (1-2\delta) + (n-1)\rho \coth\rho ) + (n-1)\rho ( \coth\rho - \rho \operatorname{csch}^2\rho ) ]
	\end{equation}
	and
	\begin{align}
		h_{1}^{\prime\prime}(\rho) =& 2(1-\delta)\rho^{-2\delta-2}\Big[ 2\delta(2\delta+1) ( (1-2\delta) + (n-1)\rho \coth\rho ) - 4\delta (n-1)\rho ( \coth\rho - \rho \operatorname{csch}^2\rho )\\& + 2(n-1)\rho^2 \operatorname{csch}^2\rho ( \rho \coth\rho - 1 )\Big].\label{h1-2}
	\end{align}
Inserting \eqref{h1-1} and \eqref{h1-2} into \eqref{Laplace com}, we obtain
	\begin{multline}\Delta_g^2(\rho^{2-2\delta}) = \Delta_gh_1(\rho) =2(1-\delta)\rho^{-2\delta}\Big[(1-2\delta)((n-1)^2+2\delta(2\delta+1)\rho^{-2}-4\delta(n-1)\rho^{-1}\coth\rho)\\+(n-1)(3-n)\operatorname{csch}^2\rho(\rho\coth\rho-(1-2\delta))\Big]\end{multline}
	which implies that $|\Delta_g^{2}(\rho^{2-2\delta})|\leq \mathfrak{D}_{n}<\infty$ for $\rho\geq1$. 

With Lemma \ref{pre-biLaplacian} in hand, we are now in position to prove that the upper bound in  \eqref{claim-uniform} is independent of $\delta$.

Let $\eta\in C_{c}^{\infty}(B_{1}^{c})$ be a smooth cut-off function with $\eta\equiv1$ on $B_{2}^{c}$. Here, $B_{R}=\{x\in\mathbb{H}^{n}:d(x,\mathbf{0})\leq R\}$ denotes the geodesic ball with radius $R$. Thus, we decompose the solution $u$ into $u=\eta u+(1-\eta)u$.
For convenience, we denote $\varphi_1=\gamma\rho^{2-2\delta}$ and $f_{\delta}=e^{\varphi_1}u\eta$. By \eqref{parabolic gaussian decay}, we get 
\begin{equation}
\|e^{\gamma\rho^{2-2\delta}}(1-\eta)u\|_{L^\infty_{t}L^2_x}\leq e^{4\gamma}\|u\|_{L^\infty_t L^{2}_x(B_2)}\leq C.    
\end{equation}
On the other hand,  $f_{\delta}$ satisfies
\begin{equation}
\partial_{t}f_{\delta}=(a+bi)\Big(e^{\varphi_{1}}(\Delta_{g}+V)e^{-\varphi_{1}}f_{\delta}+e^{\varphi_{1}}h_\delta\Big),    
\end{equation}
in which
\begin{equation}
    e^{\varphi_{1}}h_{\delta}=-e^{\varphi_1}[\Delta_{g},\eta]u=-e^{\varphi_1}\Big(\Delta_{g}\eta u+2\langle\nabla_{g}\eta,\nabla_{g} u\rangle_{g}\Big)\in C_{c}^{\infty}([0,1]\times(B_{2}\backslash B_{1})).
\end{equation}
According to the equation of $f_{\delta}$, we have
\begin{align}
    \Big|\partial_{t}f_{\delta}-\mathcal{S}f_{\delta}-\mathcal{A}f_{\delta}\Big|\leq \|V\|_{L^\infty}|f_{\delta}|+\sqrt{a^2+b^2}|e^{\varphi_{1}}h_{\delta}|.
\end{align}

We denote a quantity $M_{2,a,\delta}$ by
\begin{align}
    M_{2,a,\delta}=\sup_{t\in[0,1]}\frac{\|e^{\varphi_{1}}h_{\delta}\|_{L^2}}{\|f_{\delta}\|_{L^2}}.
\end{align}
If $f_\delta$ has non-trivial $L^2$ mass, we claim that $M_{2,a,\delta}<\infty$ since
\begin{equation}
    \|e^{\varphi_{1}}h_{\delta}\|_{L^{2}}\leq C_{\eta}e^{4\gamma}(\|u\|_{L^2(B_{2}\backslash B_{1})}+\|\nabla_{g} u\|_{L^{2}(B_{2}\backslash B_{1})})\leq C_{\eta,\gamma},
\end{equation}
which the constant is independent of $\delta$ and $a$. 

Notice that $\operatorname{Hess}(\rho^{2-2\delta})$ is positive, a direct computation gives
\begin{align}
    \int_{\Bbb H^n}[\mathcal{S},\mathcal{A}]f_{\delta}\overline{f_{\delta}}\,\dd{\rm Vol}=&(a^{2}+b^{2})\Big(\int_{\Bbb H^n}4\operatorname{Hess}(\varphi_1)(\nabla_{g}\varphi_1,\nabla_{g}\varphi_1)|f|^{2}\,\dd{\rm Vol}-\int_{\Bbb H^n}\Delta_{g}^2\varphi_{1}|f|^{2}\,\dd{\rm Vol}\\&+\int_{M}4\operatorname{Hess}(\varphi_{1})(\nabla_{g} f,\nabla_{g} f)\,\dd{\rm Vol}\Big)\\ \geq&-(a^{2}+b^{2})\int_{\Bbb H^n}|\Delta_{g}^{2}\varphi_{1}||f|^2\,\dd{\rm Vol}\\ \geq& -(a^{2}+b^{2})\int_{\Bbb H^n}\gamma\mathfrak{D}_{n}|f|^2\,\dd{\rm Vol},
\end{align}
where $\mathfrak{D}_{n}$ do not depend on $\delta$. Thus, we get the logarithmic convexity inequality
\begin{align}
    \|e^{\gamma\rho^{2-2\delta}}u\eta\|_{L^2_x}\leq e^{N((a^2+b^2)(\gamma\mathfrak{D}_{n}+M_1^2+M_2^2)+\sqrt{a^2+b^2}(M_1+M_2))}\|e^{\gamma\rho^{2-2\delta}}\eta u(0)\|^{1-t}_{L_x^2}\|e^{\gamma\rho^{2-2\delta}}\eta u(1)\|^{t}_{L_x^2},
\end{align}
which implies 
\begin{equation}
   \sup_{t\in[0,1]} \|e^{\gamma\rho^{2-2\delta}}u(t)\|_{L^2_x}\leq C,
\end{equation}
where $C$ is independent on $\delta$. Combining with the Lebesgue convergence theorem, we have 
\begin{equation}
    \sup_{t\in[0,1]}\|e^{\gamma\rho^{2}}u(t)\|_{L_x^{2}}\leq C.
\end{equation}
Hence, we complete the proof of  Lemma \ref{log convex lemma}. \end{proof}
\begin{lemma}\label{log convex lemma2}
    Under the same condition as in Lemma \ref{log convex lemma}, we have the following space-time estimate
\begin{align}
     &2\gamma(a^{2}+b^{2})\|\sqrt{t(1-t)}e^{\gamma\rho^2}\nabla_{g} u\|_{L_{t,x}^2}^{2}+16\gamma^3(a^{2}+b^{2})\int_{0}^{1}\int_{\Bbb H^n}t(1-t)(\rho^2+\rho^{3}\coth\rho)|e^{\gamma\rho^{2}}u|^{2}\,\dd{\rm Vol}\dd t\\ \leq&M_{3}\sup_{t\in[0,1]}\|e^{\gamma\rho^{2}}u\|_{L^2}^{2}+M_4\sup_{t\in[0,1]}\|e^{\gamma\rho^{2}}F\|^{2}_{L^{2}}\label{t and 1-t}
\end{align}
where $M_{3}=(M_{1}^{2}+\frac{1}{6}+2\mathfrak{C}_{n})(a^{2}+b^{2})+3$ and $M_{4}=\frac{7}{6}(a^{2}+b^{2})$.
\end{lemma}
\begin{proof}
Let $v=e^{\varphi}u$ with $\varphi=\gamma\rho^2$, recall that $v$ solves
$$  \partial_{t}v=(\mathcal{S}+\mathcal{A})v+(a+ib)(V(x)v+e^{\varphi}F),$$
where $\mathcal{S}$ and $\mathcal{A}$ are the same as Lemma \ref{technical-lemma}.  From this equation, we have
\begin{equation}\label{formula of sym anti}
    |\partial_{t}v-(\mathcal{S}+\mathcal{A})v|\leq\sqrt{a^{2}+b^{2}}\Big(M_{1}v+e^{\varphi}|F|\Big).
\end{equation}
We now apply Lemma \ref{abstract lemma1}. Let
\[
W=\sqrt{a^{2}+b^{2}}e^{\varphi}F.
\]
For $H(t)=(v,v)_{L^2}$, following the strategy in \cite{EKPV-JEMS}, we calculate
\begin{align}
    \partial_{t}^{2}H(t)=&2\partial_{t}\operatorname{Re}(\partial_{t}v-\mathcal{S}v-\mathcal{A}v,v)+2(\mathcal{S}_{t}v+[\mathcal{S},\mathcal{A}]v,v)\\&+\|\partial_{t}v-\mathcal{A}v+\mathcal{S}v\|_{L_x^2}^{2}-\|\partial_{t}v-\mathcal{A}v-\mathcal{S}v\|_{L_x^2}^{2}\\
    \geq&2\partial_{t}\operatorname{Re}(\partial_{t}v-\mathcal{S}v-\mathcal{A}v,v)+2(\mathcal{S}_{t}v+[\mathcal{S},\mathcal{A}]v,v)\\&-\|\partial_{t}v-\mathcal{A}v-\mathcal{S}v\|_{L_x^2}^{2}\\
    &\stackrel{\triangle}{=}\mathfrak{R}_1+\mathfrak{R}_2+\mathfrak{R}_3.
\end{align}
Multiplying $\partial_t^2 H(t)$ by $t(1-t)$  and integrating by parts,  we obtain
\begin{align}
    \int_{0}^{1}t(1-t)\frac{d^{2}}{dt^{2}}H(t)\,\dd t=H(1)+H(0)-2\int_{0}^{1}H(t)\,\dd t\leq 2\sup_{0\leq t\leq1}\|v(\cdot,t)\|_{L^{2}}^{2},
\end{align}
where we use the fact  $H(t)\geq0$ in the last inequality. For $\mathfrak{R}_{1}$, integrating by parts infers  
\begin{align}
  \int_{0}^{1}\int_{\mathbb{H}^n}t(1-t)  \mathfrak{R}_{1}\,\dd{\rm Vol}\dd t=&2\int_{0}^{1}\int_{\Bbb H^n}t(1-t)\partial_{t}\operatorname{Re}(\partial_{t}v-\mathcal{S}v-\mathcal{A}v)\bar{v}\,\dd{\rm Vol}\dd t\\
    =&-2\int_{0}^{1}\int_{\Bbb H^n}(1-2t)\operatorname{Re}(\partial_{t}v-\mathcal{S}v-\mathcal{A}v)\bar{v}\,\dd{\rm Vol}\dd t\\
    \geq&-\sup_{0\leq t\leq1}\|\partial_{t}v-\mathcal{S}v-\mathcal{A}v\|_{L^2}^{2}-\sup_{0\leq t\leq1}\|v(\cdot,t)\|_{L^2}^{2}\\
    \geq&-\Big((a^{2}+b^{2})M^{2}_{1}+1\Big)\sup_{0\leq t\leq1}\|v(\cdot,t)\|_{L^{2}}^{2}-(a^{2}+b^{2})\sup_{0\leq t\leq1}\|e^{\gamma\rho^{2}}F\|_{L^2}^{2}.
\end{align}
For $\mathfrak{R}_{2}$, we obtain the lower bound
\begin{align}
    &\int_{0}^{1}\int_{\mathbb{H}^n}t(1-t)  \mathfrak{R}_{2}\,\dd{\rm Vol}\dd t\\=&2\int_{0}^{1}\int_{\Bbb H^n}t(1-t)(\mathcal{S}_{t}+[\mathcal{S},\mathcal{A}])v\bar{v}\,\dd{\rm Vol}\dd t\\ \geq&-2(a^{2}+b^{2})\mathfrak{C}_{n}\sup_{t\in[0,1]}\|v\|^{2}_{L^2}+8(a^{2}+b^{2})\int_{0}^{1}\int_{\Bbb H^n}t(1-t)|\nabla_{g}v|_{g}^{2}\,\dd{\rm Vol}\dd t\\&+32(a^{2}+b^{2})\int_{0}^{1}\int_{\Bbb H^n}t(1-t)(\rho^2+\rho^{3}\coth\rho)|v|^{2}\,\dd{\rm Vol}\dd t
\end{align}
To handle  $\mathfrak{R}_{3}$, we  use \eqref{formula of sym anti} to deduce
\begin{align}
    &-\int_{0}^{1}t(1-t)\|\partial_{t}v-\mathcal{S}v-\mathcal{A}v\|_{L^{2}}^{2}\,\dd t\geq-\sup_{0\leq t\leq 1}\|\partial_{t}v-\mathcal{S}v-\mathcal{A}v\|_{L^{2}}^{2}\int_{0}^{1}t(1-t)\,\dd t\\\geq&-\frac{1}{6}(a^{2}+b^{2})M_{1}^{2}\sup_{0\leq t\leq1}\|v(t,\cdot)\|^{2}_{L^{2}}-\frac{1}{6}(a^{2}+b^{2})\sup_{0\leq t\leq1}\|e^{\gamma\rho^{2}}F(\cdot,t)\|^{2}_{L^{2}}.
\end{align}
Therefore, we obtain 
\begin{align}
  &  8\gamma(a^{2}+b^{2})\int_{0}^{1}\int_{\Bbb H^n}t(1-t)|\nabla_{g}v|_{g}^{2}\,\dd{\rm Vol}\dd t\\&+32\gamma^3(a^{2}+b^{2})\int_{0}^{1}\int_{\Bbb H^n}t(1-t)(\rho^2+\rho^{3}\coth\rho)|v|^{2}\,\dd{\rm Vol}\dd t\\\leq& M_{3}\sup_{t\in[0,1]}\|v\|_{L^2}^{2}+M_4\sup_{t\in[0,1]}\|e^{\gamma\rho^{2}}F\|^{2}_{L^{2}}
\end{align}
with $M_{3}=(M_{1}^{2}+\frac{1}{6}+2\mathfrak{C}_{n})(a^{2}+b^{2})+3$, $M_{4}=\frac{7}{6}(a^{2}+b^{2})$.
Since $u=e^{-\gamma\rho^{2}}v$, we have
\begin{equation}
\nabla_{g} u=e^{-\gamma\rho^{2}}\nabla_{g} v-2\gamma \rho\nabla_{g}\rho e^{-\gamma\rho^{2}}u,
\end{equation}
which means that
\begin{equation}
  \|\sqrt{t(1-t)}e^{\gamma\rho^2}\nabla_{g} u\|_{L^{2}}^{2}\leq2\|\sqrt{t(1-t)}\nabla_{g} v\|^{2}_{L^2}+2\|2\sqrt{t(1-t)}\gamma\rho v\|_{L^2}^2.
\end{equation}
From this inequality, we obtain
\begin{equation}
  4\gamma(a^{2}+b^{2})\|\sqrt{t(1-t)}e^{\gamma\rho^2}\nabla_{g} u\|_{L^2}^{2}\leq  8\gamma(a^{2}+b^{2})\|\sqrt{t(1-t)}\nabla_{g} v\|^2_{L^2}+32\gamma^3(a^2+b^2)\|\sqrt{t(1-t)}\rho v\|^{2}_{L^2}.
\end{equation}
Thus, we obtain the estimates 
\begin{align}
    &2\gamma(a^{2}+b^{2})\|\sqrt{t(1-t)}e^{\gamma\rho^2}\nabla_{g} u\|_{L^2}^{2}+16\gamma^3(a^{2}+b^{2})\int_{0}^{1}\int_{\Bbb H^n}t(1-t)(\rho^2+\rho^{3}\coth\rho)|v|^{2}\,\dd{\rm Vol}\dd t\\ \leq&M_{3}\sup_{t\in[0,1]}\|v\|_{L^2}^{2}+M_4\sup_{t\in[0,1]}\|e^{\gamma\rho^{2}}F\|^{2}_{L^{2}}.
\end{align}
\end{proof}
Similarly, we establish the logarithmic convexity  for solution to \eqref{parabolic regularized equation} posed on the general Riemannian manifold  under Assumption \ref{assum}.
\begin{lemma}\label{asym log}
Assume that $u\in C([0,1],L^{2}(M))\cap L^{2}([0,1],H^{1}(M))$ is a solution to \eqref{parabolic regularized equation} on $M\times[0,1]$ with $a>0$ and $b\in\Bbb R$. Let $V:M\to\Bbb C$, $F:M\times[0,1]\to\Bbb C$ and $\gamma>0$. In addition, we assume that 
\begin{align}
\sup_{t\in[0,1]}\|V\|_{L^{\infty}_{t,x}}:=M_1<\infty,\,\,\sup_{t\in[0,1]}\frac{\|e^{\gamma\rho^2}F(x,t)\|_{L^{2}_{x}}}{\|u\|_{L^{2}_{x}}}:=M_{2}<\infty.
\end{align}
Furthermore, if the following estimate holds
\begin{equation}
  \|e^{\gamma\rho^{2}}u(x,0)\|_{L^{2}_{x}}+\|e^{\gamma\rho^{2}}u(x,1)\|_{L^{2}_{x}}<\infty,
\end{equation}  
then $H(t)=\|e^{\gamma\rho^{2}}u(x,t)\|_{L^{2}_{x}}$ is finite and logarithmically convex in $t\in[0,1]$. Moreover, there exists a constant $N=N(\gamma,a,b)$ such that
\begin{equation}
  H(t)\leq e^{N[(a^{2}+b^{2})(\gamma\mathfrak{F}_{n}+M_1^2+M_2^2)+\sqrt{a^{2}+b^{2}}(M_1+M_2)]}H(0)^{1-t}H(1)^{t},
\end{equation}
where $\mathfrak{F}_{n}$ is the upper bound of $\Delta^{2}_{g}(\rho^{2})$.
\end{lemma}
\begin{proof}
Compared to the standard hyperbolic setting, the formal calculation can be applied with minor changes in estimating $\mathcal{S}_{t}+[\mathcal{S},\mathcal{A}]$. More precisely, we use a Hessian comparison theorem $\operatorname{Hess}(\rho^{2})\geq 2g$ to deal with the Hessian on general Cartan-Hadamard manifold. For the bi-harmonic term $\Delta_g^2(\rho^2)$, we can use Proposition \ref{bi-laplacian} to handle it. 

 The remaining part is to make the formal computation rigorous. To this goal, we should repeat the approximation process as in the hyperbolic case. In particular, we should illustrate that  $\Delta_g^2(\rho^{2-2\delta})$ is bounded uniformly for $\delta\ll1$. By the chain rule, we infer that  
	\begin{equation}\label{Laplace h}
    \Delta_g h(\rho)=h^{\prime\prime}(\rho)+\mathcal H(\rho) h^{\prime}(\rho)
    \end{equation}
    and here $\mathcal H(\rho)=\Delta_g\rho$. Taking $h(\rho)=\rho^{2-2\delta}$ and using \eqref{178537},
it follows
	\begin{align}\label{Laplace-h}
	    \Delta_g h(\rho)=(2-2\delta)(1-2\delta)\rho^{-2\delta}+\mathcal H(\rho)(2-2\delta)\rho^{1-2\delta}.
	\end{align}
     Next, we turn to calculate $\Delta_g^{2}h(\rho)$. For the convenience, we rewrite \eqref{Laplace-h} to
	 \[
	\Delta_g h = C_1 \rho^{-2\delta} + C_2 \mathcal{H}(\rho) \rho^{1-2\delta},
	 \]
	 where $C_1:=(2-2\delta)(1-2\delta)$ and $ C_2:=(2-2\delta)$.
	
	Taking the Laplacian again, we obtain
\begin{align}
	\Delta_g^2 h 
	=& C_1 \Delta_g(\rho^{-2\delta}) + C_2 \Delta_g\big(\mathcal H(\rho)\rho^{1-2\delta}\big)\\
    =& L_1+L_2.
	\end{align}
	
	We will calculate each term separately.
By \eqref{Laplace h}, we have
	$$
L_1=	\Delta_g(\rho^{-2\delta})
	= \partial_\rho^2(\rho^{-2\delta}) + \mathcal H(\rho) \partial_\rho(\rho^{-2\delta})
	= O(\rho^{-1-2\delta}).
$$
	Using the Leibniz rule, it holds
	\[
	\Delta_g\big(\mathcal H \rho^{1-2\delta}\big)
	= (\Delta_g \mathcal H)\rho^{1-2\delta}
	+ 2 \langle \nabla_g\mathcal H, \nabla_g \rho^{1-2\delta} \rangle_g
	+\mathcal H \Delta_g(\rho^{1-2\delta}).
	\]
	By
    \begin{equation}
        \Delta_g \mathcal{H}=\partial^{2}_{\rho}\mathcal{H}+\mathcal{H}\partial_\rho\mathcal{H}+\Delta_{S_\rho}\mathcal{H}
    \end{equation}
    and \eqref{H1}, \eqref{H2}, \eqref{H3}, \eqref{H31}, \eqref{H32}, we deduce that
    \begin{align}
        &\partial_\rho^2\mathcal{H}+\mathcal{H}\partial_\rho\mathcal{H}+\Delta_{S_\rho}\mathcal{H}\\=&2\operatorname{tr}(S^3)+2\langle R(\cdot,\partial_\rho)\partial_\rho,S \rangle_g-\partial_\rho\operatorname{Ric}(\partial_\rho,\partial_\rho)-\mathcal H|S|^2_g-\mathcal H\operatorname{Ric}(\partial_\rho,\partial_\rho)+\sum_{j,k=1}^{n-1}\nabla_{\mathbf{e}_j}\nabla_{\mathbf{e}_{k}}A_{jk}\\&+\frac{1}{2}\frac{\partial R}{\partial\rho}-\frac{\partial}{\partial\rho}\operatorname{Ric}(\partial_{\rho},\partial_\rho)-\mathcal{H}\operatorname{Ric}(\partial_\rho,\partial_\rho)+\operatorname{tr}(A\cdot\operatorname{Ric}|_{\rm tan})
    \end{align}
    Invoking the similar calculations as in the proof of Proposition \ref{bi-laplacian} in Appendix B, one obtains
	\[
	(\Delta_{g}\mathcal{H})\rho^{1-2\delta} = O(\rho^{-m-2\delta}),
	\]
	\[
	\langle \nabla_{g} \mathcal{H}, \nabla_g \rho^{1-2\delta} \rangle
	= O(\rho^{-m-2\delta}),
	\]
	\[
	\mathcal H \Delta_g(\rho^{1-2\delta}) = O(\rho^{-m-2\delta}).
	\]
	Combining all the above estimates, we conclude that
	\[
	\Delta_g^2(\rho^{2-2\delta})
	= O(\rho^{-m-2\delta}),
	\quad \text{as } \rho \to \infty.
	\]
	In particular,
	\[
	\Delta_g^2(\rho^{2-2\delta}) \in L^\infty(M\backslash B_{\rho_0}).
	\] Since $\rho_0$ is bounded, it is direct to estimate $\Delta_g^2(\rho^2)$ inside the geodesic ball.
\end{proof}
Next, we give a weighted space time estimate under asymptotic hyperbolic metric. Since the proof is similar to Lemma \ref{log convex lemma2}, we omit it.
\begin{lemma}
  Under the same condition in Lemma \ref{asym log}, we have the following space-time estimate
  \begin{equation}
  \begin{aligned}
&2\gamma(a^{2}+b^{2})\|\sqrt{t(1-t)}e^{\gamma\rho^{2}}\nabla_{g}u\|^{2}_{L^{2}_{t,x}}+16\gamma^{3}(a^{2}+b^{2})\|\sqrt{t(1-t)}\rho e^{\gamma\rho^{2}}u\|_{L^{2}_{t,x}}
\\\leq& M_5\sup_{t\in[0,1]}\|e^{\gamma\rho^{2}}u\|^{2}_{L^{2}_{x}}+M_{6}\sup_{t\in[0,1]}\|e^{\gamma\rho^{2}}F\|^{2}_{L^{2}_{x}}.
\end{aligned}
  \end{equation}
  where $M_{5}=(M_{1}^{2}+\frac{1}{6}+2\mathfrak{F}_{n})(a^{2}+b^{2})+3$ and $M_{4}=\frac{7}{6}(a^{2}+b^{2})$.
\end{lemma}

For the general manifold under Assumption \ref{assum}, to obtain the unique continuation property under the almost quadratic exponential decay, we need to establish the weighted space-time estimate with a almost quadratic exponential rate. For this purpose, we may use the logarithmic convexity and an asymptotic formula to transfer the quadratic exponential decay to almost quadratic case.

Now we give a lemma, which transforms the quadratic exponential function $e^{\gamma\rho^2}$ to $e^{\sigma\rho^2\log{\rho}}$, for all $\sigma>0$.
\begin{lemma}
	For $\rho,\sigma>0$, we have the asymptotic expansion 
	\begin{equation}\label{I rho}
		I(\rho)=\int_{\Bbb R}e^{\gamma\rho^2-\sigma e^{\frac{\gamma}{\sigma}}}\,\dd\gamma\sim \sqrt{\frac{2\sigma\pi}{\rho}}e^{\sigma\rho^2\log{\rho}-\sigma\rho},\rho\to\infty. 
	\end{equation}
\end{lemma}
\begin{proof}
	Let $\gamma=\sigma\log\rho+\sigma\mathfrak{u}$, then 
	\begin{equation}
		\begin{aligned}
			I(\rho)=\int_{\mathbb{R}}e^{\sigma\gamma\rho^2-\sigma e^{\gamma/\sigma}}\,\dd\gamma=&\int_{\mathbb{R}}e^{\sigma\rho^2\log\rho+\sigma\rho\mathfrak{u}-\sigma\rho e^\mathfrak{u}}\,\dd\mathfrak{u}\\=&\int_{\mathbb{R}}e^{\sigma\rho^2\log\rho+\sigma\rho\mathfrak{u}-\sigma\rho+\sigma\rho-\sigma\rho e^{\mathfrak{u}}}\,\dd\mathfrak{u}\\=&e^{\sigma\rho^2\log\rho-\sigma\rho}\int_{\mathbb{R}}e^{-\sigma\rho(-\mathfrak{u}-1+e^{\mathfrak{u}})}\,\dd\mathfrak{u}.
		\end{aligned}
	\end{equation}
	Then, it reduces to analyze the asymptotic formula of $\int_\mathbb{R}e^{-\sigma\rho(e^{\mathfrak{u}}-\mathfrak{u}-1)}\,\dd\mathfrak{u}$.
	Define the auxiliary function $h(\mathfrak{u})=e^{\mathfrak{u}}-\mathfrak{u}-1$. A direct calculation gives 
	\begin{equation}
		h(0)=h^{\prime}(0)=0,\,h^{\prime\prime}(0)=1.
	\end{equation}
	Since $h(\mathfrak{u})$ tends to $+\infty$ as $\mathfrak{u}\to\pm\infty$ and $h(\mathfrak{u})$ is strictly convex, zero is the unique minimal point of $h(\mathfrak{u})$.  
	
	To get the exact expansion, we decompose the integral into two parts
	\begin{equation}
		\int_{\mathbb{R}}e^{-\sigma\rho(e^{\mathfrak{u}}-\mathfrak{u}-1)}\,\dd\mathfrak{u}=\int_{|\mathfrak{u}|\geq\delta}e^{-\sigma\rho(e^{\mathfrak{u}}-\mathfrak{u}-1)}\,\dd\mathfrak{u}+\int_{|\mathfrak{u}|<\delta}e^{-\sigma\rho(e^{\mathfrak{u}}-\mathfrak{u}-1)}\,\dd\mathfrak{u}:={\rm I}+{\rm II}.
	\end{equation}
	Since $h(\mathfrak{u})\geq\max(h(-\delta),h(\delta)):=m(\delta)>0$, it holds that
	\begin{equation}
		{\rm I}=\int_{|\mathfrak{u}|\geq\delta} e^{\sigma(-\rho+1)h(\mathfrak{u})}e^{-\sigma h(\mathfrak{u})}\,\dd\mathfrak{u}\leq C_{\delta}e^{-\sigma(\rho-1)m(\delta)}.
	\end{equation}
	For the estimates of $\rm II$, using Taylor's expansion, we have 
	\begin{equation}
		e^{\mathfrak{u}}-\mathfrak{u}-1=\sum_{k=2}^{\infty}\frac{\mathfrak{u}^{k}}{k!}=\frac{\mathfrak{u}^2}{2}+\mathfrak{R}(\mathfrak{u}),
	\end{equation}
	where $\mathfrak{R}(\mathfrak{u})=O(\mathfrak{u}^{3})$.
	Change in variable $\mathfrak{t}=\sqrt{\sigma\rho}\mathfrak{u}$, one obtains
	\begin{equation}
		\begin{aligned}
			{\rm II}=&\int_{|\mathfrak{t}|<\sqrt{\sigma\rho}\delta}e^{-\sigma\rho h(\frac{\mathfrak{t}}{\sqrt{\sigma\rho}})}\frac{\dd\mathfrak{t}}{\sqrt{\sigma\rho}}\\=&\frac{1}{\sqrt{\sigma\rho}}\int_{|\mathfrak{t}|<\delta\sqrt{\sigma\rho}}e^{-\frac{\mathfrak{t}^2}{2}-\sigma\rho\mathfrak{R}(\frac{\mathfrak{t}}{\sqrt{\sigma\rho}})}\,\dd\mathfrak{t}.
		\end{aligned}
	\end{equation}
	Since $\mathfrak{R}$ is positive, by Lebesgue Dominated Convergence Theorem,  one can obtain 
	\begin{equation}
		\int_{|\mathfrak{t}|<\delta\sqrt{\sigma\rho}}e^{-\frac{\mathfrak{t}^2}{2}-\sigma\rho\mathfrak{R}(\frac{\mathfrak{t}}{\sqrt{\sigma\rho}})}\,\dd\mathfrak{t}\to\sqrt{2\pi}.
	\end{equation}
	Thus, we can get the desired asymptotic formula.
\end{proof}
Based on this asymptotic expansion, we prove the logarithmic convexity result.
\begin{proposition}\label{cor-log}
	Under the condition of Lemma \ref{log convex lemma}, for $\sigma>0$, the conditions 
	\begin{equation}
		e^{\sigma\rho^{2}\log\rho}u(x,0),  \,\,\,\,\,e^{\sigma\rho^{2}\log\rho}u(x,1)\in L^{2}(M)
	\end{equation}
	will imply that
	\begin{equation}
		e^{\sigma\rho^{2}\log\rho}u(x,t)\in L^{2}(M)
	\end{equation}
	and 
	\begin{align}
		\sqrt{t(1-t)}e^{\sigma\rho^{2}\log\rho}\nabla_{g} u,\,\,\,\,\, \sqrt{t(1-t)}e^{\sigma\rho^{2}\log\rho}\rho u\in L^{2}([0,1],L^{2}(M)).
	\end{align}
	In addition, we also have the logarithmic convex inequality
	\begin{equation}\label{log eta}
		\int_{M}e^{2\sigma\rho^{2}\log\rho}|u(x,t)|^{2}\,\dd{\rm Vol}\leq C\Big(\int_{M}e^{2\sigma\rho^{2}\log\rho}|u(x,0)|^{2}\,\dd{\rm Vol}\Big)^{1-t}\Big(\int_{M}e^{2\sigma\rho^{2}\log\rho}|u(x,1)|^{2}\,\dd{\rm Vol}\Big)^{t}
	\end{equation}
	and the regularity estimate
	\begin{align}\label{t weight}
	&	\big\|\sqrt{t(1-t)}e^{\sigma\rho^{2}\log\rho}\nabla_{g} u\big\|_{L_{t,x}^{2}}^2+\big\|\sqrt{t(1-t)}\rho e^{\sigma\rho^{2}\log\rho} u\big\|_{L_{t,x}^{2}}^2\\ \leq& C\Big(\|e^{\sigma\rho^{2}\log\rho}u(x,0)\|_{L^{2}_{x}}^{2}+\|e^{\sigma\rho^{2}\log\rho}u(x,1)\|_{L^{2}_{x}}^{2}\Big)
	\end{align}
\end{proposition}
\begin{proof}
	From the logarithmic convex inequality \eqref{log convex inequa}, multiplying   $2e^{-\sigma e^{\frac{2\gamma}{\sigma}}}$
	to both side of inequality and integrating with respect to variable $\gamma$, we obtain
	\begin{align}
		&\int_{\gamma_0}^{\infty}\int_{M}2e^{2\gamma\rho^2}e^{-\sigma e^{\frac{2\gamma}{\sigma}}}\,\dd\gamma|u(x,t)|^2\,\dd{\rm Vol}\\ \leq& C\int_{\gamma_0}^{\infty} 2e^{-\sigma e^{\frac{2\gamma}{\sigma}}}\Big(\int_{M}e^{2\gamma\rho^2}|u(x,0)|^2\,\dd{\rm Vol}\Big)^{1-t}\Big(\int_{M}e^{2\gamma\rho^2}|u(x,1)|^2\,\dd{\rm Vol}\Big)^{t}\,\dd\gamma\\=&C\int_{\gamma_0}^{\infty}\Big(\int_{M}2e^{2\gamma\rho^2}e^{-\sigma e^{\frac{2\gamma}{\sigma}}}|u(x,0)|^2\,\dd{\rm Vol}\Big)^{1-t}\Big(\int_{M}2e^{2\gamma\rho^2}e^{-\sigma e^{\frac{2\gamma}{\sigma}}}|u(x,1)|^2\,\dd{\rm Vol}\Big)^{t}\,\dd\gamma.
	\end{align}
	Applying H\"older inequality with respect to variable $\gamma$, one can obtain 
	\begin{align}
		&\int_{M}\int_{\gamma_0}^\infty 2e^{2\gamma\rho^2}e^{-\sigma e^\frac{2\gamma}{\sigma}}\,\dd\gamma|u(x,t)|^2\,\dd{\rm Vol}\\\leq& C\Big( \int_{M}\int_{\gamma_0}^\infty 2e^{2\gamma\rho^2}e^{-\sigma e^\frac{2\gamma}{\sigma}}\,\dd\gamma|u(x,0)|^2\,\dd{\rm Vol}\Big)^{1-t}\\&\times\Big( \int_{M}\int_{\gamma_0}^\infty 2e^{2\gamma\rho^2}e^{-\sigma e^\frac{2\gamma}{\sigma}}\,\dd\gamma|u(x,1)|^2\,\dd{\rm Vol}\Big)^{t},\label{log convex higher}
	\end{align}
	for any $\gamma_0>0$.
	According to 
	\eqref{I rho}, we have 
	\begin{equation}
		\int_{\gamma_0}^{\infty}2e^{2\gamma\rho^2}e^{-\sigma e^\frac{2\gamma}{\sigma}}\,\dd\gamma\leq Ce^{\sigma\rho^2\log\rho}.
	\end{equation}
	Inserting this identity into \eqref{log convex higher}, one can get \eqref{log eta}. As for \eqref{t weight}, the proof is similar and thus we omit here.
\end{proof}

\section{Carleman estimates for the Schr\"odinger and heat operator}
In this section, we will prove the Carleman estimate for both Schr\"odinger operator and heat operator with a $L^\infty$ potential. Different from the flat geometry as in \cite{EKPV-JEMS}, the minimal geodesics between two fixed points $x,y\in\Bbb H^n$ is not straight line. To overcome this difficulty, we introduce the following weight function adapted to the hyperbolic geometry.
Now, we state the Carleman estimate for Schr\"odinger operator.
\begin{theorem}[Carleman estimate $\operatorname{I}$]\label{CarlemanI}
    Let $n\geq2$, and $\textbf{e}_{1}$ is the unit directional vector. Then, for any $\varepsilon>0$, $\mu>0,$ $h=h(x,t)\in C_{0}^{\infty}(\Bbb H^n\times[0,1])$ and $R>4\mu\varepsilon^{-\frac{1}{2}}\mathfrak{C}_{n}$, the following inequality holds:
\begin{equation}
\begin{aligned}
&\frac{R}{4}\sqrt{\frac{\varepsilon}{\mu}}\Big\|e^{\mu d(x,P(t))^{2}-\frac{(1+\varepsilon)R^{2}t(1-t)}{16\mu}}h(x,t)\Big\|_{L^{2}(\Bbb H^n\times[0,1])}\\\leq &\Big\|e^{\mu d(x,P(t))^{2}-\frac{(1+\varepsilon)R^{2}t(1-t)}{16\mu}}(\partial_{t}-i\Delta_{g})h(x,t)\Big\|_{L^{2}(\Bbb H ^n\times[0,1])},
\end{aligned}
\end{equation}
where, 
\begin{equation}\label{Pt}
P(t)=\exp_{\mathbf{0}}\big(-Rt(1-t)\textbf{e}_{1}\big)
\end{equation} is defined by an exponential map, which represents a moving center point along direction $\mathbf{e}_{1}$. 
\end{theorem}
\begin{proof}
Denote by $$\varphi(x,t)=\mu d(x,P(t))^{2}-\frac{(1+\varepsilon)R^{2}t(1-t)}{16\mu},$$ where $d(x,P(t))$ is the geodesic distance between $x$ and $P(t)$. To obtain the lower estimate of \eqref{S t} and \eqref{integral of commutator},  we need the following calculation of  $\varphi(x,t)$. Denote by $\rho=d(x,P(t))$,  Lemma \ref{distant-1} and Lemma \ref{distant-2}, imply that 
\begin{align*}
    \partial_t\rho=\big\langle\nabla_g\rho,\dot P(t)\big\rangle_g
\end{align*}
and
$$\partial_{tt}\rho=\langle\ddot{P},T\rangle+\operatorname{Hess}\rho(\dot{P},\dot{P})=\coth(\rho)R^2(1-2t)^2\Big(1-\langle\nabla_{g}\rho,\textbf{e}_{1}^\parallel\rangle_g^2\Big)+2R\langle \nabla_{g}\rho,\textbf{e}_{1}^{\parallel}\rangle_g,$$
where $\mathbf{e}_1^\parallel$ denotes the parallel transport of $\textbf{e}_1$ and $T=\frac{\partial\Gamma}{\partial s}$ is the tangential component of the geodesics. For convenience, we still denote the parallel transport of $\mathbf{e}_1$
by $\mathbf{e}_{1}$. 
Then, we have 
\begin{align}
    \partial_t\varphi=2\mu\rho\partial_t\rho-\frac{(1+\varepsilon)R^{2}(1-2t)}{16\mu}
    =-2\mu\rho(1-2t)R\langle\textbf{e}_{1},T\rangle_g-\frac{(1+\varepsilon)R^{2}(1-2t)}{16\mu},
\end{align}
where $T(x,t)=\nabla_g\rho$. Taking the time derivative again, it holds
\begin{align}
\partial_{tt}\varphi=&2\mu(\rho_{t}^{2}+\rho\rho_{tt})+\frac{(1+\varepsilon)R^2}{8\mu}\\=&2\mu R^{2}(1-2t)^{2}{\color{red}\langle T,\textbf{e}_{1}\rangle_g^{2}}+2\mu\rho\Big(\coth(\rho)R^2(1-2t)^2\big(1-\langle T,\textbf{e}_{1}\rangle_g^2\big)\Big)\\&+4\mu\rho R\langle T,\textbf{e}_{1}\rangle_g+\frac{R^{2}}{8\mu}+\frac{\varepsilon R^{2}}{8\mu}\\=&:Z_1+Z_2+Z_3+Z_4+Z_5.    
\end{align}

Next, we will compute  $\nabla_g\partial_t\varphi$. Note that $\nabla_g\varphi=2\mu\rho\nabla_g\rho$, then taking the time derivative to $\nabla_g\varphi$, there holds 
 \begin{equation}
    \begin{aligned}
     \nabla_{g}\partial_{t}\varphi=&-2\mu R(1-2t)\nabla_{g}\big(\rho{\color{red}\langle\textbf{e}_{1},T\rangle_g}\big)\\=&-2\mu R(1-2t)(T\langle\textbf{e}_{1},T\rangle_g+\rho\operatorname{Hess}(\rho)(\textbf{e}_{1},T))\\=&-2\mu R(1-2t)\Big(\langle\mathbf{e}_{1},T\rangle_g T+\rho\coth\rho(\mathbf{e}_{1}-\langle\mathbf{e}_{1},T\rangle_g T)\Big),
    \end{aligned}
    \end{equation}
    where we have used the fact that 
    \begin{align*}
        \operatorname{Hess}(\rho)(\textbf{e}_{1},\cdot)=\coth\rho(\langle\textbf{e}_{1},\cdot\rangle_g-\langle\textbf{e}_{1},T\rangle_g\langle T,\cdot\rangle_g)
    \end{align*} and 
    \begin{align*}
      \operatorname{Hess}(\rho)(\textbf{e}_{1},\cdot)^{\#}=\coth\rho(\textbf{e}_{1}-\langle\textbf{e}_{1},T\rangle_g T).
    \end{align*}
Here $\operatorname{Hess}(\rho)(\operatorname{e}_{1},\cdot)$ is defined by
    \begin{equation}
     \langle\operatorname{Hess}(\rho)(\textbf{e}_{1},\cdot)^{\#},\vec{v}\rangle_{g}=\operatorname{Hess}(\rho)(\operatorname{e}_{1},\vec{v}), \vec{v}\in T\mathbb{H}^{ n}.
    \end{equation}
 We still need the following calculations around with $\operatorname{Hess}\varphi$ and $\Delta_g^2\varphi$. For the $\operatorname{Hess}\varphi$, we have 
 \begin{align}
     \operatorname{Hess}(\varphi)(\nabla_{g}\varphi,\nabla_{g}\varphi)&=8\mu^3\rho^2
 \end{align}
 and
 \begin{align}
     \operatorname{Hess}(\varphi)(\nabla_{g} f,\nabla_{g}\bar{f})=&2\mu\rho\coth\rho|\nabla_{g} f|_{g}^{2}+2\mu(1-\rho\coth\rho)|\partial_\rho f|^2\\=&2\mu\rho\coth\rho|\nabla_{g}^{\perp} f|_{g}^{2}+2\mu|\partial_{\rho}f|^{2},\label{Hess-varphi}
 \end{align}
where $\nabla^{\perp}_{g}f=\nabla_g f-\langle\nabla_g f,\partial_\rho\rangle_g\partial_\rho$.
 By \eqref{biharmonic of geodesic}, we have 
    \begin{equation}
        \begin{aligned}
\Delta_{g}^{2}\varphi=&\mu\Delta^{2}_{g}\Big(d(x,P(t))^2\Big)\\=&2\mu(n-1)\Big[(n-1)+(n-3)(1-d(x,P(t))\coth d(x,P(t)))\csch^{2}d(x,P(t))\Big].
\end{aligned}
    \end{equation}
   A direct calculation yields the boundedness of  $\Delta_{g}^{2}\varphi$.

Now, we are in position to prove the Carleman estimate, i.e. Theorem \ref{CarlemanI}.
First, we decompose $\textbf{e}_1$ into 
\begin{align}
\textbf{e}_1=\textbf{e}_1^{1}+\textbf{e}_1^2, 
\end{align}where the parallel component  $\textbf{e}_{1}^{1}=\langle e_{1},T\rangle_g T$ and its perpendicular component is written as $\textbf{e}_{1}^{2}=\textbf{e}_{1}-\langle\textbf{e}_{1},T\rangle_g T$. In addition, the cancellation property  holds: $\big\langle\textbf{e}_{1}^2,T\rangle_g=0$.
Combining with the facts  $\langle T,\nabla_{g} f\rangle_g=\partial_\rho f$ and 
$$\operatorname{Hess}(\rho)(\textbf{e}_{1},\nabla_{g}^{\perp}f)=\coth\rho\langle\textbf{e}_{1}^{2},\nabla_{g}^{\perp
}f\rangle_g,$$
the radial part  reads as
 \begin{align}
     &\Big\langle T\partial_{\rho}f+i\langle T,e_1\rangle\frac{R(1-2t)T}{2}f,T\partial_{\rho}f+i\langle T,e_1\rangle\frac{R(1-2t)T}{2}f\Big\rangle_{g}\\=&|\partial_{\rho}f|^2+\frac{R^{2}(1-2t)^{2}}{4}\langle T,\mathbf{e}_{1}\rangle_g^{2}|f|^{2}+R(1-2t)\langle T,\mathbf{e}_{1}\rangle_g\operatorname{Im}\big(\langle T,\nabla_{g} f\rangle_g\bar{f}\big),
 \end{align}
 while tangential part can be written as
 \begin{align}
&\rho\coth\rho\Big\langle\nabla_{g}^{\perp}f+\frac{i}{2}\textbf{e}_{1}^{2}R(1-2t)f,\nabla_{g}^{\perp}f+\frac{i}{2}\textbf{e}_{1}^{2}R(1-2t)f\Big\rangle_g
     \\=&\rho\coth\rho\Big(|\nabla_{g}^{\perp}f|_{g}^{2}+\frac{R^{2}(1-2t)^{2}}{4}(1-\langle T,\textbf{e}
     _{1}\rangle_g^2)|f|^{2}+R(1-2t)\operatorname{Im}(\langle\textbf{e}_{1}^{2},\nabla_{g}^{\perp} f\rangle_g\bar{f})\Big).\end{align}
In addition, we have
\begin{align}
   &32\mu^3\rho^2+Z_3+Z_{4}\\=&32\mu^3\rho^2+4\mu\rho R\langle T,\textbf{e}_{1}\rangle_g+\frac{R^{2}}{8\mu}=32\mu^{3}\Big(\rho^{2}+\frac{\rho R\langle T,\textbf{e}_{1}\rangle_g}{8\mu^{2}}+\frac{R^{2}}{256 \mu^{4}}\Big)\\ \geq&32\mu^{3}\Big(\rho-\frac{R}{16\mu^{2}}\Big)^2.
\end{align}
 Using \eqref{Hess-varphi}, we obtain  
 \begin{align}
&4\operatorname{Im}\langle\nabla_{g}\varphi_{t},\nabla_{g} f\rangle_{g}\bar{f}+4\operatorname{Hess}(\varphi)(\nabla_{g} f,\nabla_{g}\bar{f})+Z_1+Z_2\\=&8\mu\Big|\partial_{\rho}f T+i\frac{R(1-2t)}{2}f\langle T,\textbf{e}_{1}\rangle_g T\Big|_{g}^{2}+\rho\coth\rho\Big|\nabla_{g}^{\perp}f+\frac{iR(1-2t)}{2}\textbf{e}_{1}^{2}f\Big|_{g}^{2}.
 \end{align}

Then we have the lower bound for the commutator $\mathcal{S}_{t}+[\mathcal{S},\mathcal{A}]$,
\begin{align}
   &\int_{0}^{1}\int_{\Bbb H^n}\Big(\mathcal{S}_{t}+[\mathcal{S},\mathcal{A}]\Big)f\bar{f}\,\dd{\rm Vol}\dd t\\=
&2\int_{0}^{1}\int_{\Bbb H^n}\operatorname{Im}\langle\nabla_{g}\varphi_{t},\nabla_{g} f\rangle_{g}\bar{f}\,\dd{\rm Vol}\dd t+\int_{0}^{1}\int_{\Bbb H^n}\varphi_{tt}|f|^{2}\,\dd{\rm Vol}\dd t\\+&\int_{0}^{1}\int_{\Bbb H^n} 4\operatorname{Hess}(\varphi)(\nabla_{g}\varphi,\nabla_{g}\varphi)|f|^{2}\,\dd{\rm Vol}\dd t-\int_{0}^{1}\int_{\Bbb H^n}\Delta_{g}^{2}\varphi|f|^{2}\,\dd{\rm Vol}\dd t\\&+\int_{0}^{1}\int_{\Bbb H^n}4\operatorname{Hess}(\varphi)(\nabla_{g} f,\nabla_{g}\bar{f})\,\dd{\rm Vol}\dd t+2\operatorname{Im}\int_{0}^{1}\int_{\Bbb H^n} \langle\nabla_{g}\varphi_{t},\nabla_{g}f\rangle_g\bar{f}\,\dd{\rm Vol}\dd t
\\=&8\mu\int_{0}^{1}\int_{\Bbb H^n}\Big|\partial_{\rho}f T+i\langle T,\textbf{e}_{1}\rangle_{g}\frac{R(1-2t)}{2}fT\Big|_{g}^{2}+\rho\coth\rho\Big|\nabla_{g}^{\perp}f+\frac{iR(1-2t)}{2}\textbf{e}_{1}^{2}f\Big|_{g}^{2}\,\dd{\rm Vol}\dd t\\&+\int_{0}^{1}\int_{\Bbb H^n}32\mu^{3}\Big(\rho-\frac{R}{16\mu^{2}}\Big)^{2}|f|^{2}\,\dd{\rm Vol}\dd t+\int_{0}^{1}\int_{\Bbb H^n}\Big(\frac{\varepsilon R^2}{8\mu}-\mu\mathfrak{C}_{n}\Big)|f|^{2}\,\dd{\rm Vol}\dd t\\
\geq&\int_{0}^{1}\int_{\Bbb H^n}\Big(\frac{\varepsilon R^2}{8\mu}-\mu\mathfrak{C}_{n}\Big)|f|^{2}\,\dd{\rm Vol}\dd t.
\end{align}
This inequality implies the desired Carleman inequality
\begin{align}\label{Carleman for schrodinger}
    \frac{R}{4}\sqrt{\frac{\varepsilon}{\mu}}\Big\|e^{\varphi}h\Big\|_{L^2}\leq\Big\|e^{\varphi}(i\partial_{t}+\Delta_g)h\Big\|_{L^2}
\end{align}
provided that $\frac{\varepsilon R^{2}}{8\mu}\geq 2\mu\mathfrak{C}_{n}$. Hence, we complete the proof of Theorem \ref{CarlemanI}.
\end{proof}

Next, we establish the Carleman estimate for the heat operator. 
\begin{theorem}[Carleman estimate $\operatorname{II}$]\label{Carleman-parabolic-t}
    Let $n\geq2$, and $\mathbf{e}_{1}$ be the unit directional vector. Then for any $\varepsilon>0$, $\mu>0$, $h=h(x,t)\in C_{0}^{\infty}(
    \Bbb H^n\times[0,1])$ and $R>4\mu\varepsilon^{-\frac{1}{2}}\mathfrak{C}_{n}$, the following Carleman estimates holds
    \begin{equation}\label{carleman parabolic}
    \begin{aligned}
        &\frac{R}{4}\sqrt{\frac{\varepsilon}{\mu}}\Big\|e^{\mu d(x,P(t))^2+\frac{R^{2}t(1-t)(1-2t)}{6}-\frac{(1+\varepsilon)R^{2}t(1-t)}{16\mu}}h(x,t)\Big\|_{L^{2}(\mathbb{H}^{n}\times[0,1])}\\\leq& \Big\|e^{\mu d(x,P(t))^2+\frac{R^{2}t(1-t)(1-2t)}{6}-\frac{(1+\varepsilon)R^{2}t(1-t)}{16\mu}}(\partial_{t}-\Delta_{g})h(x,t)\Big\|_{L^{2}(\mathbb{H}^{n}\times[0,1])}
        \end{aligned}
    \end{equation}
    where $P(t):=\operatorname{exp}_{\mathbf{0}}(-Rt(1-t)\mathbf{e}_1)$ is the moving center point.
\end{theorem}
\begin{proof}
Now, we give the lower bound for commutator of heat operator. Different from the Schr\"odinger operator, we choose the new weight function 
\begin{equation}
    \varphi=\mu d(x,P(t))^2+\frac{R^{2}t(1-t)(1-2t)}{6}-\frac{(1+\varepsilon)R^{2}t(1-t)}{16\mu},
\end{equation}
where $P(t)$ is defined as in \eqref{Pt}.
Similar to the Schr\"odinger operator, a calculations on the derivatives of weight function $\varphi$ gives
\begin{align*}
  \nabla_{g}\varphi_{t}=-2\mu R(1-2t)\Big(\langle\mathbf{e}_{1},T\rangle_{g} T+\rho\coth\rho\big(\mathbf{e}_{1}-\langle\mathbf{e}_{1},T\rangle_{g}T\big)\Big)  
\end{align*}
 and
\begin{equation}\begin{aligned}\varphi_{tt}=&2\mu R^{2}(1-2t)^{2}\langle T,\mathbf{e}_{1}\rangle_{g}^{2}+2\mu\rho\Big(\coth(\rho)R^{2}(1-2t)^{2}(1-\langle T,\mathbf{e}_{1}\rangle_{g}^{2})\Big)\\&+4\mu\rho R\langle T,\mathbf{e}_{1}\rangle_{g}+R^{2}(2t-1)+\frac{R^2}{8\mu}+\frac{\varepsilon R^{2}}{8\mu}\\=&\sum_{j=1}^{6}K_j.\end{aligned}\end{equation}

To obtain a positive lower bound, we need a refined completion of squares that treats the radial and tangential parts separately, according to the commutator structure. Firstly, combining the fact $\rho\coth\rho\geq1$, we complete the square with respect to the tangential component
\begin{align}
 &K_2+(1-\langle T,\mathbf{e}_{1}\rangle_{g}^{2})K_{4}+(1-\langle T,\mathbf{e}_{1}\rangle_{g}^2)K_{5} \\=&  2\mu\rho\coth\rho R^{2}(1-2t)(1-\langle\mathbf{e}_{1},T\rangle_{g}^{2} )+(1-\langle T,\mathbf{e}_{1}\rangle_{g}^2)R^{2}(2t-1)+\frac{R^2}{8\mu}(1-\langle T,\mathbf{e}_{1}\rangle_{g}^2)\\ \geq&(1-\langle T,\mathbf{e}_{1}\rangle_{g}^{2})\Big(2\mu R^{2}(1-2t)^{2}+R^{2}(2t-1)+\frac{R^{2}}{8\mu}\Big)\\
 \geq&(1-\langle T,\mathbf{e}_{1}\rangle_{g}^{2})2\mu\Big(R(2t-1)+\frac{R}{4\mu}\Big)^{2}.
\end{align}

Next, we complete the square for the radial part.
\begin{align}
   & 32\mu^{3}\rho^{2}-16\mu^{2}\rho R(1-2t)\langle\mathbf{e}_{1},T\rangle_g+K_1+K_3+\langle\mathbf{e}_{1},T\rangle_g ^{2}K_{4}+\langle\mathbf{e}_{1},T\rangle_g^{2} K_{5}\\=&32\mu^{3}\Big(\rho^{2}-\frac{1}{2\mu}\rho R(1-2t)\langle\mathbf{e}_{1},T\rangle_g+\Big(\frac{R(1-2t)}{4\mu}\langle T,\mathbf{e}_{1}\rangle_g\Big)^{2}+\frac{1}{8\mu^{2}}\rho R\langle T,\mathbf{e}_{1}\rangle_g\\&+R^{2}\frac{2t-1}{32\mu^{3}}\langle T,\mathbf{e}_{1}\rangle_g^{2}+\Big(\frac{R}{16\mu^{2}}\langle T,\mathbf{e}_{1}\rangle_g\Big)^{2}\Big)\\
   =&32\mu^{3}\Big(\rho^{2}+\frac{-4\mu(1-2t)+1}{8\mu^{2}}\rho R\langle T,\mathbf{e}_{1}\rangle_g+\Big(\frac{R(2t-1)}{4\mu}\langle T,\mathbf{e}_{1}\rangle_g+\frac{R}{16\mu^{2}}\langle T,\mathbf{e}_{1}\rangle_g\Big)^{2}\Big)\\=&32\mu^{3}\Big(\rho+R\langle T,\mathbf{e}_{1}\rangle_g\frac{4\mu (2t-1)+1}{16\mu^{2}}\Big)^{2}.
\end{align}
Thus, according to the above analysis, we get the lower bound 
    \begin{align}\label{parabolic virial}
  &\int_{0}^{1}\int_{\Bbb H^n}\mathcal{S}_{t}f\bar{f}\,+[\mathcal{S},\mathcal{A}]f\bar{f}\,\dd{\rm Vol}\dd t\\=&4\int_{0}^{1}\int_{\Bbb H^n}\langle\nabla_{g}\varphi,\nabla_{g}\varphi_{t}\rangle_{g}|f|^2\,\dd{\rm Vol}\dd t+\int_{0}^{1}\int_{\Bbb H^n}\varphi_{tt}|f|^{2}\,\dd{\rm Vol}\dd t\\&+\int_{0}^{1}\int_{\Bbb H^n} 4\operatorname{Hess}(\varphi)(\nabla_{g}\varphi,\nabla_{g}\varphi)|f|^{2}\,\dd{\rm Vol}\dd t-\int_{0}^{1}\int_{\Bbb H^n}\Delta_{g}^{2}\varphi|f|^{2}\,\dd{\rm Vol}\dd t\\&+\int_{0}^{1}\int_{\Bbb H^n}4\operatorname{Hess}(\varphi)(\nabla_{g}f,\nabla_{g}\bar{f})\,\dd{\rm Vol}\dd t\\&
  \geq\int_{0}^{1}\int_{\Bbb H^n}(1-\langle T,\mathbf{e}_{1}\rangle_{g}^{2})2\mu\Big(R(2t-1)+\frac{R}{4\mu}\Big)^{2}|f|^{2}\,\dd x\dd t\\&+\int_{0}^{1}\int_{\Bbb H^n}32\mu^{3}\Big(\rho+R\langle T,\mathbf{e}_{1}\rangle_{g}\frac{4\mu (2t-1)+1}{16\mu^{2}}\Big)^{2}|f|^{2}\,\dd{\rm Vol}\dd t\\&+\int_{0}^{1}\int_{\Bbb H^n}8\mu\rho\coth\rho|\nabla^{\perp}_{g} f|_{g}^{2}+8\mu|\partial_{\rho}f|^{2}\,\dd{\rm Vol}\dd t+\int_{0}^{1}\int_{\Bbb H^n}\Big(\frac{\varepsilon R^{2}}{8\mu}-\mathfrak{C}_{n}\Big)|f|^{2}\,\dd{\rm Vol}\dd t\\ \geq&\frac{\varepsilon R^2}{16\mu}\int_{0}^{1}\int_{\Bbb H^n}|f|^{2}\,\dd{\rm Vol}\dd t.
\end{align}
We then obtain the desired estimate \eqref{carleman parabolic}, which complete the proof of Theorem \ref{Carleman-parabolic-t}.
\end{proof}

\section{Proof of Theorem \ref{thm1} and Theorem \ref{thm2}}\label{section proof of uniqueness}
In this section, we prove the unique continuation property for solutions of  Schr\"odinger equation and parabolic equation,  i.e. Theorem \ref{thm1} and Theorem \ref{thm2}. The proof  now follows from an application of Carleman estimates.

Let $\theta_{M}(x)$ and $\eta_{R}(t)$ be the cutoff function satisfying $\theta_{M}=1$ when $\rho\leq M$, $\theta_{M}=0$ when $\rho\geq 2M\geq 2R$ and $\eta_{R}\in C_{0}^{\infty}(0,1)$, $0\leq\eta_{R}\leq1$, $\eta_{R}=1$ in $[\frac{1}{R},1-\frac{1}{R}]$ and $\eta_{R}=0$ in $[0,\frac{1}{2R}]\cup[1-\frac{1}{2R},1]$. Then denote
\begin{equation}
    h(x,t)=\theta_{M}(x)\eta_{R}(t)u(x,t),
\end{equation}
which is compactly supported in $\Bbb H^n\times(0,1)$.
\subsection{Proof of Theorem \ref{thm1}}
For equation \eqref{eq1}, by direct calculation, we deduce 
\begin{equation}
\begin{aligned}
    \partial_{t}h-i(\Delta_{g} h+Vh)=&\theta_{M}(x)\eta_{R}^{\prime}(t)u(x,t)+\theta_{M}(x)\eta_{R}(t)\partial_{t}u(x,t)-i\eta_{R}(t)\Big(\Delta_{g} \theta_{M}u\\&+2\langle\nabla_{g}\theta_{M},\nabla_{g} u\rangle_{g}+\theta_{M}(x)\Delta_{g}u\Big)-iV\theta_{M}(x)\eta_{R}(t)u(x,t)\\=&\theta_{M}(x)\eta_{R}^{\prime}(t)u(x,t)-i\Big(2\langle\nabla_{g}\theta_{M},\nabla_{g}u\rangle_{g}+u\Delta_{g}(\theta_{M}(x))\Big)\eta_{R}.\label{cutoff equation}
\end{aligned}
\end{equation}
Now, we choose $\mu\leq\frac{\gamma}{1+\varepsilon}$ for some fixed small $\varepsilon$.
The first term on the right hand side of \eqref{cutoff equation} is supported in a set in which
\begin{align}
\mu\rho(x,t)^{2}\leq&\mu d(x,\mathbf{0})^2+\mu R^2t^2(1-t)^2+2\mu d(x,\mathbf{0})Rt(1-t)\\ \leq&\mu(1+\varepsilon)d(x,\mathbf{0})^{2}+(\frac{1}{\varepsilon}+1)\mu=\gamma d(x,\mathbf{0})^{2}+\frac{\gamma}{\varepsilon}.
\end{align}
The second term on the right hand side of \eqref{cutoff equation} is supported in $B_{2M}\backslash B_{M}\times[\frac{1}{2R},1-\frac{1}{2R}]$, where
\begin{equation}
    \mu\rho(x,t)^{2}\leq \gamma d(x,\mathbf{0})^2+\gamma\frac{R^{2}}{\varepsilon}
\end{equation}
and $B_M$ denote the geodesic ball with radius $M$.
Using \eqref{cutoff equation} and the Carleman estimate \eqref{Carleman for schrodinger}, we obtain
\begin{align}
    \frac{R}{4}\sqrt{\frac{\varepsilon}{\mu}}\Big\|e^{\mu\rho^{2}+\beta(t)}h\Big\|_{L^{2}(\Bbb H^n\times[0,1])}\leq 
    &\|V\|_{L^\infty}\Big\|e^{\mu\rho^{2}+\beta(t)}h\Big\|_{L^2(\Bbb H^n\times[0,1])}+Re^{\frac{\gamma}{\varepsilon}}\Big\|e^{\gamma\rho^{2}}u\Big\|_{L^{2}(\Bbb H^n\times[0,1])}\\
    &+M^{-1}e^{\gamma\frac{R^{2}}{\varepsilon}}\Big\|e^{\gamma\rho^{2}}(|u|+|\nabla_{g} u|_{g})\Big\|_{L^{2}(\Bbb H^n\times[\frac{1}{2R},1-\frac{1}{2R}])},\label{incomplete carleman}
\end{align}
where $\beta(t)=-\frac{(1+\varepsilon)R^2t(1-t)}{16\mu}$.
For large $R$, it holds that $\frac{R}{4}\sqrt{\frac{\varepsilon}{\mu}}\geq 2\|V\|_{L^{\infty}}$, and thus the first term on right hand side of \eqref{incomplete carleman} can be absorbed by the left. Since we also have the fact 
\begin{equation}
    N_{\gamma}:=\sup_{t\in[0,1]}\|e^{\gamma\rho^{2}}u\|_{L^2}+\|\sqrt{t(1-t)}e^{\gamma\rho^{2}}\nabla_{g} u\|_{L^{2}}<\infty,
\end{equation}
the third term tends to $0$ as $M\to\infty$, which implies that the third term on right hand side of \eqref{incomplete carleman} can be neglected.
In $B_{\varepsilon(1-\varepsilon)^{2}\frac{R}{4}}\times[\frac{1-\varepsilon}{2},\frac{1+\varepsilon}{2}]$, we have the fact
\begin{align}
   & \mu\rho(x,t)^{2}-\frac{(1+\varepsilon)R^{2}t(1-t)}{16\mu}\\\geq&\frac{R^{2}}{16\mu}\Big(16\frac{\mu^{2}}{R^{2}}\rho^{2}-\frac{1}{4}(1+\varepsilon)\Big)\\
    \geq&\frac{R^{2}}{16\mu}\Big(16\frac{\mu^{2}}{R^{2}}(d(x,\mathbf{0})^{2}-2d(x,\mathbf{0})Rt(1-t)+R^{2}t^{2}(1-t)^{2})-\frac{1}{4}(1+\varepsilon)\Big)\\ \geq&\frac{R^{2}}{16\mu}\Big(\mu^{2}\big((1-\varepsilon^{2})^{2}-\frac{1}{2}\varepsilon(1-\varepsilon)^{2}\big)-\frac{1}{4}(1+\varepsilon)\Big).
\end{align}
Consequently, by choosing $\mu^{2}>\frac{1}{4}\frac{1+\varepsilon}{(1-\varepsilon^{2})^{2}-\frac{1}{2}\varepsilon(1-\varepsilon)^{2}}$, we obtain  
\begin{equation}
    \mu\rho^{2}-\frac{(1+\varepsilon)R^{2}t(1-t)}{16\mu}\geq \mathcal{C}_{\mu,\varepsilon}R^{2}.
\end{equation}
Then, the Carleman estimates reduces to 
\begin{equation}
    \frac{R}{8}\sqrt{\frac{\varepsilon}{\mu}}e^{\mathcal{C}_{\mu,\varepsilon}R^{2}}\|u\|_{L^{2}(B_{\varepsilon(1-\varepsilon)^{2}\frac{R}{4}}\times[\frac{1-\varepsilon}{2},\frac{1+\varepsilon}{2}])}\leq CRe^{\frac{\gamma}{\varepsilon}}.
\end{equation}
In addition, we have 
\begin{equation}
    \|u(0)\|_{L^{2}(\Bbb H^n)}\sim_{\|V\|_{L^{\infty}}}\|u(t)\|_{L^{2}(\Bbb H^n)}
\end{equation}
and 
\begin{equation}
    \|u(t)\|_{L^{2}(\Bbb H^n)}\leq \|u(t)\|_{L^{2}(B_{\varepsilon(1-\varepsilon)^{2}\frac{R}{8}})}+e^{-\gamma\varepsilon^{2}(1-\varepsilon)^{4}\frac{R^{2}}{64}}N_{\gamma}.
\end{equation}
These two inequalities imply that
\begin{equation}
    e^{\mathcal{C}_{\gamma,\varepsilon}R^{2}}\|u(0)\|_{L^2}\leq N_{\gamma,\varepsilon,V}\Rightarrow u\equiv0.
\end{equation}
\subsection{Proof of Theorem \ref{thm2}}
 For heat equation, the direct computation yields 
\begin{align}
    \partial_{t}h-(\Delta_{g} h+Vh)=&\theta_{M}(x)\eta_{R}^{\prime}(t)u(x,t)+\theta_{M}(x)\eta_{R}(t)\partial_{t}u(x,t)-\eta_{R}(t)\Big(\Delta_{g} \theta_{M}u\\&+2\langle\nabla_{g}\theta_{M},\nabla_{g} u\rangle_{g}+\theta_{M}(x)\Delta_{g}u\Big)-V\theta_{M}(x)\eta_{R}(t)u(x,t)\\=&\theta_{M}(x)\eta_{R}^{\prime}(t)u(x,t)-\Big(2\langle\nabla_{g}\theta_{M},\nabla_{g}u\rangle_{g}+u\Delta_{g}(\theta_{M}(x))\Big)\eta_{R}.\label{cutoff equation2}
\end{align}

Now, we choose $\mu\leq\frac{\gamma}{1+\varepsilon}$ for some fixed small $\varepsilon$.
The first term on the right hand side of \eqref{cutoff equation2} is supported on the region in which
\begin{align}
&\mu\rho(x,t)^{2}+\frac{R^2t(1-t)(1-2t)}{6}-\frac{(1+\varepsilon)R^{2}t(1-t)}{16\mu}\\\leq&\mu d(x,\mathbf{0})^2+\mu R^2t^2(1-t)^2+2\mu d(x,\mathbf{0})Rt(1-t)+|\frac{R^2t(1-t)(1-2t)}{6}|\\ \leq&\mu(1+\varepsilon)d(x,\mathbf{0})^{2}+(\frac{1}{\varepsilon}+1)\mu+\frac{R}{6}(1-\frac{1}{R})^2\\=&\gamma d(x,\mathbf{0})^{2}+\frac{\gamma}{\varepsilon}+\frac{R}{6}.
\end{align}
The second term on the right hand side of \eqref{cutoff equation2} is supported in $B_{2M}\backslash B_{M}\times[\frac{1}{2R},1-\frac{1}{2R}]$, where
\begin{align}
    &\mu\rho(x,t)^{2}+\frac{R^2t(1-t)(1-2t)}{6}-\frac{(1+\varepsilon)R^{2}t(1-t)}{16\mu}\\\leq &\gamma d(x,\mathbf{0})^2+\gamma\frac{R^{2}}{\varepsilon}+\frac{R^{2}}{6}
\end{align}
and $B_M$ denote the geodesic ball with radius $M$.
Using \eqref{cutoff equation2} and the Carleman estimate \eqref{Carleman for schrodinger}, we obtain
\begin{align}\label{incomplete carleman2}
    \frac{R}{4}\sqrt{\frac{\varepsilon}{\mu}}\Big\|e^{\mu\rho^{2}+\beta(t)}h\Big\|_{L^{2}(\Bbb H^n\times[0,1])}\leq &\|V\|_{L^\infty}\Big\|e^{\mu\rho^{2}+\beta(t)}h\Big\|_{L^{2}(\Bbb H^n\times[0,1])}+Re^{\frac{\gamma}{\varepsilon}+\frac{R}{6}}\Big\|e^{\gamma\rho^{2}}u\Big\|_{L^{2}(\Bbb H^n\times[0,1])}\\&+M^{-1}e^{\gamma\frac{R^{2}}{\varepsilon}+\frac{R^2}{6}}\Big\|e^{\gamma\rho^{2}}(|u|+|\nabla_{g} u|_{g})\Big\|_{L^{2}(\Bbb H^n\times[\frac{1}{2R},1-\frac{1}{2R}])},
\end{align}
where for convenience we denote $\beta(t)=\frac{R^{2}t(1-t)(1-2t)}{6}-\frac{(1+\varepsilon)R^2t(1-t)}{16\mu}$.
For large $R$, it holds that $\frac{R}{4}\sqrt{\frac{\varepsilon}{\mu}}\geq 2\|V\|_{L^{\infty}}$, and thus the first term on right hand side of \eqref{incomplete carleman2} can be absorbed by the left. In addition, according to \eqref{log convex lemma} and \eqref{log convex lemma2}, we have 
\begin{equation}
    N_{\gamma}:=\sup_{t\in[0,1]}\|e^{\gamma\rho^{2}}u\|_{L^2}+\|\sqrt{t(1-t)}e^{\gamma\rho^{2}}\nabla_{g} u\|_{L^{2}}<\infty,
\end{equation}
which implies that the third term tends to $0$ as $M\to\infty$, and the third term on the right hand side of \eqref{incomplete carleman2} can be neglected.
In $B_{\varepsilon(1-\varepsilon)^{2}\frac{R}{4}}\times[\frac{1-\varepsilon}{2},\frac{1+\varepsilon}{2}]$, we have the fact
\begin{align}
   & \mu\rho(x,t)^{2}+\frac{R^{2}t(1-t)(1-2t)}{6}-\frac{(1+\varepsilon)R^{2}t(1-t)}{16\mu}\\\geq&\frac{R^{2}}{16\mu}\Big(16\frac{\mu^{2}}{R^{2}}\rho^{2}-\frac{1}{4}(1+\varepsilon)+\frac{t(1-t)(1-2t)}{6}\mu\Big)\\
    \geq&\frac{R^{2}}{16\mu}\Big(16\frac{\mu^{2}}{R^{2}}(d(x,\mathbf{0})^{2}-2d(x,\mathbf{0})Rt(1-t)+R^{2}t^{2}(1-t)^{2})-\frac{1}{4}(1+\varepsilon)-\frac{\varepsilon-\varepsilon^3}{384}\mu\Big)\\ \geq&\frac{R^{2}}{16\mu}\Big(\mu^{2}\big((1-\varepsilon^{2})^{2}-\frac{1}{2}\varepsilon(1-\varepsilon)^{2}-\varepsilon\big)-\frac{1}{4}(1+\varepsilon)-\frac{1}{4}\frac{\varepsilon(1-\varepsilon^{2})^{2}}{384^{2}}\Big).
\end{align}
Consequently, by choosing $\gamma^2>\mu^{2}>\frac{1}{4}\frac{1+\varepsilon+\frac{\varepsilon(1-\varepsilon^2)^2}{384^2}}{(1-\varepsilon^{2})^{2}-\frac{1}{2}\varepsilon(1-\varepsilon)^{2}}$, we obtain  
\begin{equation}
    \mu\rho^{2}-\frac{(1+\varepsilon)R^{2}t(1-t)}{16\mu}\geq \mathcal{C}_{\mu,\varepsilon}R^{2}.
\end{equation}
Then, the Carleman estimates reduces to 
\begin{equation}
    \frac{R}{8}\sqrt{\frac{\varepsilon}{\mu}}e^{\mathcal{C}_{\mu,\varepsilon}R^{2}}\|u\|_{L^{2}(B_{\varepsilon(1-\varepsilon)^{2}\frac{R}{4}}\times[\frac{1-\varepsilon}{2},\frac{1+\varepsilon}{2}])}\leq CRe^{\frac{\gamma}{\varepsilon}},
\end{equation}
which implies that $u\equiv0$ on $\Bbb H^n\times[\frac{1-\varepsilon}{2},\frac{1+\varepsilon}{2}]$. By logarithmic convexity and continuous of solution $u\in C([0,1], L^2(\Bbb H^n))$, we deduce that $u\equiv0$.  

\section{Proof of Theorem \ref{thm3} under asymptotic hyperbolic metric}
In this section, we prove the unique continuation property for Schr\"odinger equation \eqref{parabolic} on a Cartan-Hadamard manifold under Assumption \ref{assum}. The main ingredient to deduce the  uniqueness is the following Carleman estimate.
Now, we establish the Carleman estimate
\begin{lemma}\label{carleman lemma}
	Suppose that $d\geq2$, then for any $\varepsilon>0$, $\ell\in\mathbb N$ and any $f\in C_{0}^{\infty}(M\backslash B_{\rho_{0}}\times[0,1])$, the following inequalities hold
	\begin{align}\label{carleman estimate}
		&\frac{\mu}{R^{2}}\int_{0}^{1}\int_{M}|\nabla_{g} f|_{g}^{2}\,\dd{\rm Vol}\dd t+\frac{\mu^{3}}{R^{6}}\int_{0}^{1}\int_{M}|\rho f|^{2}\,\dd{\rm Vol}\dd t\\\leq&\int_{0}^{1}\int_{M}\Big|e^{\mu\frac{\rho^2}{R^2}+\mu^{\mathcal{Q}(\ell,R)}\varphi(t)}(\partial_{t}-i\Delta_{g})(e^{-\mu\frac{\rho^2}{R^2}-\mu^{\mathcal{Q}(\ell,R)}\varphi(t)}f)\Big|^2\,\dd{\rm Vol}\dd t
	\end{align} if
	\begin{equation}\label{mu}
		\mu\geq \max\Big(\tfrac{\rho_{0}^{-1}\mathfrak{F}_{n}^{\frac{1}{2}}}{4}R^{2},\Big(\tfrac{\|\varphi_{tt}\|_{L^\infty}}{8\rho^{2}_{0}}\Big)^{\frac{1}{3-\mathcal{Q}(\ell,R)}}R^{\frac{6}{3-\mathcal{Q}(\ell,R)}}\Big),\quad R\gg1.
	\end{equation} Here $\mathfrak{F}_{n}$ is the upper bound of $|\Delta_{g}^{2}(\rho^{2})|$ and $\mathcal{Q}(\ell,R)=3-\frac{6}{2+\frac{\log{\frac{\log{R}}{\ell}}}{\log{R}}}$ satisfying $R^{\frac{6}{3-\mathcal{Q}(\ell,R)}}=\frac{1}{\ell}R^{2}\log{R}$ with any fixed $\ell\in\mathbb N$.
\end{lemma}
\begin{proof} Let $\phi=\mu\frac{\rho^2}{R^2}+\mu^{\mathcal{Q}(\ell,R)}\varphi(t)$ with 
	$\mathcal{Q}(\ell,R)=3-\frac{6}{2+\frac{\log{\frac{\log{R}}{\ell}
		}}{\log R}}$.
	Notice that 
	\begin{align*}
		e^{\phi}(\partial_{t}-i\Delta_{g})\big(e^{-\phi}f\big)=(\partial_{t}-\mathcal{S}-\mathcal{A})f,
	\end{align*}
	so we have
	\begin{align}
		&\big\|(\partial_{t}-\mathcal{S}-\mathcal{A})f\big\|^{2}_{L^{2}(M\times[0,1])}\\
		=&\|(\partial_{t}-\mathcal{A})f\|^{2}_{L^{2}(M\times[0,1])}+\|\mathcal{S}f\|^{2}_{L^{2}(M\times[0,1])}-2\operatorname{Re}\int_{M\times[0,1]}\mathcal{S}f\overline{(\partial_{t}-\mathcal{A})f}\,\dd{\rm Vol}\dd t\\
		\geq&\int_{M\times[0,1]}\big(\partial_{t}\mathcal{S}+[\mathcal{S},\mathcal{A}]\big)f\bar{f}\,\dd{\rm Vol}\dd t.
	\end{align}
Hence, we reduce the proof of Carleman's estimate to the lower estimate for commutator $\partial_{t}\mathcal{S}+[\mathcal{S},\mathcal{A}]$.
	Here, $\mathcal{S}$ and $\mathcal{A}$ are defined by \eqref{formula sym} and \eqref{formula anti} respectively. For Schr\"odinger operator, we set $a=0$, $b=1$. Inserting into \eqref{integral of commutator} and \eqref{S t}, we obtain
	\begin{align}
		&\int_{0}^{1}\int_{M}\big(\partial_{t}\mathcal{S}+[\mathcal{S},\mathcal{A}]\big)f\bar{f}\,\dd{\rm Vol}\dd t\\=& \Bigg(\int_{0}^{1}\int_M 32\frac{\mu^3}{R^{6}}|\rho f|^{2}\,\dd{\rm Vol}\dd t-\frac{\mu}{R^2}\int_{0}^{1}\int_{M}\Delta_{g}^{2}(\rho^2)|f|^{2}\,\dd{\rm Vol}\dd t\\&+\int_{0}^{1}\int_{M}8\frac{\mu}{R^2}|\nabla_{g} f|_{g}^2\,\dd{\rm Vol}\Bigg)+\mu^{\mathcal{Q}(\ell,R)}\int_{0}^{1}\int_M\varphi_{tt}|f|^{2}\,\dd{\rm Vol}\\ \geq&\Bigg(\int_{0}^{1}\int_M 32\frac{\mu^3}{R^{6}}|\rho f|^{2}\,\dd{\rm Vol}\dd t-\rho_{0}^{-2}\frac{\mu}{R^2}\int_{0}^{1}\int_{M}\Delta_{g}^{2}(\rho^2)|\rho f|^{2}\,\dd{\rm Vol}\dd t\\&+\int_{0}^{1}\int_{M}8\frac{\mu}{R^2}|\nabla_{g} f|_{g}^2\,\dd{\rm Vol}\Bigg)-\mu^{\mathcal{Q}(\ell,R)}\|\varphi_{tt}\|_{L^\infty}\rho_{0}^{-2}\int_{0}^{1}\int_M|\rho f|^{2}\,\dd{\rm Vol}. 
	\end{align}
	Since $\mu$ satisfies \eqref{mu}, 
	we deduce that 
	\begin{equation}
		32\frac{\mu^3}{R^6}-\rho_{0}^{-2}\frac{\mu}{R^{2}}\mathfrak{F}_{n}\geq16\frac{\mu^3}{R^6},\,\,\,\, 8\frac{\mu^3}{R^6}\geq\mu^{\mathcal{Q}(\ell,R)}\|\varphi_{tt}\|_{L^{\infty}}\rho_{0}^{-2}
	\end{equation}
	which implies the Carleman inequality.
\end{proof}
With the Carleman estimate at hand, we provide the lower bound of the localized $L^2_x$ mass.
\begin{theorem}\label{mass lower}
	Suppose that $u\in C([0,1],L^{2}(M))\cap L^{2}([0,1],H^{1}(M))$ is a solution of
	\begin{equation}
		i\partial_{t}u+\Delta_{g}u+Vu=0
	\end{equation}
	In addition, assume that there exists $\varepsilon>0$ and $R_0>0$  such that 
	\begin{equation}
		\int_{\frac{1}{8}}^{\frac{7}{8}}\int_{M}\Big(|u|^{2}+|\nabla_{g} u|_{g}^{2}\Big)\,\dd{\rm Vol}\dd t\leq E_{1}^{2}<\infty
	\end{equation}
	and 
	\begin{align}
		\int_{\frac{1}{4}}^{\frac{3}{4}}\int_{B_{R_{0}}\backslash B_{2\varepsilon}}|u|^{2}\,\dd{\rm Vol}\dd t\geq E_{2}^{2},
	\end{align}
	then there exist parameters $R_{1}=R_{1}(\varepsilon,\|V\|_{L^{\infty}},E_{1},E_{2},\|\varphi\|_{C^2},R_0)$ so that 
	\begin{align}\label{goal estimate}
		\delta(R)=\int_{\frac{1}{8}}^{\frac{7}{8}}\int_{B_{R}\backslash B_{R-1}}|u|^{2}+|\nabla_{g} u|_{g}^{2}\,\dd{\rm Vol}\dd t\geq Ce^{-C_{0}\frac{1}{\ell}R^{2}\log R}
	\end{align}
	holds for $R\geq R_1$ and any fixed $\ell\in\mathbb{N}$.
\end{theorem}
\begin{proof}
	First, we denote a bump function $\varphi$ by $\varphi(t)\in C_{c}^{\infty}\big((\frac{1}{8},\frac{7}{8})\big)$ and $\varphi=3$ on $[\frac{1}{4},\frac{3}{4}]$. Next, we define $\theta\in C^{\infty}(\Bbb R)$ by 
	\begin{equation}
		\theta(r)=\begin{cases}
			0,r\leq1,\\1,r\geq2
		\end{cases}
	\end{equation}
	For any $\varepsilon>0$, we define the rescaled function
	$\theta_{\varepsilon}(x):=\theta(\frac{\rho(x)}{\varepsilon})$. Let  $\kappa_{R}(x)$ be an another bump function with $\kappa_R =1$ in $x\in B_{R-1}$ and $\kappa_{R}=0$ in $ x\in B_{R}^{c}$. Here $B_R$ is the geodesic ball with radius $R$. 
	Moreover, we have
	\begin{equation}
|\varphi'|+|\theta'|+\sup_{j=1,2}|\nabla_{g}^{j}\kappa_{R}|_g\leq C, \,|\nabla_g\theta_{\varepsilon}|_g\leq \frac{C}{\varepsilon}.
	\end{equation}
	 Indeed, one can prove these two estimates by making use of the chain rule of derivative and the calculation on geodesic direction. Then, 
	we write
	\begin{equation*}
		\psi=\mu^{1-\mathcal{Q}(\ell,R)}\frac{\rho^{2}}{R^{2}}+\varphi(t)
	\end{equation*}
	and denote
	\begin{equation}
		h=\theta(\psi)\kappa_{R}(x)\theta_{\varepsilon}(x)u(x,t)=\theta\big(\mu^{1-\mathcal{Q}(\ell,R)}\frac{\rho^{2}}{R^{2}}+\varphi(t)\big)\kappa_{R}\theta_{\varepsilon}u,
	\end{equation}
	where $\mathcal{Q}(\ell,R)=3-\frac{6}{2+\frac{\log\frac{\log R}{\ell}}{\log R}}\in(0,1)$ for $R\gg1$ and \begin{equation}\label{selection of mu and R}
		\mu=\mathcal{C}R^{\frac{6}{3-\mathcal{Q}(\ell,R)}},\, \mathcal{C}=\mathcal{C}(\|\theta\|_{C^2},\|\kappa_{R}\|_{C^2},E_1,E_2).
	\end{equation}
	 By basic calculation, we have
    \begin{equation}\label{Q cal1}
        3-\mathcal{Q}(\ell,R)=\frac{6\log R}{2\log R+\log\frac{\log R}{\ell}},
    \end{equation}
    and  
    \begin{equation}\label{Q cal2}
        \mathcal{Q}(\ell,R)-1=\frac{2\log{\frac{\log R}{\ell}}-2\log{R}}{2\log{R}+\log\frac{\log R}{\ell}}.
    \end{equation}
    Then, we obtain the following properties of $h$. 
	\begin{itemize}
		\item According to the support of $\kappa_{R}$ and $\theta_{\varepsilon}$,
		\begin{equation}\label{suppg1}
			\operatorname{supp}h\subset[0,1]\times(B_{R}\backslash B_{\varepsilon}).
		\end{equation}
		\item The support of the function $\theta(\psi)$ yields that
		\begin{equation}
			\operatorname{supp}h\subset\{(x,t)\in M\times[0,1]:\psi(x,t)\geq1\}.
		\end{equation}
		When $t\in[0,\tfrac{1}{8}]\cup[\tfrac{7}{8},1]$, $\varphi(t)=0$, which implies $\psi=\mu^{1-\mathcal{Q}(\ell,R)}\frac{\rho^2}{R^2}\leq1$ in $([0,\frac{1}{8}]\cup[\frac{7}{8},1])\times B_{R\mu^{-\frac{1}{2}+\frac{\mathcal{Q}(\ell,R)}{2}}}=([0,\frac{1}{8}]\cup[\frac{7}{8},1])\times B_{\mathcal{C}^{\frac{\mathcal{Q}-1}{2}}R^{\frac{2\mathcal{Q}(\ell,R)}{3-\mathcal{Q}(\ell,R)}}}$.
        From \eqref{Q cal1}, we obtain
        \begin{equation}
        R^{\frac{2\mathcal{Q}(\ell,R)}{3-\mathcal{Q}(\ell,R)}}=R^{\frac{\log\frac{\log R}{\ell}}{\log R}}=\frac{\log R}{\ell}\to\infty, \,
        \textit{as}\, R\to\infty.
        \end{equation}
        Thus, together with the definition of $\kappa_{R}$, we get
		\begin{equation}
			h=0\,\,\operatorname{for}\,\,t\in[0,\tfrac{1}{8}]\cup[\tfrac{7}{8},1].
		\end{equation}
		Furthermore, one can verify that  
		\begin{equation}
			\operatorname{supp}(\theta'(\psi))\subset\{(x,t)\in M\times[0,1]:1\leq\psi\leq2\},
		\end{equation}
		which will provide a suitable control of time cutoff terms in Carleman estimates.
		
		\item When $t\in[\frac{1}{4},\frac{3}{4}]$, $\varphi=3$ and we have the lower bound
		\begin{equation}\label{lower bound of psi}
			\psi=\mu^{1-\mathcal{Q}(\ell,R)}\frac{\rho^{2}}{R^{2}}+\varphi\geq3,
		\end{equation}
		which means that $\theta(\psi)=1$ in $[\frac{1}{4},\frac{3}{4}]$.
        \end{itemize}
        In conclusion, we have 
		\begin{equation}\label{h equal u}
			h=u\,\,\textit{on}\,\,(B_{R-1}\backslash B_{2\varepsilon})\times[\tfrac{1}{4},\tfrac{3}{4}].
		\end{equation}
		Utilizing these properties, we turn to apply the  Carleman estimate  established in Lemma \ref{carleman lemma} to function $f=e^{\mu\frac{\rho^{2}}{R^{2}}+\mu^{\mathcal{Q}(\ell,R)}\varphi(t)}h$.  With $\rho_{0}=\varepsilon$ and large $R\geq R_0+1$  
		it follows from \eqref{suppg1}, \eqref{lower bound of psi} and \eqref{h equal u},
		\begin{align}
			&  \operatorname{(LHS)\, of \,}\eqref{carleman estimate}\geq  \frac{\mu^{3}}{R^{6}}\int_{\mathbb{R}}\int_{M}\rho^{2}|f|^{2}\,\dd{\rm Vol}\dd t\\
			=&\frac{1}{2}\frac{\mu^{3}}{R^{6}}\int_{\mathbb{R}}\int_{M}\rho^{2}|f|^{2}\,\dd{\rm Vol}\dd t+\frac{1}{2}\frac{\mu^{3}}{R^{6}}\int_{\mathbb{R}}\int_{M}\rho^{2}|f|^{2}\,\dd{\rm Vol}\dd t\\
			\geq&\frac{1}{2}\mu^{3}R^{-6}\varepsilon^{2}\int_{\mathbb{R}}\int_{M}|e^{\mu\frac{\rho^{2}}{R^{2}}+\mu^{\mathcal{Q}(\ell,R)}\varphi(t)}h|^{2}\,\dd{\rm Vol}\dd t\\&+\frac{1}{2}\mu^{3}R^{-6}(2\varepsilon)^{2}\int_{1/4}^{3/4}\int_{B_{R_{0}}\backslash B_{2\varepsilon}}e^{2\mu|\frac{2\varepsilon}{R}|^{2}+6\mu^{\mathcal{Q}(\ell,R)}}|u|^{2}\,\dd{\rm Vol}\dd t\\
			\geq&\frac{1}{2}\mu^{3}R^{-6}\varepsilon^{2}\int_{\mathbb{R}}\int_{M}e^{2\mu\frac{\rho^2}{R^2}+2\mu^{\mathcal{Q}(\ell,R)}\varphi(t)}|h|^{2}\,\dd{\rm Vol}\dd t+2\mu^{3}R^{-6}\varepsilon^{2}e^{2(4\mu\frac{\varepsilon^{2}}{R^{2}}+3\mu^{\mathcal{Q}(\ell,R)})}E_{2}^{2}.
		\end{align}
	To obtain the upper bound on the right hand side of  \eqref{carleman lemma}, we  compute
	\begin{align}
		(i\partial_{t}+\mathcal{L})h=& (i\partial_{t}+\mathcal{L})\big(\theta(\psi)\kappa_{R}\theta_{\varepsilon}u\big)\nonumber\\=&i\theta'(\psi)\partial_{t}\varphi\kappa_{R}\theta_{\varepsilon}u+i\theta(\psi)\kappa_{R}\theta_{\varepsilon}\partial_{t}u+\Delta_{g}(\theta(\psi)\kappa_{R}\theta_{\varepsilon}))u+2\big\langle\nabla_{g}(\theta\kappa_{R}\theta_{\varepsilon}),\nabla_{g} u\big\rangle_{g}\nonumber\\&+\theta(\psi)\kappa_{R}\theta_{\varepsilon}\Delta_{g} u\nonumber\\=&i\theta(\psi)\kappa_{R}\theta_{\varepsilon}\Big(i\partial_{t}u+\Delta_{g} u\Big)+i\theta'(\psi)\partial_{t}\varphi\kappa_{R}\theta_{\varepsilon}u+\Delta_{g}(\theta(\psi)\kappa_{R}\theta_{\varepsilon}))u\nonumber\\
		&+2\langle\nabla_{g}(\theta\kappa_{R}\theta_{\varepsilon}),\nabla_{g} u\rangle_g\nonumber\\=& Vh+i\theta'(\psi)\partial_{t}\varphi\kappa_{R}\theta_{\varepsilon}u+\Delta_{g}(\theta(\psi)\kappa_{R}\theta_{\varepsilon}))u+2\langle\nabla_{g}(\theta\kappa_{R}\theta_{\varepsilon}),\nabla_{g} u\rangle_{g}.\label{incomplete cutoff}
	\end{align}
	For the term $\Delta_{g}(\theta(\psi)\kappa_{R}\theta_{\varepsilon}))u$, using the formula 
	\begin{equation}
		\Delta_{g}(fh)=h\Delta_{g} f+f\Delta_{g} h+\langle\nabla_{g} f,\nabla_{g} h\rangle_{g},
	\end{equation}
	we deduce that
	\begin{equation}
		\begin{aligned}
			\Delta_{g}(\theta(\psi)\kappa_{R}\theta_{\varepsilon}))u=&\Delta_{g}(\theta(\psi))\kappa_{R}\theta_{\varepsilon}u+\theta(\psi)\Delta_{g}\kappa_{R}\theta_{\varepsilon}u+\theta(\psi)\kappa_{R}\Delta_{g}\theta_{\varepsilon} u+2\langle\nabla_{g}\theta(\psi),\nabla_{g}\kappa_{R}\rangle_{g} \theta_{\varepsilon}u\\&+2\langle\nabla_{g}\theta(\psi),\nabla_{g}\theta_{\varepsilon}\rangle_{g}\kappa_{R}u+2\langle\nabla_{g}\kappa_{R},\nabla_{g}\theta_{\varepsilon}\rangle_{g}\theta(\psi)u\\=&\Big(\theta^{\prime\prime}(\psi)|\nabla_{g}\psi|_g^{2}+\theta'(\psi)\Delta_{g}\psi\Big)\kappa_{R}\theta_{\varepsilon}u+\theta(\psi)\Delta_{g}\kappa_{R}\theta_{\varepsilon}u+\theta(\psi)\kappa_{R}\Delta_{g}\theta_{\varepsilon} u\\&+2\langle\theta'(\psi)\nabla_{g}\psi,\nabla_{g}\kappa_{R}\rangle_{g}\theta_{\varepsilon}u+2\langle\theta'(\psi)\nabla_{g}\psi,\nabla_{g}\theta_{\varepsilon}\rangle_{g}\kappa_{R}u+2\langle\nabla_{g}\kappa_{R},\nabla_{g}\theta_{\varepsilon}\rangle_{g}\theta(\psi)u.\label{incomplete cutoff2}
		\end{aligned}
	\end{equation}
	For the term $2\Big\langle\nabla_{g}(\theta(\psi)\kappa_{R}\theta_{\varepsilon}),\nabla_{g} u\Big\rangle_{g}$, we can compute directly by Leibniz's formula
	\begin{equation}
		\Big\langle\nabla_{g}(\theta(\psi)\kappa_{R}\theta_{\varepsilon}),\nabla_{g} u\Big\rangle_{g}=\Big\langle\theta'(\psi)\nabla_{g}\psi\kappa_{R}\theta_{\varepsilon}+\theta(\psi)\nabla_{g}\kappa_{R}\theta_{\varepsilon}+\theta(\psi)\kappa_{R}\nabla_{g}\theta_{\varepsilon},\nabla_{g} u\Big\rangle_{g}.\label{incomplete cuoff3}
	\end{equation}
	Since $\psi=\mu^{1-\mathcal{Q}(\ell,R)}\frac{\rho^{2}}{R^{2}}+\varphi$, we have $$\nabla_{g}\psi=\mu^{1-\mathcal{Q}(\ell,R)}\frac{\rho}{R^{2}}\nabla_{g}\rho,\quad \Delta_{g}(\psi)=\mu^{1-\mathcal{Q}(\ell,R)}\frac{2+2\rho\Delta_{g}\rho}{R^{2}}.$$ Indeed, we need the comparison theorem for Laplacian: 
	\begin{equation}\label{laplace compa}
		\Delta_g\rho\leq (n-1)\coth\rho.
	\end{equation}
	Inserting \eqref{incomplete cutoff2},\eqref{incomplete cuoff3} into  \eqref{incomplete cutoff}, we obtain
	\begin{align}
		(  i\partial_{t}+\mathcal{L})h=&Vh+i\theta'(\psi)\partial_{t}\varphi\kappa_{R}\theta_{\varepsilon}u+O(1)\mu^{1-\mathcal{Q}(\ell,R)}\theta'(\psi)\Big(\frac{2+2(n-1)\rho\coth\rho}{R^2}\kappa_{R}\theta_{\varepsilon}\\&+2\frac{\rho}{R^{2}}|\nabla_{g}\kappa_{R}|_{g}\theta_{\varepsilon}+2\frac{\rho}{R^{2}}|\nabla_{g}\kappa_{R}|_{g}\theta_{\varepsilon}\Big)u+O(1)\mu^{1-\mathcal{Q}(\ell,R)}\theta'(\psi)\frac{\rho}{R^{2}}\kappa_{R}\theta_{\varepsilon}\nabla_{g} u\\&+\mu^{2-2\mathcal{Q}(\ell,R)}\theta^{\prime\prime}(\psi)\frac{\rho^{2}}{R^{4}}\kappa_{R}\theta_{\varepsilon}u\\&+\theta(\psi)\kappa_{R}\Delta_{g}\theta_{\varepsilon}u+2\langle\nabla_{g}\kappa_{R},\nabla_{g}\theta_{\varepsilon}\rangle_{g}\theta(\psi)u+\langle\theta(\psi)\kappa_{R}\nabla_{g}\theta_{\varepsilon},\nabla_{g} u\rangle_{g}\\&+\theta(\psi)\Delta_{g}\kappa_{R}\theta_{\varepsilon}u+\langle\theta(\psi)\nabla_{g}\kappa_{R}\theta_{\varepsilon},\nabla_{g} u\rangle_{g}\\=&:Vh+\sum_{k=1}^{3}J_k,
	\end{align}
	where 
	\begin{align}
		J_{1}=&i\theta'(\psi)\partial_{t}\varphi\kappa_{R}\theta_{\varepsilon}u+O(1)\mu^{1-\mathcal{Q}(\ell,R)}\theta'(\psi)\Big(\frac{2+2\rho\Delta_{g}\rho}{R^2}\kappa_{R}\theta_{\varepsilon}\\&+2\frac{\rho}{R^{2}}|\nabla_{g}\kappa_{R}|_g\theta_{\varepsilon}+2\frac{\rho}{R^{2}}|\nabla_{g}\kappa_{R}|_g\theta_{\varepsilon}\Big)u+O(1)\mu^{1-\mathcal{Q}(\ell,R)}\theta'(\psi)\frac{\rho}{R^{2}}\kappa_{R}\theta_{\varepsilon}\nabla_{g} u\\&+\mu^{2-2\mathcal{Q}(\ell,R)}\theta^{\prime\prime}(\psi)\frac{\rho^{2}}{R^{4}}\kappa_{R}\theta_{\varepsilon}u,
	\end{align}
\begin{align}
J_2=\theta(\psi)\kappa_{R}\Delta_{g}\theta_{\varepsilon}u+2\langle\nabla_{g}\kappa_{R},\nabla_{g}\theta_{\varepsilon}\rangle_{g}\theta(\psi)u+\langle\theta(\psi)\kappa_{R}\nabla_{g}\theta_{\varepsilon},\nabla_{g} u\rangle_{g},
	\end{align}
\begin{align}
J_3=\theta(\psi)\Delta_{g}\kappa_{R}\theta_{\varepsilon}u+\langle\theta(\psi)\nabla_{g}\kappa_{R}\theta_{\varepsilon},\nabla_{g} u\rangle_{g}.
	\end{align}
	Now, we estimate the first term $J_{1}$. In addition, we note that
	\begin{equation}
		\operatorname{supp}\theta'(\psi)\subset\{1\leq\psi\leq2\}
	\end{equation} which implies $$\mu^{\mathcal{Q}(\ell,R)}\psi\in[\mu^{\mathcal{Q}(\ell,R)},2\mu^{\mathcal{Q}(\ell,R)}].$$
	Also we have
	\begin{equation}
		\operatorname{supp}h\subset\Big\{t\in\big[\tfrac{1}{8},\tfrac{7}{8}\big]:\mu^{1-\mathcal{Q}(\ell,R)}\frac{\rho^2}{R^{2}}\leq1\Big\}.
	\end{equation}
	A direct calculation infers that
	\begin{align}
		&\int_{0}^{1}\int_{M}e^{2\mu^{\mathcal{Q}(\ell,R)}\psi}|J_{1}|^{2}\,\dd{\rm Vol}\dd t\leq C\int_{\frac{1}{8}}^{\frac{7}{8}}\int_{M}e^{4\mu^{\mathcal{Q}(\ell,R)}}(\mu^{1-\mathcal{Q}(\ell,R)}+\mu^{2-2\mathcal{Q}(\ell,R)})(|u|^{2}+|\nabla_{g}u|_g^{2})\,\dd{\rm Vol}\dd t\\\leq& C_{1}(\mu^{1-\mathcal{Q}(\ell,R)}+\mu^{2-2\mathcal{Q}(\ell,R)})e^{4\mu^{\mathcal{Q}(\ell,R)}}E_{1}^{2},
	\end{align}
	where $C_1=C_1(\|\theta\|_{C^2},\|\kappa_{R}\|_{C^{1}})$.
	
	Next, we estimate $J_2$. Since $\operatorname{supp}\nabla_g\theta_{\varepsilon}\subset B_{2\varepsilon\backslash B_{\varepsilon}}$ and $0\leq\varphi\leq3$, it holds that 
	\begin{align}
		&\int_{0}^{1}\int_{M}e^{2\mu^{\mathcal{Q}(\ell,R)}\psi}|J_{2}|^{2}\,\dd{\rm Vol}\dd t\leq C_{2}e^{(2\frac{(2\varepsilon)^{2}}{R^{2}}\mu+6\mu^{\mathcal{Q}(\ell,R)})}\int_{\frac{1}{8}}^{\frac{7}{8}}\int_{B_{2\varepsilon}\backslash B_{\varepsilon}}(|u|^{2}+|\nabla_{g}u|_g^{2})\,\dd{\rm Vol}\dd t\\
		\leq& C_{2}e^{(2\frac{(2\varepsilon)^{2}}{R^{2}}\mu+6\mu^{\mathcal{Q}(\ell,R)})}E_{1}^{2},
	\end{align} 
	in which $C_2=C_2(\varepsilon, \|\theta\|_{C^2},\|\kappa_{R}\|_{C^1})$.
	
	Finally, for $J_{3}$, utilizing the fact \eqref{suppg1}, we have
	\begin{align}
		e^{\mu\frac{\rho^{2}}{R^{2}}+\mu^{\mathcal{Q}(\ell,R)}\varphi}\leq e^{\mu+3\mu^{\mathcal{Q}(\ell,R)}}, 
	\end{align}
	which means that
	\begin{align}
		&\int_{0}^{1}\int_{M}e^{2\mu\frac{\rho^2}{R^2}+2\mu^{\mathcal{Q}(\ell,R)}\varphi(t)}|J_5|^2\,\dd{\rm Vol}\dd t\leq C_{3}e^{2\mu+6\mu^{\mathcal{Q}(\ell,R)}}\int_{\frac{1}{8}}^{\frac{7}{8}}\int_{B_{R}\backslash B_{R-1}}(|u|^{2}+|\nabla_{g}u|_g^{2})\,\dd{\rm Vol}\dd t\\=&C_{3}e^{2\mu+6\mu^{\mathcal{Q}(\ell,R)}}\delta(R),
	\end{align}
	 where $C_3=C_3(\|\theta\|_{L^\infty},\|\kappa_R\|_{C^2})$.
	It remains to estimate the term involving the potential $Vh$.
	Since $V(t,x)\in L^\infty (M\times[0,1])$, we obtain
	\begin{align}
		\int_{0}^{1}\int_{M}e^{2\mu\frac{\rho^{2}}{R^{2}}+2\mu^{\mathcal{Q}(\ell,R)}\varphi(t)}|Vh|^{2}\,\dd {\rm Vol}\dd t\leq \|V\|_{L^{\infty}}^{2}\int_{0}^{1}\int_{M}e^{2\mu\frac{\rho^2}{R^2}+2\mu^{\mathcal{Q}(\ell,R)}\varphi(t)}|h|^{2}\,\dd {\rm Vol}\dd t.
	\end{align}
	Now, we collect all the estimates above to obtain
	\begin{align}
		&\frac{1}{2}\mu^{3}R^{-6}\varepsilon^{2}\int_{0}^{1}\int_{M}e^{2\mu\frac{\rho^{2}}{R^{2}}+2\mu^{\mathcal{Q}(\ell,R)}\varphi(t)}|h|^{2}\,\dd{\rm Vol}\dd t+\mu^{3}R^{-6}\varepsilon^{2}e^{2(\mu\frac{4\varepsilon^{2}}{R^{2}}+3\mu^{\mathcal{Q}(\ell,R)})}E_{2}^{2}\\
		\leq&\|V\|_{L^{\infty}}\int_{M}e^{2\mu\frac{\rho^{2}}{R^{2}}+2\mu^{\mathcal{Q}(\ell,R)}\varphi(t)}|h|^{2}\,\dd{\rm Vol}\dd t+C_{1}(\mu^{1-\mathcal{Q}(\ell,R)}+\mu^{2-2\mathcal{Q}(\ell,R)})e^{4\mu^{\mathcal{Q}(\ell,R)}}E_{1}^{2}\\&+C_{2}e^{2(\mu\frac{4\varepsilon^{2}}{R^{2}}+3\mu^{\mathcal{Q}(\ell,R)})}E_{1}^{2}+C_{3}e^{2\mu+6\mu^{\mathcal{Q}(\ell,R)}}\delta(R).\label{control0}
	\end{align}
	In order to obtain the lower bound of $\delta(R)$, we need to absorb the first third terms on the right-hand side by the left terms. To achieve this, we need 
	the following two statements hold:
	\begin{equation}\label{control1}
		\frac{1}{2}\mu^{3}R^{-6}\varepsilon^2\geq\|V\|_{L^{\infty}},
	\end{equation}
	and
	\begin{equation}\label{control2}
		\mu^{3}R^{-6}\varepsilon^{2}e^{2(\mu\frac{4\varepsilon^{2}}{R^{2}}+3\mu^{\mathcal{Q}(\ell,R)})}E_{2}^{2}\geq 2C_{1}(\mu^{1-\mathcal{Q}(\ell,R)}+\mu^{2-2\mathcal{Q}(\ell,R)})e^{4\mu^{\mathcal{Q}(\ell,R)}}E_{1}^{2}+C_{2}e^{2(\mu\frac{4\varepsilon^{2}}{R^{2}}+3\mu^{\mathcal{Q}(\ell,R)})}E_{1}^{2}.
	\end{equation}
	Indeed, from \eqref{control1} and \eqref{control2}, 
	we have
	\begin{equation}
		\begin{aligned}
			&\frac{1}{4}\mu^{3}R^{-6}\varepsilon^{2}\int_{0}^{1}\int_{M}e^{2\mu\frac{\rho^{2}}{R^{2}}+2\mu^{\mathcal{Q}(\ell,R)}\varphi(t)}|h|^{2}\,\dd{\rm Vol}\dd t+\frac{1}{2}\mu^{3}R^{-6}\varepsilon^{2}e^{2(\mu\frac{4\varepsilon^{2}}{R^{2}}+3\mu^{\mathcal{Q}(\ell,R)})}E_{2}^{2}\\\leq& C_{3}e^{2\mu+6\mu^{\mathcal{Q}(\ell,R)}}\delta(R).
		\end{aligned}
		\label{final mu R}
	\end{equation}
	From this estimate we can get
	\begin{equation}
		\frac{1}{2}\mu^{3}R^{-6}\varepsilon^{2}e^{2(\mu\frac{4\varepsilon^{2}}{R^{2}}+3\mu^{\mathcal{Q}(\ell,R)})}E_{2}^{2}\leq C_{3}e^{2\mu+6\mu^{\mathcal{Q}(\ell,R)}}\delta(R).
	\end{equation}
	It remains to verify two conditions: \eqref{control1} and \eqref{control2}.
	The condition \eqref{control1} is equivalent to the following 
	\begin{equation}\label{condition of mu R1}
		\frac{\mu}{R^{2}}\geq\big(\frac{4\|V\|_{L^{\infty}}}{\varepsilon^2}\big)^{\frac{1}{3}}.
	\end{equation} 
	To achieve \eqref{control2}, we require  
	\begin{align}\label{control3}
		&(\mu^{3}R^{-6}\varepsilon^{2}E_{2}^{2}-E_{1}^{2}C_{2})e^{\frac{8\varepsilon^{2}\mu}{R^{2}}+2\mu^{\mathcal{Q}(\ell,R)}}\geq 2C_{1}E_{1}^{2}(\mu^{1-\mathcal{Q}(\ell,R)}+\mu^{2-2\mathcal{Q}(\ell,R)}),
	\end{align}
	which is a consequence of 
	\begin{equation}\label{control4}
		\mu^{3}R^{-6}\varepsilon^{2}E_{2}^{2}-E_{1}^{2}C_{2}\geq 4C_{1}E_{1}^{2},
	\end{equation}
	and 
	\begin{equation}\label{mystery}
		e^{2\mu^{\mathcal{Q}(\ell,R)}}\geq 2\mu^{2-2\mathcal{Q}(\ell,R)}.
	\end{equation}
	In an equivalent way, we only need to assume
	\begin{equation}\label{condition of mu R2}
		\frac{\mu}{R^{2}}\geq E_{2}^{-2/3}\varepsilon^{-2/3}(4C_{1}E_{1}^{2}+E_{1}^{2}C_{2})^{1/3}.
	\end{equation}
	For completeness, we give a proof of \eqref{mystery}. It is sufficient to prove 
    \begin{equation}
        e^{\mu^{\mathcal{Q}(\ell,R)}}\mu^{\mathcal{Q}(\ell,R)}\geq\sqrt{2}\mu.
    \end{equation}
   Since we take $\mu=\mathcal{C}R^{\frac{6}{3-\mathcal{Q}(\ell,R)}}$, it is reduced to prove that
   \begin{equation}
F(R):=e^{(\mathcal{C}R)^{\frac{6\mathcal{Q}(\ell,R)}{3-\mathcal{Q}(\ell,R)}}}(\mathcal{C}R)^{\frac{6(\mathcal{Q}(\ell,R)-1)}{3-\mathcal{Q}(\ell,R)}}\geq\sqrt{2}.
   \end{equation}
Using \eqref{Q cal1}, \eqref{Q cal2}, it holds that
\begin{equation}
\begin{aligned}
&\exp\Big({(\mathcal{C}R)^{\frac{3\log(\frac{\log R}{\ell})}{\log R}}}\Big)(\mathcal{C}R)^{2\frac{\log(\frac{\log R}{\ell})-\log R}{\log R}}\\=&\exp\Big({\mathcal{C}^{\frac{3\log(\frac{\log R}{\ell})}{\log R}}(\frac{\log R}{\ell})^{3}}\Big)\mathcal{C}^{2\frac{\log(\frac{\log R}{\ell})-\log R}{\log R}}R^{2\frac{\log(\frac{\log R}{\ell})-\log R}{\log R}}\\=&\exp\Big({\mathcal{C}^{\frac{3\log(\frac{\log R}{\ell})}{\log R}}(\frac{\log R}{\ell})^{3}}\Big)\mathcal{C}^{2\frac{\log(\frac{\log R}{\ell})-\log R}{\log R}}(\frac{\log R}{\ell})^{2}\frac{1}{R^{2}}
    \end{aligned}
\end{equation}
We denote a quantity $C_0(R)$ by
\begin{equation}
C_0(R):=\mathcal{C}^{\frac{3\log\big(\frac{\log R}{\ell}\big)}{\log R}},\,\,C_1(R)=\mathcal{C}^{2\frac{\log\big(\frac{\log R}{\ell}\big)-\log R}{\log R}}.
\end{equation}
By taking the logarithmic of $F(R)$, we have 
\begin{align*}
   \log F(R)=& \log\big(e^{C_0(R)\big(\frac{\log R}{\ell}\big)^3}C_1(R)\big(\frac{\log R}{\ell}\big)^2\frac{1}{R^2}\big)\\
=&C_0(R)\big(\frac{\log R}{\ell}\big)^3+\log C_1(R)+2\log\log R-2\log \ell-2\log R,
\end{align*}
where the last three terms can be absorbed by the first term, which tends to $\infty$ as $R\to\infty$. Hence, we can obtain the lower bound of $F(R)\geq\sqrt 2$. As a consequence, we obtain the required estimate \eqref{mystery}.
\end{proof}
\begin{proof}[Proof of  Theorem \ref{thm3}]
	We proceed by a contradiction argument.
	Assume that  there exists $R_{0}>0$ large and $\varepsilon>0$ small such that
	\begin{equation}
		\int_{1/4}^{3/4}\int_{B_{R_{0}}\backslash B_{2\varepsilon}}|u|^{2}\,\dd{\rm Vol}\dd t\in(0,\infty).
	\end{equation}
	In addition, from logarithmic convexity, we know 
	\begin{equation}
		\int_{\frac{1}{8}}^{\frac{7}{8}}\int_{M}(|u|^{2}+|\nabla_{g}u|_g^{2})\,\dd{\rm Vol}\dd t<\infty.
	\end{equation}
	From the lower bound (Theorem \ref{mass lower}), we have
	\begin{equation}\label{contra part}
		\delta(R)=\int_{\frac{1}{8}}^{\frac{7}{8}}\int_{B_{R}\backslash B_{R-1}}\Big(|u|^{2}+|\nabla_{g}u|_g^{2}\Big)\,\dd{\rm Vol}\dd t\geq C_{1}e^{-C_{0}R^{\frac{6}{3-\mathcal{Q}(\ell,R)}}}=C_1 e^{-C_{0}\frac{1}{\ell}R^{2}\log{R}}.
	\end{equation}
	On the other hand, from Proposition \ref{cor-log}, we obtain 
	\begin{equation}
		\int_{0}^{1}\int_{M}e^{\sigma\rho^{2}\log{\rho}}\Big(|u|^{2}+t(1-t)|\nabla_{g}u|_{g}^{2}\Big)\,\dd{\rm Vol}\dd t<\infty,
	\end{equation}
	which implies that
	\begin{equation}\label{reduction last}
		\lim_{R\to\infty}e^{\sigma R^{2}\log{R}}\int_{\frac{1}{8}}^{\frac{7}{8}}\int_{B_{R+1}\backslash B_{R}}\big(|u|^{2}+|\nabla_{g}u|_g^{2}\big)\,\dd{\rm Vol}\dd t=0.
	\end{equation}
	Taking $\ell$ large such that $\frac{C_{0}}{\ell}\leq\frac{\sigma}{2}$, we have $C_{1}e^{-C_{0}R^\frac{6}{3-\mathcal{Q}(\ell,R)}}=C_{1}e^{-C_{0}\frac{1}{\ell
		}R^{2}\log{R}}>e^{-\sigma R^{2}\log{R}}$ for sufficiently large $R$, which contradicts to \eqref{contra part}.
\end{proof}
\section*{Acknowledgement} 
The authors thank to Prof. C.E. Kenig   for his  valuable suggestions and Prof. Pilod for drawing our attention to \cite{Jensen}. This work was supported by the National Key Research and
Development Program of China (Grant number No. 2022YFA1005700) and the National Natural Science
Foundation of China (Grant numbers No.12371095 and No.12531005).

\appendix
\section{Curvature tensor of asymptotic hyperbolic manifold}
In this appendix, we provide the calculation of several geometric quantities on Cartan-Hadamard manifold under Assumption \ref{assum}.  
Given the warped product metric,\[
g = d\rho^2 + \sinh^2\rho \, \Upsilon_{jk}(\rho,\theta) \, d\theta^j d\theta^k,
\]
where $\Upsilon_{jk}$ is a symmetric positive-definite tensor depending on $\rho$ and $\theta$.  Given a frame of  coordinates: $x^0 = \rho$, $x^i = \theta^i$ ($i=1,\dots,n-1$). Hence, the metric can be written in component:
$$
g_{00}=1,\quad g_{0i}=0,\quad g_{ij}= \sinh^2\rho \,\Upsilon_{ij}.
$$
Its inverse metric reads as
\[
g^{00}=1,\quad g^{0i}=0,\quad g^{ij}= \frac{1}{\sinh^2\rho}\,\Upsilon^{ij},
\]
where $\Upsilon^{ij}:=(\Upsilon^{-1})_{ij}$ is the $(i,j)$-th entry of the inverse of $\Upsilon_{ij}$. 

Next, we introduce the notation of Christoffel symbol. For a frame $\{x_1,\cdots,x_n\}$, the second-kind Christoffel symbol reads as
\begin{align*}
\Gamma^k_{ij} = \frac12 g^{kl}(\partial_i g_{lj} + \partial_j g_{li} - \partial_l g_{ij}).
\end{align*}
In our setting, the Christoffel symbol has the form
\begin{gather}
	\Gamma^0_{0i} = 0,\,\Gamma_{00}^i=0,\,\,0\leq i\leq n-1,\label{C-1}\\
\Gamma^0_{ij} = -\frac12 \partial_0 g_{ij} 
= -\frac12 \partial_\rho (\sinh^2\rho \,\Upsilon_{ij}) 
= -\sinh\rho\cosh\rho \,\Upsilon_{ij} - \frac12 \sinh^2\rho \,\partial_\rho\Upsilon_{ij},\label{C-2}\\
\Gamma^i_{0j} = \frac12 g^{ik} \partial_0 g_{kj} 
= \frac12 \cdot \frac{1}{\sinh^2\rho}\,\Upsilon^{ik} \cdot \partial_\rho(\sinh^2\rho \,\Upsilon_{kj}) 
= \coth\rho \,\delta^i_j + \frac12 (\Upsilon^{-1}\partial_\rho\Upsilon)^i_j,\label{C-3}
\end{gather}where $(\Upsilon^{-1}\partial_\rho\Upsilon)^i_j = \Upsilon^{ik}\partial_\rho\Upsilon_{kj}$. The off-diagonal components read as 
\begin{align*}
	\Gamma^i_{jk} = \frac12 g^{il}(\partial_j g_{lk} + \partial_k g_{lj} - \partial_l g_{jk}) 
= \frac12 \Upsilon^{il}(\partial_j \Upsilon_{lk} + \partial_k \Upsilon_{lj} - \partial_l \Upsilon_{jk}) 
\stackrel{\triangle}{=} \Gamma^{i}_{jk}(\Upsilon).
\end{align*}

Recall the definition of  Riemannian curvature tensor
\( R^a_{bcd} \),
\[
R^a_{bcd} = \partial_c \Gamma^a_{db} - \partial_d \Gamma^a_{cb} + \Gamma^a_{ce} \Gamma^e_{db} - \Gamma^a_{de} \Gamma^e_{cb}, \,\,\,a,b,c,d,e\in\{0,1,\dots,n-1\}.
\]
We will use \eqref{C-1}, \eqref{C-2} and \eqref{C-3} to calculate the Riemannian curvature.
For simplicity, we introduce the following notations:
\begin{itemize}
	\item $ \dot{\Upsilon}_{ij} := \partial_\rho \Upsilon_{ij},\quad \ddot{\Upsilon}_{ij} := \partial_\rho^2 \Upsilon_{ij} $,
	\item $ (\dot{\Upsilon}^i)_j := \Upsilon^{ik} \dot{\Upsilon}_{kj},\quad (\dot{\Upsilon}^2)^i_j := \Upsilon^{ik} \dot{\Upsilon}_{kl} \Upsilon^{ml} \dot{\Upsilon}_{mj} $,
	\item $ \tilde{R}^{m}_{jkl} $: denotes the Riemannian curvature tensor of the metric $ \Upsilon_{ij} $ with fixed $ \rho $,
	\item $\tilde\nabla $: represents the covariant derivative on geodesic sphere $S_\rho$ with respect to $ \Upsilon_{ij} $ and satisfies the compatible equation $\tilde \nabla_k \Upsilon_{ij} = 0 $.
\end{itemize}
With the above notations at hand, we can define the trace of $\dot{\Upsilon}$ as follows
\begin{align*}
  \operatorname{tr}_\Upsilon(\dot\Upsilon):=\sum_{i=1}^{n}(\dot{\Upsilon}^i)_i=\sum_{i,j=1}^n\Upsilon^{ij}\dot{\Upsilon}_{ij}.
\end{align*}
Using the Christoffel symbols, we now turn to calculate the each components of Riemannian curvature tensor.

First, we consider components of $R_{i0j}^0$:
\begin{equation*}
\begin{split}
	R_{i0j}^0 &= \partial_\rho \Gamma^0_{ij} - \Gamma^0_{ik} \Gamma^k_{0j} \\
	&= -\sinh^2 \rho \, \Upsilon_{ij} - \sinh\rho \cosh\rho \, \dot{\Upsilon}_{ij} - \frac{1}{2} \sinh^2 \rho \, \ddot{\Upsilon}_{ij} + \frac{1}{4} \sinh^2 \rho \, (\dot{\Upsilon}\Upsilon^{-1} \dot{\Upsilon})_{ij} \\
	&= -\sinh^2 \rho \left[ \Upsilon_{ij} + \coth\rho \, \dot{\Upsilon}_{ij} + \frac{1}{2} \ddot{\Upsilon}_{ij} - \frac{1}{4} (\dot{\Upsilon}\Upsilon^{-1} \dot{\Upsilon})_{ij} \right].
\end{split}
\end{equation*}

Second, we compute the components involving $R_{0j0}^i$. 
Using the identity $$ R_{0j0}^i = -R_{00j}^i$$ we infer that 
\[
R_{0j0}^i = -\delta^i_j - \coth\rho \, \dot{\Upsilon}^i_j - \frac{1}{2} (\Upsilon^{-1}\ddot{\Upsilon})^i_j + \frac{1}{4} (\dot{\Upsilon}^2)^i_j.
\]

Next, the pure tangential components reads as
\begin{align*}
		R_{jkl}^i &= \tilde{R}_{jkl}^i + \Gamma^i_{k0} \Gamma^0_{lj} - \Gamma^i_{l0} \Gamma^0_{kj} \\
	&= \tilde{R}_{jkl}^i - \cosh^2 \rho \left( \delta^i_k \Upsilon_{lj} - \delta^i_l \Upsilon_{kj} \right) \\
	&\quad - \frac{1}{2} \sinh\rho \cosh\rho \left( \delta^i_k \dot{\Upsilon}_{lj} - \delta^i_l \dot{\Upsilon}_{kj} + \dot{\Upsilon}^i_k \Upsilon_{lj} - \dot{\Upsilon}^i_l \Upsilon_{kj} \right) \\
	&\quad - \frac{1}{4} \sinh^2 \rho \left( \dot{\Upsilon}^i_k \dot{\Upsilon}_{lj} - \dot{\Upsilon}^i_l \dot{\Upsilon}_{kj} \right).
\end{align*}

The last parts that we need to calculate are $ R^0_{ijk} $ and $ R^i_{0jk} $. Since the compatible equation does not possess the cyclic permutation between $i,j,k$. Therefore, these terms arise naturally. 
Using Gaussian normal coordinates, we obtain:
\begin{gather*}
	R^0_{ijk} = -\frac{1}{2} \sinh^2 \rho \left( \tilde\nabla_j \dot{\Upsilon}_{ki} - \tilde\nabla_k \dot{\Upsilon}_{ji} \right)\\
	R^i_{0jk} = \frac{1}{2} \left( \tilde\nabla_j \dot{\Upsilon}^i_k - \tilde\nabla_k \dot{\Upsilon}^i_j \right).
\end{gather*}

By making use of  \( R_{0ijk} = g_{il} R^l_{0jk}\),
 the Ricci tensor can be written as
\begin{align}\label{Ricci tensor}
	\operatorname{Ric}_{ij}=&\widetilde{\operatorname{Ric}}_{ij}-\left[(n-2)\cosh^2\rho+\sinh^2\rho\right]\Upsilon_{ij}\\
	&-\frac12\sinh\rho\cosh\rho\left[ \text{tr}_\Upsilon(\dot{\Upsilon})\,\Upsilon_{ij} + (n-1)\dot{\Upsilon}_{ij} \right] \\
	&- \frac12\sinh^2\rho\left[ \ddot{\Upsilon}_{ij} - (\dot{\Upsilon}^2)_{ij} +\frac{1}{2}\text{tr}_\Upsilon(\dot{\Upsilon})\,\dot{\Upsilon}_{ij}\right],
\end{align}
where $\widetilde{\operatorname{Ric}}_{ij}$ is the Ricci tensor associated to $\tilde{R}^{m}_{jkl}$, $i,j,k,l,m\in\{1,\cdots,n-1\}$.  
The radial component of Ricci tensor reads as
\begin{equation}\label{Ricci tensor radial}
	Ric_{00} = -(n-1) - \coth\rho \, \operatorname{tr}_\Upsilon(\dot{\Upsilon}) - \frac{1}{2}\operatorname{tr}_\Upsilon(\ddot{\Upsilon}) + \frac{1}{4}\operatorname{tr}_\Upsilon(\dot{\Upsilon}^2).
\end{equation}
The cross component of Ricci tensor is
\begin{equation}\label{Ricci tensor2}
	Ric_{0i} = \frac{1}{2}\left( \tilde{\nabla}_j \dot{\Upsilon}_i^{~j} - \tilde{\nabla}_i \operatorname{tr}_\Upsilon(\dot{\Upsilon}) \right).
\end{equation}
Since $M$ is  Cartan-Hadamard manifold, the sectional curvature is non-positive. Under the assumption of the perturbation metric, we are concerned with the asymptotic behavior of curvature if $\rho\to\infty$. More precisely, we claim that the sectional curvature $\mathcal{K}\to-1$ as $\rho\to\infty$.


We first state the definition of sectional curvature. 
\begin{definition}
The sectional curvature of plane spanned by the linearly independent  tangent vectors $X$, $Y$ is defined by 
\begin{equation}\label{sectional curvature}
  \mathcal{K}(X,Y):= \operatorname{sec}(X,Y)=\frac{\langle \operatorname{Rm}(X,Y)Y,X\rangle_{g}}{\langle X,X\rangle_{g}\langle Y,Y\rangle_{g}-\langle X,Y\rangle_{g}^{2}}.
\end{equation}    
\end{definition}
\begin{proposition}
    Under the assumption on perturbed metric $\Lambda$, the sectional curvature \eqref{sectional curvature} tends to $-1$ as $\rho\to+\infty$.
\end{proposition}
\begin{proof}
     Let $X,Y\in\{\partial_\rho,\partial_j\}$. We first compute $\mathcal{K}(\partial_\rho,\partial_{\theta_{j}})$. By definition of sectional curvature \eqref{sectional curvature}, 
     \begin{align}
         \mathcal{K}(\partial_{\rho},\partial_{\theta_{j}})=&\frac{\langle\operatorname{Rm}(\partial_{\rho},\partial_{\theta_{j}})\partial_{\theta_{j}},\partial_{\rho}\rangle_{g}}{\langle\partial_\rho,\partial_{\rho}\rangle_g\langle\partial_{\theta_j},\partial_{\theta_j}\rangle_g-\langle\partial_\rho,\partial_{\theta_{j}}\rangle_g^2}=\frac{R_{0jj0}}{\langle\partial_{\theta_{j}},\partial_{\theta_{j}}\rangle_g}\\=&\frac{1}{\langle\partial_{\theta_j},\partial_{\theta_j}\rangle_g}\sum_{a=0}^{n-1} g_{a0}R_{0jj}^{a}\\=&\frac{-\sinh^2\rho[\Upsilon_{jj}+\coth\rho\dot{\Upsilon}_{jj}+\frac{1}{2}\ddot{\Upsilon}-\frac{1}{4}(\dot{\Upsilon}\Upsilon^{-1}\dot{\Upsilon})_{jj}]}{\sinh^2\rho\Upsilon_{jj}}\to-1,
     \end{align}
     as $\rho\to\infty$. In the above calculation, we use the fact that 
$g_{a0}=0$ when $a\neq0$.

     Similarly, 
     \begin{align}
         &\mathcal{K}(\partial_{\theta_{j}},\partial_{\theta_{k}})=\langle\operatorname{Rm}(\partial_{\theta_{j}},\partial_{\theta_{k}})\partial_{\theta_{k}},\partial_{\theta_{j}}\rangle_g\\=&-\frac{R_{jkkj}}{\langle\partial_{\theta_j},\partial_{\theta_j}\rangle_g\langle\partial_{\theta_k},\partial_{\theta_k}\rangle_g-\langle\partial_{\theta_j},\partial_{\theta_k}\rangle_g^2}\\=&-\frac{1}{(\Upsilon_{jj}\Upsilon_{kk}-\Upsilon_{jk}^{2})\sinh^4\rho}\Bigg[g_{jl}\Big[\tilde{R}^{l}_{kkj}-\cosh^2\rho(\delta^{l}_{k}\Upsilon_{kj}-\delta^{l}_{j}\Upsilon_{kk})\\&-\frac{1}{2}\sinh\rho\cosh\rho(\delta^{l}_{k}\dot{\Upsilon}_{jk}-\delta^{l}_{j}\dot{\Upsilon}_{kk}+\dot{\Upsilon}^{l}_{k}\Upsilon_{kj}-\dot{\Upsilon}^{l}_{j}\Upsilon_{kk})-\frac{1}{4}\sinh^2\rho(\dot{\Upsilon}^{l}_{k}\dot{\Upsilon}_{jk}-\dot{\Upsilon}^{l}_{j}\dot{\Upsilon}_{kk})\Big]\Bigg]\\=&-\frac{1}{(\Upsilon_{jj}\Upsilon_{kk}-\Upsilon_{jk}^{2})\sinh^4\rho}\Big(\sinh^{2}\rho\tilde{R}_{jkkj}-\sinh^2\rho\cosh^2\rho(\Upsilon_{jk}^{2}-\Upsilon_{kk}\Upsilon_{jj})\\&-\frac{1}{2}\sinh^3\rho\cosh\rho(\dot{\Upsilon}_{jk}\Upsilon_{jk}-\dot{\Upsilon}_{kk}\Upsilon_{jj}+\dot{\Upsilon}^{l}_{k}\Upsilon_{jk}\Upsilon_{jl}-\dot{\Upsilon}^{l}_{j}\Upsilon_{kk}\Upsilon_{jl})-\frac{1}{4}\sinh^4\rho(\dot{\Upsilon}_{k}^{l}\dot{\Upsilon}_{jk}\\&-\dot{\Upsilon}_{j}^{l}\dot{\Upsilon}_{kk})\Upsilon_{jl}\Big)\to-1
     \end{align}
     as $\rho\to\infty$, which is acceptable.
\end{proof}
\section{Proof of Proposition \ref{bi-laplacian}}
In this section, we aim to provide a rigorous proof of Proposition \ref{bi-laplacian}, which is a key in establishing the logarithmic convexity and Carleman estimate. To derive the lower bound of the commutator $\mathcal{S}_t+[\mathcal{S},\mathcal{A}]$, the bilaplacian of weight function appears in the virial structure. We take $\varphi=\gamma\rho^2$. Indeed, by Lemma \ref{technical-lemma}, in order to obtain the suitable lower bound of $\mathcal{S}_{t}+[\mathcal{S},\mathcal{A}]$, one need to control the Hessian term $\operatorname{Hess}(\rho^2)(\nabla_gf,\nabla_g\bar f)$ and the bilaplacian of squared distant function $\Delta_g^2(\rho^2)$.

A difficult estimate is to bound $\Delta_g^{2}(\rho^{2})$ where $\rho=d(x,\mathbf{0})$ is the geodesic distance in the asymptotic hyperbolic space. Unlike the hyperbolic space, no explicit formula for $\Delta_g^{2}(\rho^{2})$ can be used. To overcome this difficulty, we use the fundamental equations of submanifolds in Riemannian geometry to derive the explicit formula between $\Delta_g^{2}(\rho^{2})$ and curvature tensor. The notations and definitions used in this paper are partially borrowed from
Petersen \cite{Petersen} Chow-Lu-Ni \cite{Chow}.

By virtue of the Hessian comparison theorem together with the identity 
\begin{align*}
    \operatorname{Hess}(\rho^2)=2\rho\operatorname{Hess}(\rho)+2d\rho\otimes d\rho,
\end{align*}
 one deduces the estimate
\begin{align}
\operatorname{Hess}\rho^{2}(v,v)\geq2g(v,v),
\end{align}
where $v$ is the tangent vector at $p\in M$. Consequently, it yields
    \begin{equation}
        \int_{M}4\operatorname{Hess}(\rho^2)(\nabla_g f,\nabla_g\bar{f})\,\dd{\rm Vol}\geq8\int_{M}|\nabla_g f|_g^{2}\,\dd{\rm Vol}.
    \end{equation} 
It now remains to control $\Delta_g^2(\rho^2)$, a term which is significantly intricate than the Hessian part.  
To this end, we first show some geometric facts concerning $\rho^2$, with $\rho$ being the geodesic distance from $x$ to the origin $\mathbf{0}$. Leibniz's rule gives 
$$\Delta_g(\rho^{2})=2\rho\Delta_g\rho+2|\nabla_g\rho|_g^{2}=2\rho\Delta_g\rho+2$$
    where $\rho$ satisfies the eikonal equation  $|\nabla_g\rho|_g=1$.    
    Then, we can obtain
    \begin{equation}\label{biharmonic}
        \Delta_g^{2}(\rho^{2})=2\Delta_g(\rho\Delta_g\rho)=2(\Delta_g\rho)^{2}+4\langle\nabla_g\rho,\nabla_g\Delta_g\rho\rangle_g+2\rho\Delta_g^{2}\rho.
    \end{equation}
    Utilizing Bochner's formula, it holds that 
    \begin{equation}
        0=\frac{1}{2}\Delta_g(|\nabla_g\rho|_g^{2})=|\nabla_{g}^{2}\rho|^{2}+\langle\nabla_g\rho,\nabla_g\Delta_g\rho\rangle_g+\operatorname{Ric}(\nabla_g\rho,\nabla_g\rho),
    \end{equation}
     where $|\nabla^2f|_g^2$ denote the Hilbert-Schmidt norm of Hessian $\nabla^2f$. 
As a consequence, we have 
    \begin{equation}\label{bochner2}
\langle\nabla_g\rho,\nabla_g\Delta_g\rho\rangle_g=-|\nabla^{2}_g\rho|_g^{2}-\operatorname{Ric}(\nabla_g\rho,\nabla_g\rho),
    \end{equation}
   
Substituting \eqref{bochner2} into \eqref{biharmonic}, we deduce
\begin{align}
    \Delta_g^{2}(\rho^{2})=&2|\Delta_g\rho|^{2}+4\langle\nabla_g\rho,\nabla_g\Delta_g\rho\rangle_g+2\rho\Delta_{g}^{2}\rho\\=&2|\Delta_g\rho|^{2}+2\rho\Delta_g^{2}\rho-4\operatorname{Ric(\nabla_g\rho,\nabla_g\rho)}-4|\nabla^{2}\rho|_g^{2}\\=&2|\mathcal{H}|^2+2\rho\Delta_g^2\rho-4\operatorname{Ric}(\nabla_g\rho,\nabla_g\rho)-4|A|_g^2.
\end{align}
where  $\mathcal{H}=\Delta_g\rho$ denotes the mean curvature of geodesic sphere $S_\rho$ and ${A}=\nabla^2\rho$ is a  second fundamental form. 
Indeed, since $|\nabla_g \rho|_g = 1$ holds almost everywhere, the outward unit normal vector field to $S_\rho$  is given by
$N = \nabla_g \rho$.
For any vector fields $X, Y$ tangent to $S_\rho$, the second fundamental form  of $S_\rho$ with respect to $N$ is defined by
\[
A(X, Y) = \langle \nabla_X N, Y \rangle_g.
\]
On the other hand, the Hessian of $\rho$ acting on $X, Y$ is
\[
\nabla^2 \rho (X, Y) = X(Y\rho) - (\nabla_X Y)\rho.
\]
Using the compatibility of $\nabla$ with the metric and $Y\rho = \langle \nabla_g\rho, Y \rangle = \langle N, Y \rangle$, we compute
\[
X\langle N, Y \rangle_g = \langle \nabla_X N, Y \rangle_g + \langle N, \nabla_X Y \rangle_g,
\]
and therefore
\[
\nabla^2 \rho (X, Y) = \langle \nabla_X N, Y \rangle_g + \langle N, \nabla_X Y \rangle_g - \langle N, \nabla_X Y \rangle_g 
= \langle \nabla_X N, Y \rangle_g = \mathrm{I\!I}(X, Y).
\]
Thus the restriction of $\nabla^2 \rho$ to $TS_\rho \otimes TS_\rho$ is exactly the second fundamental form of $S_\rho$. Since $\mathcal{H}$ is a trace of second fundamental form, it is precisely  the mean curvature of $S_\rho$.

Since the estimates of $\mathcal{H}$, $\operatorname{Ric}(\nabla_g\rho,\nabla_g\rho)$, and $|A|_g^2$ are relatively direct, the main difficulty lies in handling the term $\rho\Delta^2\rho$.
 For this purpose, in geodesic polar coordinates, we get
\begin{align}
    \Delta_g(\Delta_g\rho)=\partial_{\rho}^{2}\mathcal H+{\mathcal H}\partial_{\rho}{\mathcal H}+\Delta_{S_\rho}{\mathcal H}.
\end{align}
We now analyze $\partial_{\rho}^{2}\mathcal H$, $\mathcal H\partial_{\rho}\mathcal H$, $\Delta_{g}\mathcal H$ separately. To calculate the derivatives   of mean curvature, we utilize the second fundamental form of geodesic spheres $S_{\rho}$ and the 
Ricatti equation. 
 For convenience, we denote the shape operator (Weingarten map) by $S(X) = \nabla_{X}(\nabla\rho)$, where $X \in TS_{\rho}$. It satisfies the Riccati equation  
\begin{equation}\label{equation of shape operator}
\partial_{\rho}S+S^{2}+R(\cdot,\partial_{\rho})\partial_{\rho}=0.
\end{equation}
Taking the trace of \eqref{equation of shape operator} yields the following equation of $\mathcal{H}$,
\begin{equation}\label{equation of H}
    \partial_\rho\mathcal H=-|S|_g^{2}-\operatorname{Ric}(\partial_{\rho},\partial_{\rho}),
\end{equation}
and consequently
\begin{equation}\label{H2}
  \mathcal  H\partial_{\rho} \mathcal H=- \mathcal H|S|_{g}^{2}- \mathcal H\operatorname{Ric}(\partial_\rho,\partial_\rho).
\end{equation}
On the other hand, the Ricatti equation \eqref{equation of shape operator} also provides that
\begin{equation}\label{integral of A}
    \partial_{\rho}(|S|_g^{2})+2\operatorname{tr}(S^{3})+2\langle R(\cdot,\partial_{\rho})\partial_{\rho},S\rangle_{g}=0.
\end{equation}
Differentiating \eqref{equation of H} and substituting \eqref{integral of A}, we deduce that
\begin{equation}\label{H1}
    \partial^{2}_{\rho}\mathcal H=2\operatorname{tr}(S^{3})+2\langle R(\cdot,\partial_{\rho})\partial_{\rho},S\rangle-\partial_{\rho}\operatorname{Ric}(\partial_{\rho},\partial_{\rho}).
\end{equation}
In order to treat $\Delta_{S_\rho}\mathcal H$, we will employ the Codazzi--Mainardi equation.  
From the definition of Riemannian curvature tensor,
    $$R(U,V)W=\nabla_U\nabla_V W-\nabla_V\nabla_U W-\nabla_{[U,V]}W,$$
and taking $U=\partial_\rho$, $V=X$, $W=\partial_\rho$, we get $\nabla_{\partial_{\rho}}\nabla_{X}\partial_{\rho}=R(\partial_{\rho},X)\partial_{\rho}$, where we use the facts $[X,\partial_\rho]=0$, $\nabla_{\partial_{\rho}}\partial_{\rho}=0$.
Then, applying the Codazzi-Mainardi equation
\begin{equation}
    (\nabla_{X}A)(Y,Z)-(\nabla_{Y}A)(X,Z)=-\langle R(X,Y)Z,\partial_{\rho}\rangle_{g},\,\, \forall X,Y,Z\perp\partial_{\rho},
\end{equation}
and choosing $X=\mathbf{e}_{j}$, $Z=\mathbf{e}_{j}$ in the normal geodesic frame $(r,\mathbf{e}_1,\cdots,\mathbf{e}_{n-1})$, we get
\begin{equation}\label{trace of A}
    (\nabla_{\mathbf{e}_{j}}A)(Y,\mathbf{e}_{j})-(\nabla_{Y}A)(\mathbf{e}_{j},\mathbf{e}_{j})=-\langle R(\mathbf{e}_{j},Y)\mathbf{e}_{j},\partial_{\rho}\rangle_{g}.
\end{equation}
Since the second fundamental form $A$ is symmetric, equation \eqref{trace of A} turns to
\begin{equation}
    (\nabla_{\mathbf{e}_{j}}A)(\mathbf{e}_{j},Y)-(\nabla_{Y}A)(\mathbf{e}_{j},\mathbf{e}_{j})=-\langle R(\mathbf{e}_{j},Y)\mathbf{e}_{j},\partial_{\rho}\rangle_{g},
\end{equation}
which implies that
\begin{equation}
    \operatorname{div}_{S_\rho}A(Y)-\nabla_Y\mathcal H=-\operatorname{Ric}(Y,\partial_{\rho}).
\end{equation}
Thus we have
\begin{equation}
   {\color{red} \dd H} =\operatorname{div}_{S_\rho}A+\operatorname{Ric}(\partial_{\rho},\cdot|_{TS_\rho}).
\end{equation}
Applying the musical isomorphism, we obtain
\begin{equation}\label{tangential grad}
\nabla_{S_\rho}\mathcal H=(\operatorname{div}_{S_\rho}A)^{\#}+(\operatorname{Ric}(\partial_{\rho},\cdot|_{TS_\rho}))^{\#}.
\end{equation}
Here, for a  $1$-form $\omega\in \Lambda^1(M)$, the shift operator $\omega^{\#}$ is defined by  
\begin{equation}
    \langle\omega^{\#},X\rangle=\omega(X).
\end{equation}
Taking the divergence on both side of \eqref{tangential grad} leads to
\begin{align}\label{H3}
    \Delta_{S_\rho}\mathcal H=&\operatorname{div}_{S_\rho}((\operatorname{div}_{S_{\rho}}A)^{\#})+\operatorname{div}_{S_{\rho}}((\operatorname{Ric}(\partial_{\rho},\cdot|_{TS_\rho})))^{\#}.
\end{align}
We now claim the following identities:
\begin{equation}\label{H31}
\operatorname{div}_{S_\rho}\big((\operatorname{div}_{S_\rho}A)^{\#}\big)=\sum_{j,k}^{n-1}\nabla_{e_{j}}\nabla_{e_{k}}A_{jk},
\end{equation}
and 
\begin{equation}
\begin{aligned}\label{H32}
\operatorname{div}_{S_\rho}\big((\operatorname{Ric}(\partial_\rho,\cdot|_{TS_\rho})^\#\big)=\frac{1}{2}\frac{\partial R}{\partial\rho}-\frac{\partial}{\partial\rho}\operatorname{Ric}(\partial_\rho,\partial_\rho)-\mathcal H\operatorname{Ric}(\partial_\rho,\partial_\rho)+\operatorname{tr}(A\cdot\operatorname{Ric}|_{\rm tan}),
\end{aligned}
\end{equation}
where $\operatorname{Ric}_{\rm tan}$ denotes the restriction of the Ricci tensor to the tangent space of $S_\rho$, and $\operatorname{tr}(A\cdot\operatorname{Ric}_{\rm tan}) = A_{j}^{k}\operatorname{Ric}_{jk}$. For completeness, we provide the detailed proofs of \eqref{H31} and \eqref{H32} below.

In the following lemma, we prove \eqref{H31}.
\begin{lemma}
   Let  $\{\rho,\textbf{e}_{j}\}$ with $j=1,\cdots,n-1$ be a normal geodesic frame. We have the identity
    \begin{equation}
        \operatorname{div}_{S_\rho}\big((\operatorname{div}_{S_\rho}A)^{\#}\big)=\sum_{j,k}^{n-1}\nabla_{e_{j}}\nabla_{e_{k}}A_{jk},
    \end{equation}
\end{lemma}
\begin{proof}
    The definition of $\operatorname{div}_{S_{\rho}}X, X\in TS_{\rho}$ is defined by
    \begin{equation}
        \operatorname{div}_{S_\rho}X=\langle\nabla_{\textbf{e}_{j}}X,\textbf{e}_{j}\rangle_g, \textbf{e}_{j}\in TS_\rho.
    \end{equation}
Since the Christoffel symbol $\Gamma_{ij}^k$ vanishes, we have $\nabla_{\textbf{e}_{k}}\textbf{e}_{j}|_{p}=0$.  
From the definition of divergence, it holds \begin{equation}
    \begin{aligned}
\operatorname{div}_{S_\rho}\big((\operatorname{div}_{S_\rho}A)^\#\big)=&\sum_j\langle\nabla_{\textbf{e}_{j}}(\operatorname{div}_{S_\rho}A)^\#,\textbf{e}_{j}\rangle_g\\=&\sum_j\nabla_{\textbf{e}_{j}}\langle(\operatorname{div}_{S_\rho}A)^\#,\textbf{e}_{j}\rangle_g\\=&\sum_{j,k}\nabla_{\textbf{e}_{j}}(\nabla_{\textbf{e}_{k}}A)(\textbf{e}_{k},\textbf{e}_{j})\\
=&\sum_{j,k}\nabla_{\textbf{e}_{j}}\nabla_{\textbf{e}_{k}}A_{jk}.
    \end{aligned}
    \end{equation}
\end{proof}
Before providing the proof of \eqref{H32},  we give a lemma which connects the second fundamental form on $S_\rho$ with the metric $g$.
\begin{lemma}
  In geodesic coordinate, the second fundamental form on $S_\rho$ can be represented as
    $$A=\frac{1}{2}\partial_{\rho}g.$$ 
  Equivalently, the shape operator is given by $$S = \frac{1}{2}g^{-1}\partial_\rho g,$$ where $g^{-1}$ is the inverse metric.
\end{lemma}
Next, we prove \eqref{H32} in the following lemma. 
\begin{lemma}
    On geodesic sphere $S_\rho$ where $\partial_\rho$ is normal vector, we have the identity \begin{equation}
    \label{required}\operatorname{div}_{S_\rho}\big((\operatorname{Ric}(\partial_\rho,\cdot|_{TS_\rho})^{\#}\big)=\frac{1}{2}\frac{\partial R}{\partial\rho}-\frac{\partial}{\partial\rho}\operatorname{Ric}(\partial_\rho,\partial_\rho)-\mathcal H\operatorname{Ric}(\partial_\rho,\partial_\rho)+\operatorname{tr}(A\cdot\operatorname{Ric}_{\rm tan}),
    \end{equation}
\end{lemma}
\begin{proof}
    
    We denote $W=(\operatorname{Ric}(\partial_\rho,\cdot|_{TS_\rho}))^{\#}$, then $\langle W,X\rangle_g=\operatorname{Ric}(\partial_\rho,X)$. We decompose $W=W^{\top}+\operatorname{Ric}(\partial_{\rho},\partial_\rho)\partial_\rho$. From the definition of divergence operator we know that 
    \begin{equation}
    	\operatorname{div}_{S_\rho}W=\operatorname{div}_{S_\rho}W^{\top}.
    \end{equation}
    For any vector field $Z=Z^\top+\langle Z,\partial_\rho\rangle_g\partial_\rho$, we have
    \begin{equation}
        \begin{aligned}
        \operatorname{div}Z=&\langle\nabla_{\textbf{e}_{j}}Z,\textbf{e}_{j}\rangle_{g}+\langle\nabla_{\partial_{\rho}}Z,\partial_{\rho}\rangle_{g}
        \\=&\langle\nabla_{\mathbf{e}_{j}}Z^{\top},\mathbf{e}_{j}\rangle_{g}+\langle \nabla_{\mathbf{e}_{j}}(\langle Z,\partial_\rho\rangle_{g}\partial_\rho),\mathbf{e}_{j}\rangle_{g}+\nabla_{\partial_\rho}\langle Z,\partial_\rho\rangle_{g}
        \\=&\operatorname{div}_{S_\rho}Z^{\top}+\partial_{\rho}\langle Z,\partial_\rho\rangle_{g}+\mathcal H\langle Z,\partial_\rho\rangle_{g}.
        \end{aligned}
    \end{equation}
    Here we use the fact
    \begin{equation}
        \big\langle\nabla_{\mathbf{e}_{j}}(\langle Z,\partial_{\rho}\rangle\partial_{\rho}),\mathbf{e}_{j}\big\rangle_{g}=\langle Z,\partial_\rho\rangle_{g}\langle\nabla_{\mathbf{e}_{j}}\partial_{\rho},\mathbf{e}_{j}\rangle_{g}=\mathcal H\langle Z,\partial_{\rho}\rangle_{g}.
    \end{equation}
    Taking $Z=W$, we obtain 
    \begin{equation}\label{div-decompo}
        \operatorname{div}W=\operatorname{div}_{S_\rho}W^{\top}+\partial_{\rho}\operatorname{Ric}(\partial_\rho,\partial_\rho)+\mathcal H\operatorname{Ric}(\partial_\rho,\partial_\rho).
    \end{equation}
    Recall the  second Bianchi indentity
    \begin{equation}\label{second Bianchi identity}
    (\nabla_{X}\operatorname{Rm})(Y,Z)+(\nabla_{Y}\operatorname{Rm})(Z,X)+(\nabla_{Z}\operatorname{Rm})(X,Y)=0,
    \end{equation}
  we have the local coordinate representation
    \begin{equation}\label{local coordinate chart}
        (\nabla_a \operatorname{Rm})_{bcde}+(\nabla_{b}\operatorname{Rm})_{cade}+(\nabla_{c}\operatorname{Rm})_{abde}=0.
    \end{equation}
    Taking double contraction to \eqref{local coordinate chart},   we obtain 
    \begin{equation}
       g^{ad}g^{bc}\big((\nabla_a \operatorname{Rm})_{bcde}+(\nabla_{b}\operatorname{Rm})_{cade}+(\nabla_{c}\operatorname{Rm})_{abde}\big)=0, 
    \end{equation}
    which means that
    \begin{equation}
        \nabla^{a}\operatorname{Ric}_{ab}=\frac{1}{2}\nabla_b R.
    \end{equation}
    In normal coordinate,  we take $b=\partial_{\rho}$ the radial vector, $\mathbf{e}_{j}$ the tangential vector. 
    Under this normal coordinate, the metric is identity. We have $$\nabla^{a}\operatorname{Ric}_{ab}=g^{ac}\nabla_{c}\operatorname{Ric}_{ab}=\nabla_{a}\operatorname{Ric}_{a\partial_\rho}.$$
    By the definition of divergence of vector, we know that 
    \begin{equation}
    \begin{aligned}
    \operatorname{div}\big((\operatorname{Ric}(\partial_\rho,\cdot|_{TS_\rho})^{\#}\big)=&\langle\nabla_{\textbf{e}_{j}}(\operatorname{Ric}(\partial_\rho,\cdot|_{TS_\rho}))^{\#},\textbf{e}_{j}\rangle_{g}+\langle\nabla_{\partial_\rho}(\operatorname{Ric}(\partial_\rho,\cdot|_{TS_\rho})^{\#}),\partial_{\rho}\rangle_{g}\\=&\nabla_{\textbf{e}_{j}}\big\langle(\operatorname{Ric}(\partial_\rho,\cdot|_{TS_\rho}))^{\#},\textbf{e}_{j}\big\rangle_{g}+\nabla_{\partial_\rho}\langle \operatorname{Ric}(\partial_\rho,\cdot|_{TS_\rho})^{\#},\partial_\rho\rangle_{g}\\=&\nabla_{\textbf{e}_{j}}\big(\operatorname{Ric}(\partial_{\rho},\textbf{e}_{j})\big)+\nabla_{\partial_{\rho}}\operatorname{Ric}(\partial_{\rho},\partial_\rho|_{TS_\rho})\\=&(\nabla_{\textbf{e}_{j}}\operatorname{Ric})(\partial_\rho,\textbf{e}_{j})+\operatorname{Ric}(\nabla_{\textbf{e}_{j}}\partial_\rho,\textbf{e}_{j})+0\\=&\frac{1}{2}\partial_\rho R+\operatorname{tr}(\operatorname{Ric}|_{\rm tan}\cdot A)\big).
    \end{aligned}
    \end{equation}
    This together with \eqref{div-decompo} implies the desired identity \eqref{required}.
\end{proof} 
With these preliminaries in hand, we can express $\Delta_g^2\rho$ as follows:  
\begin{align}
    \Delta_g^2\rho=&2\operatorname{tr}S^3+2\langle R(\partial_\rho,\cdot)\partial_\rho,S\rangle_{g}-2\partial_\rho\operatorname{Ric}(\partial_{\rho},\partial_{\rho})
    \\&+\mathcal H(-|S|_{g}^{2}-2\operatorname{Ric}(\partial_\rho,\partial_\rho))
    \\&+\nabla_{\textbf{e}_{j}}\nabla_{\textbf{e}_k} A_{jk}+\frac{1}{2}\frac{\partial R}{\partial\rho}+\operatorname{tr}(A\cdot\operatorname{Ric}|_{\rm tan}).
\end{align}
Consequently, we obtain the full expression for $\Delta_g^{2}(\rho^{2})$:
\begin{align}\label{original biharmonic formula}
    \Delta_g^{2}(\rho^{2})=&2(\Delta_g\rho)^{2}-4|\nabla^2\rho|^{2}-4\operatorname{Ric}(\nabla_g\rho,\nabla_g\rho)+2\rho\Delta_g^{2}\rho
    \\=&2\mathcal H^{2}-4|S|_{g}^{2}-4\operatorname{Ric}(\nabla_g\rho,\nabla_g\rho)+2\rho\Big(2\operatorname{tr}S^3+2\langle R(\partial_\rho,\cdot)\partial_\rho,S\rangle_{g}-2\partial_\rho\operatorname{Ric}(\partial_{\rho},\partial_{\rho})
    \\&-\mathcal  H|S|_{g}^{2}-2\mathcal H\operatorname{Ric}(\nabla_{g}\rho,\nabla_{g}\rho)+\operatorname{div}_{S_\rho}(\operatorname{div}_{S_{\rho}}A)+\frac{1}{2}\frac{\partial R}{\partial \rho}+\operatorname{tr}(A\cdot\operatorname{Ric}|_{\rm tan})\Big).
\end{align}

Since $\rho^2$ is a smooth function on the whole manifold and there is no cut locus by the Cartan-Hadamard Theorem, $\Delta_g^{2}(\rho^{2})$ remains bounded on any geodesic ball with bounded radius. Therefore, it suffices to control   $\Delta_g^{2}(\rho^{2})$ as $\rho \to \infty$. Under Assumption \ref{assum}, the metric takes the warped product form  

\begin{equation}
    g=g_{0}+\delta g=d\rho^{2}+\sum_{j,k=1}^{n-1}(\sinh\rho)^{2}\Upsilon(\rho,\theta)\dd \theta^{j}\dd\theta^{k},
    \end{equation}
in which
    \begin{equation}
    g_{0}=\dd\rho^2+\sum_{j,k=1}^{n-1}(\sinh\rho)^{2}h(\theta)\dd \theta^{j}\dd\theta^{k},\,\,\, \delta g=\sum_{j,k=1}^{n-1}(\sinh\rho)^2\Lambda(\rho,\theta)\dd\theta^{j}\theta^{k},
\end{equation}
and  $\Upsilon(\rho,\theta)=h+\Lambda$.
 Here $h$ denotes the standard spherical metric whose radial derivative vanishes and $\Lambda$ is a perturbation satisfying the decay estimates  
\begin{equation}\label{decay-1}
   \|\partial^{j}_{\rho}\Lambda\|_{L^\infty} \leq C\rho^{-m-j} \quad (j \in \mathbb{N}), \qquad \text{as } \rho \to \infty.
\end{equation}

The main strategy now consists in decomposing every geometric quantity appearing in \eqref{original biharmonic formula} into a part corresponding to the standard hyperbolic metric and a small perturbation.

First, we treat $\operatorname{tr}(A\cdot \operatorname{Ric}|_{\operatorname{tan}})$.
Let $S_\rho$ denote a geodesic sphere with radius $\rho$. The second fundamental form of $S_\rho$ with respect to the outward unit normal $N = \partial_\rho$ is defined by $A_{ij} = g(\nabla_{\partial_i} N, \partial_j) = \frac{1}{2}\partial_\rho g_{ij}$ with $i,j\in\{1,\cdots,n-1\}$. Substituting $g_{ij} = \sinh^2\rho\, \Upsilon_{ij}$ yields
\[
	A_{ij} = \sinh\rho\cosh\rho\, \Upsilon_{ij} + \frac{1}{2}\sinh^2\rho\, \dot{\Upsilon}_{ij},
\]
	where $\dot{\Upsilon}_{ij} := \partial_\rho \Upsilon_{ij}$. Notice that  $g^{ij} = \frac{\Upsilon^{ij}}{\sinh^2\rho}$, we have
	\begin{equation}\label{eq:A_raised}
		A^{ij} = \frac{\coth\rho}{\sinh^2\rho}\, \Upsilon^{ij} + \frac{1}{2\sinh^2\rho}\, \dot{\Upsilon}^{ij},
	\end{equation}
	with $\dot{\Upsilon}^{ij} := \Upsilon^{ik}\Upsilon^{jl}\dot{\Upsilon}_{kl}$.
	
From \eqref{Ricci tensor}, the Ricci tensor can be written as
	\begin{align*}
		\operatorname{Ric}_{ij} = \widetilde{\operatorname{Ric}}_{ij} &- \alpha\, \Upsilon_{ij} - \frac{1}{2}\sinh\rho\cosh\rho\big[ \operatorname{tr}_\Upsilon(\dot{\Upsilon})\Upsilon_{ij} + (n-1)\dot{\Upsilon}_{ij} \big] \\
		&- \frac{1}{2}\sinh^2\rho\big[ \ddot{\Upsilon}_{ij} - (\dot{\Upsilon}^2)_{ij} + \tfrac{1}{2}\operatorname{tr}_\Upsilon(\dot{\Upsilon})\dot{\Upsilon}_{ij} \big],
	\end{align*}
	where $\alpha = (n-2)\cosh^2\rho + \sinh^2\rho$, $\widetilde{\operatorname{Ric}}_{ij}$ is the Ricci tensor of $(S_\rho,\Upsilon)$. 
	The contraction $\operatorname{tr}(A\cdot\operatorname{Ric}|_{\tan}) = \operatorname{Ric}_{ij}A^{ij}$ decomposes naturally via \eqref{eq:A_raised} as
	\begin{equation}\label{tr-1}
	\operatorname{Ric}_{ij}A^{ij} = \frac{\coth\rho}{\sinh^2\rho}\, \mathcal{X} + \frac{1}{2\sinh^2\rho}\, \mathcal{Y},
	\end{equation}
	with $\mathcal{X}:=\operatorname{Ric}_{ij}\Upsilon^{ij}$ and $\mathcal{Y} := \operatorname{Ric}_{ij}\dot{\Upsilon}^{ij}$. 
    
    Since 
    \begin{align}
        (\dot{\Upsilon}^2)_{ij}\dot{\Upsilon}^{ij}=&(\dot{\Upsilon}^2)_{ij}\dot{\Upsilon}_{kl}\Upsilon^{ik}\Upsilon^{jl}=\Upsilon^{ik}\Upsilon^{jl}(\dot{\Upsilon}^2)_{ij}\dot{\Upsilon}_{kl}\\=&(\dot{\Upsilon}^2)^{kl}\dot{\Upsilon}_{kl}=\operatorname{tr}_{\Upsilon}(\dot\Upsilon^3)
    \end{align}
    by contracting, we can give the further calculations of $\mathcal{X}$ and $\mathcal{Y}$
	\begin{equation}
		\mathcal{X} = \tilde{R}- (n-1)\alpha - (n-1)\sinh\rho\cosh\rho\,\operatorname{tr}_{\Upsilon}\dot{\Upsilon} - \frac{1}{2}\sinh^2\rho\Big(\operatorname{tr}_{\Upsilon}\ddot{\Upsilon} - \operatorname{tr}_{\Upsilon}(\dot{\Upsilon}^{2}) + \frac{1}{2}(\operatorname{tr}_\Upsilon\dot{\Upsilon})^2\Big),\label{tr-X} \\
\end{equation}	
\begin{equation}
\begin{aligned}
	\mathcal{Y} &= \langle \widetilde{Ric}, \dot{\Upsilon} \rangle_\Upsilon - \alpha \operatorname{tr}_{\Upsilon}(\dot{\Upsilon}) - \frac{1}{2}\sinh\rho\cosh\rho\Big((\operatorname{tr}_{\Upsilon}\dot{\Upsilon})^2 + (n-1)\operatorname{tr}_{\Upsilon}(\dot{\Upsilon}^2)\Big)
\\&- \frac{1}{2}\sinh^2\rho\Big(\langle\dot{\Upsilon},\ddot{\Upsilon}\rangle_{\Upsilon} - \operatorname{tr}_{\Upsilon}(\dot{\Upsilon}^3) + \frac{1}{2}\operatorname{tr}_\Upsilon\dot{\Upsilon}\operatorname{tr}_{\Upsilon}(\dot{\Upsilon}^2)\Big),\label{tr-Y}
	\end{aligned}
\end{equation}
	where $\tilde{R} = \operatorname{tr}_\Upsilon(\widetilde{\operatorname{Ric}})$ and $\langle \tilde{R}, \dot{\Upsilon} \rangle_\Upsilon=\tilde{R}_{ij}\dot{\Upsilon}^{ij}$. 
Putting \eqref{tr-1}, \eqref{tr-X} and \eqref{tr-Y} together, we obtain 
	\begin{equation}
		\begin{aligned}
			\operatorname{tr}(A\cdot\operatorname{Ric}|_{\tan} ) 
			&= \frac{1}{\sinh^2\rho}\Big\{ \coth\rho\Big( \tilde{R} - (n-1)\alpha - (n-1)\sinh\rho\cosh\rho\,\operatorname{tr}_{\Upsilon}\dot{\Upsilon} \\
			&\qquad\qquad - \tfrac{1}{2}\sinh^2\rho\big(\operatorname{tr}_{\Upsilon}\ddot{\Upsilon} - \operatorname{tr}_{\Upsilon}(\dot{\Upsilon}^2) + \tfrac{1}{2}(\operatorname{tr}_{\Upsilon}\dot{\Upsilon})^2\big) \Big) \\
			&\quad + \tfrac{1}{2}\Big( \langle \tilde{R}, \dot{\Upsilon} \rangle_\Upsilon - \alpha\operatorname{tr}_{\Upsilon}\dot{\Upsilon} - \tfrac{1}{2}\sinh\rho\cosh\rho\big((\operatorname{tr}_{\Upsilon}\dot{\Upsilon})^2 + (n-1)\operatorname{tr}_{\Upsilon}(\dot\Upsilon^2)\big) \\
			&\qquad\qquad - \tfrac{1}{2}\sinh^2\rho\big( \langle\dot{\Upsilon},\ddot{\Upsilon}\rangle_{\Upsilon} - \operatorname{tr}_{\Upsilon}(\dot{\Upsilon}^3) + \frac{1}{2}\operatorname{tr}_\Upsilon\dot{\Upsilon}\operatorname{tr}_{\Upsilon}(\dot{\Upsilon}^2)\big) \Big) \Big\}\\
&:=\operatorname{tr}(A_0\cdot\operatorname{Ric_0}|_{\rm tan})+\operatorname{tr}(A_1\cdot\operatorname{Ric_1}|_{\rm tan}),
		\end{aligned}\label{eq:trace_final}
\end{equation}
where 
$$\operatorname{tr}(A_0\cdot \operatorname{Ric_0}|_{\operatorname{tan}}):=\frac{1}{\sinh^2\rho}\Big\{\coth\rho\Big(\tilde{R}_0-(n-1)\alpha\Big)\Big\}$$
and 
\begin{align*}\operatorname{tr}(A\cdot \operatorname{Ric_1}|_{\operatorname{tan}})&:=\frac{1}{\sinh^2\rho}\Big\{(\tilde{R}-\tilde{R}_{0})\\&-\coth\rho\Big((n-1)\sinh\rho\cosh\rho\operatorname{tr}_{\Upsilon}\dot{\Upsilon}-\frac{1}{2}\sinh^2\rho(\operatorname{tr}_{\Upsilon}\ddot{\Upsilon}-\operatorname{tr}_{\Upsilon}(\dot{\Upsilon}^2))\\&
 + \tfrac{1}{2}\Big( \langle \tilde{R}, \dot{\Upsilon} \rangle_\Upsilon - \alpha\operatorname{tr}_{\Upsilon}\dot{\Upsilon} - \tfrac{1}{2}\sinh\rho\cosh\rho\big((\operatorname{tr}_{\Upsilon}\dot{\Upsilon})^2\\&\quad + (n-1)\operatorname{tr}_{\Upsilon}(\dot\Upsilon^2)\big) - \tfrac{1}{2}\sinh^2\rho\big( \langle\dot{\Upsilon},\ddot{\Upsilon}\rangle_{\Upsilon} - \operatorname{tr}_{\Upsilon}(\dot{\Upsilon}^3) + \frac{1}{2}\operatorname{tr}_\Upsilon\dot{\Upsilon}\operatorname{tr}_{\Upsilon}(\dot{\Upsilon}^2)\big) \Big)
    \Big\},
    \end{align*}
    with  scalar curvature
\begin{equation}
    \tilde{R}_{0}=-(n-1)(n-2),\,\,
    \tilde{R}=g^{ij}\widetilde{\operatorname{Ric}}_{ij}=\frac{\Upsilon^{ij}}{\sinh^2\rho}\widetilde{\operatorname{Ric}}_{ij}.
\end{equation}

Next, we provide a decomposition for shape operator $S$. A straight computation yields
\begin{equation}
    S=(\coth\rho)I+\frac{1}{2}(\partial_{\rho}k)\cdot\Upsilon^{-1}:=S_0+S_1.
\end{equation}

Let $(S_1)^i_j = (\Upsilon^{-1}\partial_\rho{\Upsilon})^i_j$,  the shape operator can be written as
\begin{equation}
S^i_j = \frac{f'}{f}\delta^i_j + \frac{1}{2}(S_1)^i_j,
\end{equation}
from which we deduce that
\begin{equation}
\begin{aligned}
    \operatorname{tr}(S^{3})=&\operatorname{tr}(S_0^3)+3(\coth\rho)^{2}\operatorname{tr}(S_{1})+3\coth\rho\operatorname{tr}(S_{1}^{2})+\operatorname{tr}(S_{1}^{3})\\=&(n-1)\coth^3 \rho 
    + \frac{3}{2}\coth^2 \rho\cdot\text{tr}(S_1) \\
    &+ \frac{3}{4}\coth \rho\cdot\text{tr}(S_1^2) 
    +\frac{1}{8}\text{tr}(S_1^3).
    \end{aligned}
\end{equation}
Similarly, for the mean curvature $\mathcal H$, we have
\begin{align}\label{mean curvature}
   \mathcal H&:=\mathcal H_0+\mathcal H_1=\operatorname{tr}S_{0}+\operatorname{tr}S_1\\
   &=(n-1)\coth\rho+\frac{1}{2}\operatorname{tr}\big((\partial_{\rho}k)\cdot\Upsilon^{-1}\big).
\end{align}
Next, we decompose $\operatorname{Ric}(\nabla_{g}\rho,\nabla_{g}\rho)=\operatorname{Ric}_{0}(\nabla_{g}\rho,\nabla_{g}\rho)+\operatorname{Ric}_1(\nabla_{g}\rho,\nabla_{g}\rho)$. By \eqref{Ricci tensor radial}, a direct computation infers that
\begin{equation}
    \operatorname{Ric}(\nabla_{g}\rho,\nabla_{g}\rho)=-(n-1)-\coth\rho\operatorname{tr}_{\Upsilon}(\dot{\Upsilon})-\frac{1}{2}\operatorname{tr}_{\Upsilon}(\ddot{\Upsilon})+\frac{1}{4}\operatorname{tr}_{\Upsilon}(\dot{\Upsilon}^2).
\end{equation}
 Since $\operatorname{Ric}_{0}(\partial_\rho,\partial_\rho)=-(n-1)$, we know that 
 \begin{equation}
 \operatorname{Ric}_{1}(\nabla_{g}\rho,\nabla_{g}\rho)=-\coth\rho\operatorname{tr}_{\Upsilon}\dot{\Upsilon}-\frac{1}{2}\operatorname{tr}_{\Upsilon}\ddot{\Upsilon}+\frac{1}{4}\operatorname{tr}_{\Upsilon}(\dot{\Upsilon}^2)=O(\rho^{-m-1}).
 \end{equation}
Here, $\dot{\Upsilon}:=\partial_\rho\Upsilon=\partial_{\rho}k$. 

From \eqref{Ricci tensor}, \eqref{Ricci tensor2} and  \eqref{Ricci tensor radial}, the scalar curvature admits the decomposition
\begin{equation}
    R=R_0+R_1,\,\, R_{0}=-n(n-1)
\end{equation}
Using $\cosh^2\rho-\sinh^2\rho=1$, we calculate
\begin{align}
    R=&g^{00}\operatorname{Ric}_{00}+\frac{1}{\sinh^2\rho}\Upsilon^{ij}\operatorname{Ric}_{ij}\\=&-(n-1)-\coth\rho\operatorname{tr}_{\Upsilon}\dot{\Upsilon}-\frac{1}{2}\operatorname{tr}_{\Upsilon}\ddot{\Upsilon}+\frac{1}{4}\operatorname{tr}_{\Upsilon}(\dot{\Upsilon}^2)+\frac{1}{\sinh^2\rho}\Bigg[\tilde{R}-[(n-2)\cosh^2\rho+\sinh^2\rho](n-1)\\&-\frac{1}{2}\sinh\rho\cosh\rho[\operatorname{tr}_{\Upsilon}\dot{\Upsilon}+\operatorname{tr}_{\Upsilon}(\dot\Upsilon)](n-1)-\frac{1}{2}\sinh^2\rho\Big[\operatorname{tr}_{\Upsilon}\ddot{\Upsilon}-\operatorname{tr}_{\Upsilon}(\dot{\Upsilon}^3)+\frac{1}{2}(\operatorname{tr}_{\Upsilon}\dot{\Upsilon})^{2}\Big]\Bigg]\\=&-n(n-1)-\coth\rho\operatorname{tr}_{\Upsilon}\dot{\Upsilon}-\frac{1}{2}\operatorname{tr}_{\Upsilon}\ddot{\Upsilon}+\frac{1}{4}\operatorname{tr}_{\Upsilon}(\dot{\Upsilon}^2)+\frac{1}{\sinh^2\rho}\Bigg[\tilde{R}-(n-1)(n-2)\Big]\\&-\frac{1}{2}\sinh\rho\cosh\rho[\operatorname{tr}_{\Upsilon}\dot{\Upsilon}+\operatorname{tr}_{\Upsilon}(\dot\Upsilon)](n-1)-\frac{1}{2}\sinh^2\rho\Big[\operatorname{tr}_{\Upsilon}\ddot{\Upsilon}-\operatorname{tr}_{\Upsilon}(\dot{\Upsilon}^3)+\frac{1}{2}(\operatorname{tr}_{\Upsilon}\dot{\Upsilon})^{2}\Big]\Bigg]
\end{align}
which implies  
\begin{equation}
    R_{1}=O(\rho^{-1-m}).
\end{equation}

Furthermore, we decompose $A=A_0+A_1$, where $A_0$ is the second fundamental  form of geodesic sphere in standard hyperbolic space, which is independent of tangential variable. Using the decay condition \eqref{decay-1} on $\Lambda(\rho,\theta)$, we obtain the following estimates
\begin{align}
    &\operatorname{div}_{S_\rho}^{2}A_{0}=0,\\&\operatorname{div}_{S_\rho}^{2}A_{1}=\frac{1}{2}\operatorname{div}_{S_\rho}^{2}(\partial_{\rho}k\cdot\Upsilon^{-1})=O(\rho^{-m-1})
\end{align}
Putting all the above analysis together, we conclude that $\Delta_g^{2}(\rho^2)$ can be decomposed into two parts. The first part is equal to  \eqref{biharmonic of geodesic} which has the uniform bound, and the second part possesses the decay property. Putting these two parts together, we have the following estimate
\begin{align}
    |\Delta_g^{2}(\rho^{2})|\leq\mathfrak{C}_{n}+O(\rho^{-m}),\,\,\,\rho\to\infty.
\end{align}
Thus, we complete the proof of Proposition \ref{biharmonic-2}.
\begin{remark}
We now focus on the toy model $\Bbb H^n$. Using the special structure, we have computed the biharmonic of squared distant function.  When $\Upsilon_{ij}$ is the standard round sphere metric $\sigma_{ij}$ on $\mathbb{S}^{n-1}$, the metric $g$ describes the hyperbolic space $\mathbb{H}^n$ of constant sectional curvature $-1$. In this case $\dot{\Upsilon}_{ij} \equiv 0$, so $$ \operatorname{tr}_\Upsilon(\dot{\Upsilon})= \operatorname{tr}_\Upsilon(\dot{\Upsilon}^2) = \operatorname{tr}_\Upsilon(\dot{\Upsilon}^3) = \operatorname{tr}_\Upsilon(\ddot{\Upsilon}) = \langle \dot{\Upsilon}, \ddot{\Upsilon} \rangle_\Upsilon=\langle \tilde{R}, \dot{\Upsilon} \rangle_\Upsilon = 0.$$Then \eqref{eq:trace_final} reduces to
$$
\operatorname{tr}(A\cdot\operatorname{Ric}|_{\tan}) = \frac{\coth\rho}{\sinh^2\rho}\Big( (n-1)(n-2) - (n-1)\alpha \Big)=-(n-1)^2\coth\rho,
$$
where we use the facts that $\alpha = (n-2)\cosh^2\rho + \sinh^2\rho$ and  $\cosh^2\rho - 1 = \sinh^2\rho$. 
On the other hand, the direct computation for  Ricci tensor on $\Bbb H^n$ yields  $R_{ab} = -(n-1)g_{ab}$, while $A_{ij} = \coth\rho\, g_{ij}$. This implies $A^{ij} = \coth\rho\, g^{ij}$. Contracting the spatial components yields $R_{ij}A^{ij} = -(n-1)\coth\rho\, (g_{ij}g^{ij}) = -(n-1)^2\coth\rho$ which confirms the validity of \eqref{eq:trace_final}.
\end{remark}


\end{document}